\documentclass[12pt,reqno]{amsart}
\usepackage{latexsym,amsmath,amssymb, amsthm, mathscinet}
\usepackage{cases, verbatim}
\usepackage[dvipdfmx]{graphicx}
\usepackage{bmpsize}

\usepackage{amsfonts,amsmath,amsthm, amssymb}
\usepackage{cases}
\usepackage[usenames]{color}
\usepackage{enumerate}
\usepackage{bm}
\theoremstyle{plain}
\newtheorem{thm}{Theorem}[section]

\newtheorem{defn}{Definition}

\newtheorem{lem}[thm]{Lemma}

\newtheorem{prop}[thm]{Proposition}
\theoremstyle{remark}
\newtheorem{rem}{\bf{Remark}}
\numberwithin{equation}{section}

\newcommand{\N}{\mathbb{N}}
\newcommand{\NN}{\rm {N}}
\newcommand{\C}{\rm {C}}

\newcommand{\R}{\mathbb{R}}

\newcommand{\T}{\mathbb{T}}
\newcommand{\Z}{\mathbb{Z}}

\newcommand{\AC}{{\rm AC\,}}

\newcommand{\BUC}{{\rm BUC\,}}
\newcommand{\USC}{{\rm USC\,}}
\newcommand{\LSC}{{\rm LSC\,}}
\newcommand{\Li}{L^{\infty}}
\newcommand{\Lip}{{\rm Lip\,}}
\newcommand{\SP}{{\rm SP}}

\newcommand{\ep}{\varepsilon}

\newcommand{\ol}{\overline}

\newcommand{\dist}{{\rm dist}\,}

\begin{document}
\title[Quantitative homogenization of HJ equations]{Quantitative homogenization of 
\\convex Hamilton--Jacobi equations 
\\with Neumann type boundary conditions}

\author{Hiroyoshi Mitake}
\address[H. Mitake]{
	Graduate School of Mathematical Sciences,
	University of Tokyo
	3-8-1 Komaba, Meguro-ku, Tokyo, 153-8914, Japan}
\email{mitake@g.ecc.u-tokyo.ac.jp}
%\thanks{The first author is the corresponding author.}

\author{Panrui Ni}
\address[P. Ni]{
Department 1: Graduate School of Mathematical Sciences, University of Tokyo, 3-8-1 Komaba, Meguro-ku, Tokyo 153-8914, Japan; Department 2: Shanghai Center for Mathematical Sciences, Fudan University, Shanghai 200438, China}
\email{panruini@g.ecc.u-tokyo.ac.jp}

%\thanks{}

\makeatletter
\@namedef{subjclassname@2020}{\textup{2020} Mathematics Subject Classification}
\makeatother

\date{\today}

\begin{abstract}
We study the periodic homogenization for convex Hamilton--Jacobi equations on perforated domains under the Neumann type boundary conditions. We consider two types of conditions, the oblique derivative boundary condition and the prescribed contact angle boundary condition, which is important in the front propagation. We first establish a new representation formula for the solution by using the Skorokhod problem and modified Lagrangians. By using this formula essentially, we prove the sub and superadditivity properties of the extended metric functions, which will be applied to obtain the optimal convergence rate $O(\ep)$ for homogenization of Neumann type problems.
\end{abstract}

\date{\today}

%\thanks{
%The work of YG was partially supported by Japan Society for the Promotion of Science (JSPS) through grants KAKENHI \#26220702,  \#16H03948.
%The work of HM was partially supported by KAKENHI \#15K17574, \#26287024, \#16H03948.
%The work of HT was partially supported by NSF grant DMS-1664424.}
%
%\author[H. Mitake]{Hiroyoshi Mitake}
%\address[H. Mitake]{
%Institute of Engineering, Division of Electrical, Systems and Mathematical Engineering,
%Hiroshima University 1-4-1 Kagamiyama, Higashi-Hiroshima-shi 739-8527, Japan}
%\email{hiroyoshi-mitake@hiroshima-u.ac.jp}
%
%\keywords{Forced Mean Curvature equation, Crystal growth, Asymptotic growth speed}
%
%\subjclass[2010]{
%35B40, %Asymptotic behavior of solutions
%35K93, %Quasilinear parabolic equations with mean curvature operator
%35K20. %Initial-boundary value problems for second-order parabolic equations
%}

%\begin{abstract}
%\end{abstract}
%

\maketitle
%\tableofcontents

%%%%%%%%%%%%%%%%%%%%%%%%%%%%%%%%%%%%%%%%%%%%%%%%%%%%%%%%%%%%%%%%%%%%%%%%%%%%%%%%%%%%%%%%%%%%%%%%%%%%%%%%%%%%%%%%%%%%%%%%%%%%%%%%%%%%

\section{Introduction}

We are concerned with the periodic homogenization for convex Hamilton-Jacobi equations on perforated domains under the Neumann type boundary conditions of the form: 
\begin{numcases}
{}
u_t^\ep(x,t)+H\left(\frac{x}{\varepsilon},Du^\ep(x,t)\right)=0
& \textrm{in} \quad $\Omega_\varepsilon\times(0,\infty)$, \label{eq:CN1}
\\
B\left(\frac{x}{\varepsilon},Du^\ep(x,t)\right)=0 
& \textrm{on} \quad$\partial\Omega_\varepsilon\times(0,\infty)$, \label{eq:CN2}\\
u^\ep(x,0)=u_0(x) 
& \textrm{on} \quad$\overline{\Omega}_\ep$,  \label{eq:CN3}
\end{numcases}
%\begin{equation}\label{eq:CN}
%  \left\{
%   \begin{aligned}
%   &u_t^\ep(x,t)+H\left(\frac{x}{\varepsilon},Du^\ep(x,t)\right)=0
%   & \textrm{in}&\quad \Omega_\varepsilon\times(0,\infty),
%   \\
%   &B\left(\frac{x}{\varepsilon},Du^\ep(x,t)\right)=0 
%   & \textrm{on} &\quad\partial\Omega_\varepsilon\times(0,\infty), \\
%   &u^\ep(x,0)=u_0(x) 
%   & \textrm{on} &\quad\overline{\Omega}_\ep, 
%   \end{aligned}
%   \right.
%\end{equation}
where $n\in\N$ with $n\ge2$, $\Omega\subset\R^n$ is an open connected set with $C^1$ boundary which has a periodic structure, i.e., 
$\Omega+\Z^n=\Omega$, and $\Omega_\varepsilon:=\varepsilon\Omega$ for $\ep>0$. 
Here, 
$H:\R^n\times\R^n\to\R$, 
$B:\R^n\times\R^n\to\R$ and $u_0:\R^n\to\R$ are given 
continuous functions. 
Although \eqref{eq:CN1}-\eqref{eq:CN3} only require $H$, $B$ and $u_0$ to be defined on $\overline{\Omega}\times\R^n$, $\partial{\Omega}\times\R^n$ and $\ol\Omega_\ep$, respectively, 
it is more convenient for us to consider 
$H$, $B$ on $\R^n\times\R^n$ and $u_0$ on $\R^n$ for convenience.  
\textit{Throughout} the paper we \textit{always} assume 
\begin{itemize}
    \item[(A1)] $H\in C(\R^n\times \R^n)$, and $y\mapsto H(y,p)$ is $\Z^n$-periodic for all $p\in \R^n$, that is, 
\[
H(y+z,p)=H(y,p) \quad\text{for all} \  y\in\R^n, z\in\Z^n, \ \text{and}  \ p\in\R^n,
\] 
    \item[(A2)] $\lim_{\left|p\right| \to \infty}\inf_{y \in \R^n}H(y,p) = \infty$, 
    \item[(A3)] $p \mapsto H(y, p)$ is convex for all $y\in \R^n$, 
    \item[(A4)] $u_0 \in \mathrm{BUC}(\mathbb{R}^n) \cap \mathrm{Lip} (\mathbb{R}^n)$.  
\end{itemize}
%Here, we denote by BUC$(X)$ the set of bounded and uniformly continuous functions on a metric space $X$, and Lip$(X)$ the set of Lipschitz continuous functions on $X$.

In this paper we consider two types of boundary conditions: 
\begin{equation}\label{bdr:N}
B_{\NN}(y,p):=\gamma(y)\cdot p-g(y), 
\end{equation}
and 
\begin{equation}\label{bdr:C}
B_{\C}(y,p):=\nu(y)\cdot p-\cos\theta(y)|p|.  
\end{equation}
Here, $\gamma:\partial\Omega\to\R^n$ is a continuous vector field which satisfies
\[
\gamma(y)\cdot\nu(y)>0\quad 
\text{for all} \ y\in\partial\Omega, 
\]
where $\nu(y)$ denotes the outward unit normal vector at $y\in\partial\Omega$. 
The functions $g, \theta:\R^n\to\R$ are given, and we assume 
\begin{itemize}
    \item[(A5)]
    $g\in C(\partial\Omega)$,  
    $g(y+z)=g(y)$ for $y\in\partial\Omega$, $z\in\Z^n$. 
    \item[(A6)] 
    $\theta\in C(\partial\Omega)$,  
    $\theta(y+z)=\theta(y)$, 
    and  
    $\theta(y)\in[\frac{\pi}{2},\pi)$ for $y\in\partial\Omega$, $z\in\Z^n$. 
\end{itemize}    
If $B$ is given by \eqref{bdr:N}, then we call \eqref{eq:CN2} the \textit{oblique derivative boundary condition}, and 
if $B$ is given by \eqref{bdr:C}, then we call \eqref{eq:CN2} the \textit{prescribed contact angle boundary condition} or the \textit{capillary boundary condition}. Throughout the paper, we set $h(y):=\cos\theta(y)$. As explained above, we consider $h$ as a function defined on $\R^n$ instead of on $\partial\Omega$.

A typical example which we have in our mind is a level set equation of the front propagation problem in a domain of the form:
\begin{equation}\label{eq:front}
\left\{
\begin{aligned}
&V=-c\left(\frac{x}{\ep}\right) & \text{on} & \ \Gamma_t^\ep\cap\Omega_\ep \ \text{for} \ t>0,
\\ &\angle (\nu^\ep, \nu_{\partial\Omega_\ep})=\theta\left(\frac{x}{\ep}\right) & \text{on} & \ \partial\Gamma_t^\ep\cap\partial\Omega_\ep \ \text{for} \ t\ge0,
\end{aligned}
\right.
\end{equation}
where $\angle$ means the angle between two vectors, $\{\Gamma_t^\ep\}_{t\ge0}\subset\R^n$ is an evolving hypersurface to be determined, and $c:\R^n\to(0,\infty)$, $\theta:\R^n\to[\frac{\pi}{2},\pi)$ are given continuous periodic functions. Here, $\Gamma_t^\ep$ is defined as the zero level set of a function $u^\ep(x,t)$, that is,
\[\Gamma_t^\ep=\{ x \in \Omega_\ep:\ u^\ep(x,t) = 0\},\]
with the interior region given by $\{x \in \Omega_\ep:\ u^\ep(x,t)<0\}$ and the exterior region by $\{x \in \Omega_\ep:\ u^\ep(x,t)>0\}$. The outward unit normal vector $\nu^\ep$ on $\Gamma_t^\ep$ is then
\[
\nu^\ep = \frac{Du^\ep}{|Du^\ep|},
\]
and the normal velocity of the interface is
\[
V = \frac{u^\ep_t}{|Du^\ep|},
\]
which is in the direction of $-\nu^\ep$, see \cite[Section 1.2]{G}. The unit normal on $\partial\Omega_\ep$ is denoted by $\nu_{\partial\Omega_\ep}$, and $\theta(\frac{\cdot}{\ep})$ describes the contact angle between $\Gamma_t^\ep$ and $\partial\Omega_\ep$ for each $t\ge 0$.
 See Section \ref{sec:example} below for further discussions on the front propagation problem. In that example, $V=-1$, which means that the front propagates towards the exterior region. We give a new representation formula for the solution of the prescribed contact angle boundary condition in Theorem \ref{thm:value}.

The qualitative homogenization for the Neumann type boundary value problem on perforated domains in a stationary case has been well established in \cite{HI}. One can also refer to \cite{A,AI} for the qualitative homogenization for stationary Hamilton-Jacobi equations with the Dirichlet type boundary value problem on perforated domains. By a similar argument as in \cite{HI},  
we can see that the limit problem is given by 
\begin{equation}\label{eq:limit}
  \left\{
   \begin{aligned}
   &u_t+\overline{H}(Du)=0 & \text{in}  & \ \R^n\times(0,\infty), \\
   &u(x,0)=u_0(x) 
   & \text{on} & \ \R^n.  
  \end{aligned}
   \right.
\end{equation}
Here, $\overline{H}$ is the effective Hamiltonian which depends on the boundary condition. 
More precisely, for any $p\in\R^n$, $\overline{H}(p)$ is the unique constant such that there exists a viscosity solution $v=v(\cdot;p)$ of the associated cell problem: 
\begin{equation}\label{eq:cell}
  \left\{
   \begin{aligned}
   &H(y,p+D_yv)=\overline{H}(p) & \text{in}  & \ \Omega, \\
   &B(y,p+D_yv)=0 & \text{on}  & \ \partial\Omega.  
   \end{aligned}
   \right.
\end{equation}
In this paper, our main purpose is to develop quantitative homogenization theory for Neumann type boundary value problems \eqref{eq:CN1}--\eqref{eq:CN3} on perforated domains 
and to establish a rate of convergence, 
which has never been studied in the literature.

Quantitative homogenization for first order Hamilton--Jacobi equations 
in the periodic setting has been started to be studied in \cite{CDI}, and 
a rate $O(\ep^{1/3})$ of convergence is established for general non-convex Hamilton--Jacobi equations. 
The recent development of quantitative homogenization for convex  Hamilton--Jacobi equations has been started in \cite{MTY}, and the optimal rate $O(\ep)$ is established in \cite{TY}. 
There has been much interest and development in this direction, and we refer to \cite{HJ,HJMT, HTZ,MN,MS,MNT,NT,QSTY,T}, and the references therein. We also refer to \cite{MMTXY,XY} for important developments of homogenization related to G-equations. The paper \cite{MTY} appears to be the first work that studies quantitative homogenization for convex Hamilton--Jacobi equations using the optimal control viewpoint. It systematically initiated the study of convergence rates and suggested that $O(\ep)$ should be optimal. Subsequently, by applying Alexander's theorem from first-passage percolation within the optimal control framework, \cite{C} established the rate $O(\ep \log(C+\ep^{-1}t))$. Independently and almost simultaneously, \cite{TY} recovered an older result of Burago concerning the convergence rate of stable norms in metric geometry \cite{Burago}, which yields the optimal rate $O(\ep)$, thereby settling the problem. In particular, our work is inspired by \cite{HJMT}, which studies the quantitative homogenization for convex Hamilton-Jacobi equations under the state constraint boundary condition on perforated domains. 
In \cite{HJMT}, the authors redevelop the framework introduced in \cite{TY} to apply it to the state constraint problem on perforated domains. 
More precisely, the extended metric functions associated with problems are introduced, and a kind of subadditivity and superadditivity properties, which are key ingredients in the paper, 
are proved to establish quantitative results of homogenization. 
This method is robust. 
However, it critically depends on the structure of the representation formula of viscosity solutions to the problem which is given by the associated value functions in the optimal control. 
Therefore, if we change the boundary condition, then it requires great care. 
As explained precisely below, we need to take care of the reflection effect of the trajectories into account when we consider the representation formula for the viscosity solution to the Neumann type problem \eqref{eq:CN1}--\eqref{eq:CN3}, which is expressed by the Skorokhod problem \eqref{eq:Sk} below. 
This creates new difficulties, and requires careful arguments to establish quantitative results.  
We point out that there has been still no PDE arguement to obtain a rate of convergence which is better than $O(\ep^{1/3})$ even in a convex setting.

%The key method in \cite{TY}, followed by \cite{HJMT} is the curve cutting method, more precisely, Burago's cutting lemma. In this paper, we consider the solutions of the Skorokhod problem instead of absolutely continuous curves. We still use the curve cutting method. One has to be careful when solutions of the Skorokhod problem hit the boundary and deal with the modified Lagrangians. 
%We point out that there has been still no PDE arguement to obtain a rate of convergence which is better than $O(\ep^{1/3})$ even in a convex setting, and the proof in \cite{TY, HJMT} critically depends on the representation formula based on the value function in the optimal control theory. 

It is worth mentioning that, in the review article \cite{L}, the qualitative and quantitative homogenization theory is listed as a major development in the study of partial differential equations. The equations considered in \cite{L} are elliptic PDEs. It is pointed out that the Neumann problem is more difficult than the Dirichlet problem. In \cite{KLS}, the authors solve the Neumann problem for $\gamma=\nu$. For general case where $\gamma$ is nowhere tangential to the boundary, \cite{L} points out that the problem is not trivial even for the Laplacian operator and is an interesting and challenging problem. 
See \cite{J2, S} for more recent development on this direction 
for instance. In contrast, our work concerns first-order Hamilton--Jacobi equations. While the nature of the equations is different, the analogy highlights the technical complexity of homogenization with Neumann-type boundary conditions.

In our paper we define the value functions 
$V_{\NN}^\ep, V_{\C}^\ep:\overline{\Omega}_\ep\times[0,\infty)\to\R$ for \eqref{eq:CN1}--\eqref{eq:CN3} by 
\begin{align}
&V_{\NN}^\ep(x,t):=
\inf\left\{\int_0^t
L_{\NN}\left(\frac{\eta(s)}{\varepsilon},-v(s),-l(s)\right)\,ds
+u_0(\eta(t))\mid (\eta,v,l)\in \SP_\ep(x)\right\}, 
\label{func:VN}\\
&V_{\C}^\ep(x,t):=
\inf\left\{\int_0^t
L_{\C}\left(\frac{\eta(s)}{\varepsilon},-v(s),-l(s)\right)\,ds
+u_0(\eta(t))\mid (\eta,v,l)\in \SP_\ep(x)\right\}, 
\label{func:VC}
\end{align}
where we denote by $\SP_\ep(x)$ for $x\in\overline{\Omega}_\ep$ 
the family of solutions to the \textit{Skorokhod problem}: triples $(\eta,v,l)$ such that, 
for given $x\in\overline{\Omega}_\ep$ and $v\in L^1_{\rm{loc}}([0,\infty),\R^n)$, a pair $(\eta,l)\in \AC_{\rm loc}([0,\infty),\R^n)\times L^1_{\rm loc}([0,\infty),\R)$ solves
\begin{equation}\label{eq:Sk}
\begin{cases}
\eta(0)=x, \\
\eta(s)\in\overline{\Omega}_\ep &\text{for all} \ s\in [0,\infty),\\ 
\dot{\eta}(s)+l(s)\gamma(\frac{\eta(s)}{\ep})=v(s),\quad & \text{for} \ a.e.\ s\in (0,\infty), \\ 
l(s)\ge0 & \text{for} \ a.e.\ s\in (0,\infty), \\
l(s)=0\quad \textrm{if}\quad \eta(s)\in \Omega_\ep,\quad &\text{for} \ a.e.\ s\in (0,\infty). 
\end{cases}
\end{equation}
In the following, we denote by $\SP(x)$ the set $\SP_1(x)$. Note that when we consider the prescribed contact angle boundary value problem, we always consider $\gamma=\nu$ in \eqref{eq:Sk}. 
The solvability of the Skorokhod problem has been well established in \cite[Theorem 4.1]{I11} (see also \cite{I13}), which is classically studied in \cite{LS}. 
Note that the uniqueness of $(\eta,l)$ satisfying \eqref{eq:Sk} is not known as far as the authors know. 
Here, the functions $L_{\NN}, L_{\C}:\R^n\times\R^n\times\R\to\R$ are modified Lagrangians which are defined according to the boundary conditions by 
\begin{align}
&L_{\NN}(y,q,l):=
\sup_{p\in\R^n}\big\{p\cdot q-H(y,p)\big\}-g(y)l, \label{def:L1}\\
&L_{\C}(y,q,l):=
\sup_{p\in\R^n}\big\{p\cdot q-H(y,p)-h(y)l|p|\big\}. \label{def:L2}
\end{align}
Here we provide some clarification on $L_{\C}$. Since $L_{\C}$ is defined as the supremum of a family of affine functions of $q$, it is convex in $q$. Moreover, by Proposition \ref{prop:lip}, one can modify $H(x,p)$ outside a sufficiently large ball to make it superlinear in $p$, which ensures that $L_{\C}$ is finite. When proving that the value function $V^\ep_{\C}$ is a viscosity solution to \eqref{eq:CN1}--\eqref{eq:CN3}, it is necessary to guarantee that the function 
\[
p \mapsto H(y,p) + h(y)\,l\,|p|
\] 
is convex in $p$. To ensure this, we assume $h \le 0$ and restrict to the region where $l \le 0$ in the definition of $L_{\C}$. 
This restriction is consistent with \eqref{eq:Sk}, since the integrand in \eqref{func:VC} is $L_{\C}\big(\frac{\eta(s)}{\ep}, -v(s), -l(s)\big)$ with $l(s) \ge 0$.

We first establish the following result. 
\begin{thm}\label{thm:value}
Assume that {\rm(A1)--(A4)} hold. 
\begin{enumerate}
\item[{\rm(i)}]
Assume that {\rm(A5)} holds. Let $V_{\NN}$ be the value function defined by \eqref{func:VN}. 
Then,  $V_{\NN}$ is the unique viscosity solution to \eqref{eq:CN1}--\eqref{eq:CN3} with $B=B_{\NN}$. 
\item[{\rm(ii)}]
Assume that {\rm(A6)} holds. Let $V_{\C}$ be the value function defined by \eqref{func:VC}. 
Then,  $V_{\C}$ is the unique viscosity solution to \eqref{eq:CN1}--\eqref{eq:CN3} with $B=B_{\C}$. 
\end{enumerate}
\end{thm}

Theorem \ref{thm:value} (i) is known, and we refer to \cite[Theorem 5.1]{I11} or \cite[Theorem 5.5]{I13} for the proof. 
Theorem \ref{thm:value} (ii), however, is genuinely new: to the best of our knowledge, no representation formula of this kind has appeared before. When $h(y)\leq 0$, that is, the contact angle $\theta(y)\in [\frac{\pi}{2},\pi)$, $p\mapsto B_{\C}(y,p)$ is convex, the value function corresponding to \eqref{eq:CN1}--\eqref{eq:CN3} is provided in \cite{BIM}. While their formula covers our assumptions, it is based on a different version of the Skorokhod problem and leads to a structurally distinct formulation. Their work provides important insights into problems with general nonlinear boundary conditions and develops a robust viscosity solution framework. In our setting, however, we introduce a new representation formula \eqref{func:VC}, which is more explicit and tailored to the specific analysis we pursue, particularly the construction of the metric function in Section \ref{sec:mt} (see the discussion preceding Lemma \ref{lem:lscc}). This formulation is made possible by the modified Lagrangian \eqref{def:L2}, which incorporates the boundary term $|Du^\ep|$ directly into the Legendre transform. It is a natural and, to our knowledge, previously unexplored idea. This idea plays a central role in our analysis and provides a foundation not only for homogenization, but also for further study of problems with similar Neumann-type conditions. In particular, in Section \ref{sec:example}, we use this representation formula to characterize the evolution of the front as it passes through a hole under a prescribed contact angle, and compute the corresponding optimal trajectories. This analysis reveals that the front develops a singular feature in its shape upon passing through the perforation, see Figure \ref{Fig:Example02}. Lemma \ref{lem:value-4} is a key ingredient to prove Theorem \ref{thm:value} (ii), which has a difficulty coming from the term $h(x)l |p|$ in the modified Lagrangian \eqref{def:L2} comparing to the proof for the oblique derivative boundary value problem.

By using Theorem \ref{thm:value} we establish our main result. 
\begin{thm}\label{thm:main}
Assume that {\rm(A1)--(A4)} hold. 
\begin{enumerate}
\item[{\rm(i)}]
Assume that {\rm(A5)} holds. 
Let $u^\ep$ be the unique viscosity solution to \eqref{eq:CN1}--\eqref{eq:CN3} with 
$B=B_{\NN}$, and $u$ be the limit function of $u^\ep$ as $\ep\to0$, solving \eqref{eq:limit}.
There exists $C>0$ such that 
\[
\|u^\varepsilon-u\|_{\Li(\overline{\Omega}_\ep\times[0,\infty))}\le
C\varepsilon.
\]
\item[{\rm(ii)}]
Assume that {\rm(A6)} holds. 
Let $u^\ep$ be the unique viscosity solution to \eqref{eq:CN1}--\eqref{eq:CN3} with 
$B=B_{\C}$, and $u$ be the limit function of $u^\ep$ as $\ep\to0$, solving \eqref{eq:limit}.
There exists $C>0$ such that 
\[
\|u^\varepsilon-u\|_{\Li(\overline{\Omega}_\ep\times[0,\infty))}\le
C\varepsilon.
\]
\end{enumerate}
\end{thm}

There was no known quantitative convergence rate for \eqref{eq:CN1}--\eqref{eq:CN3} in the literature before our work. To prove Theorem \ref{thm:main}, we essentially use the subadditivity and superadditivity properties of the metric function defined by \eqref{func:mn-star} and \eqref{func:mc-star}, respectively, see Lemmas \ref{lem:subad} and \ref{prop:super-ad} below. 
%, which follows the arguments in \cite{HJMT}. 
The arguments to prove these results are inspired by those in \cite{HJMT}, but with a crucial difference: we work with the Skorokhod problem \eqref{eq:Sk}. This introduces a new control variable $l$ in addition to $\dot\eta$. Estimating this reflection control turns out to be technically demanding and requires a new argument. We give a proof on the estimate of $l$ in Proposition \ref{prop:extremal-main}. Furthermore, obtaining Lemmas \ref{lem:subad} and \ref{prop:super-ad} demands a delicate analysis that carefully accounts for the interaction between $\dot\eta$ and $l$. The details of these estimates are presented in Section \ref{subsec:sa}.

In this paper, we focus on Hamiltonians that depend only on the fast variable $\frac{x}{\ep}$. However, we believe that one can obtain a convergence rate for multi-scale setting where the Hamiltonian depends also on the slow variable $x$, by adapting the techniques developed in this work together with \cite{HJ}. We plan to investigate this direction in future work.

\bigskip
\noindent

\noindent \textbf{Notation.}
For $X\subset\R^n$, $\Lip(X)$ 
(resp., $\BUC(X)$, $\USC(X)$, and $\LSC(X)$) denotes the
space of Lipschitz continuous (resp., bounded uniformly continuous, upper semicontinuous and lower semicontinuous) functions on $X$ with
values in $\R$. 
For $A\subset \R^n$, $B\subset \R^m$ with $n,m\in\mathbb N$, we denote by $\mathrm{AC}(A,B)$ the family of absolutely continuous functions on $A$ with values in $B$.  
For a possibly discontinuous function $\eta:\R\to X$, $t\in \R$, $\eta(t^+)$ and $\eta(t^-)$ stand for the right-hand and left-hand limits of $\eta$ at $t$, respectively. %For $x,y\in \R^n$, $[x,y]$ denotes the line segment connecting $x$ and $y$.
For $U, V \subset \R^n$,  we set $\dist(U,V)=\inf_{x\in U, y\in V} |x-y|$. 
The set $B_r(x)\subset \R^n$ stands for the open ball centered at $x$ with the radius $r$. %Let $\{e_1,e_2,\ldots,e_n\}$ be the canonical basis of $\R^n$.
We write $Y=[-\frac{1}{2},\frac12]^n$ %WJ: to be consistent with $Y_m$
as the unit cube in $\R^n$. 
%For $m\in \Z^n$, set $Y_m=m+[-\frac{1}{2},\frac{1}{2}]^n$, which is the cube of unit size centered at $m$.
%For $\emptyset \neq U, V \subset \R^n$,  $\dist(U,V)=\inf_{x\in U, y\in V} |x-y|$. For $x,y\in \R^n$, $[x,y]$ denotes the line segment connecting $x$ and $y$. Let $\mathrm{AC}(J,U)$ be the set of absolutely continuous curves $\xi:J \to U$.
%To avoid confusion, we recall that $\overline{H}=\ol H_\Om$ is the effective Hamiltonian corresponding to $H$ of the state-constraint problem on $\ol \Omega$, and $\ol H_0=\ol H_{\R^n}$ is the effective Hamiltonian corresponding to $H$ in the whole space.
If a function $h:\R^n\to \R$ is $\Z^n$-periodic, we can think of $h$ as a function from $\T^n$ to $\R$ and vice versa.

\bigskip
\noindent
\textbf{Organization. } In Section \ref{sec:pre}, we first recall the definition of viscosity solutions of \eqref{eq:CN1}-\eqref{eq:CN3} and the comparison principles, and give some basic properties of solutions of the Skorokhod problem \eqref{eq:Sk}. Then, the proof of Theorem \ref{thm:value} (ii) is given in Section \ref{sec:pre}. In Section \ref{sec:bs}, we prove the existence and a Lipschitz estimate of minimizing curves of the value functions defined in \eqref{func:VN} and \eqref{func:VC}, respectively. In Section \ref{sec:mt}, we define a class of metric functions, which plays a central role in the proof of Theorem \ref{thm:main}. We establish fundamental properties of metric functions, including the subadditivity and the superadditivity. Finally, Section \ref{sec:pthm} is devoted to give a proof of Theorem \ref{thm:main}.

\section{Value function}\label{sec:pre}
In this section we give a proof of Theorem \ref{thm:value}. 
First, 
we recall the notion of viscosity solutions to \eqref{eq:CN1}--\eqref{eq:CN3}, and give some basic results. 

\begin{defn} 
{\rm(i)} 
Let $u^\ep\in\USC(\overline{\Omega}_\ep\times [0,\infty))$. 
The function $u^\ep$ is said to be a viscosity subsolution of \eqref{eq:CN1}--\eqref{eq:CN3}  
if $u^\ep(\cdot,0)\leq u_0$ on $\overline{\Omega}_\ep$, and, 
for any $\varphi\in C^1(\overline{\Omega}_\ep\times[0,\infty))$, 
if $(\hat{x},\hat{t})\in \overline{\Omega}_\ep\times(0,\infty)$ is a maximizer of $u^\ep-\varphi$, and if $\hat{x}\in\Omega_\ep$, 
then 
$$
\varphi_t(\hat{x},\hat{t})+H\left(\frac{\hat{x}}{\ep},D\varphi(\hat{x},\hat{t})\right)\leq 0;
$$
if $\hat{x}\in\partial\Omega_\ep$, then 
$$
\min\left\{
\varphi_t(\hat{x},\hat{t})+H\left(\frac{\hat{x}}{\ep},D\varphi(\hat{x},\hat{t})\right), 
B\left(\frac{\hat{x}}{\ep},D\varphi(\hat{x},\hat{t})\right)
\right\}\leq 0.
$$
{\rm(ii)}
Let $u^\ep\in\LSC(\overline{\Omega}_\ep\times [0,\infty))$. 
The function $u^\ep$ is said to be a viscosity supersolution of \eqref{eq:CN1}--\eqref{eq:CN3} 
if $u^\ep(\cdot,0)\geq u_0$ on $\overline{\Omega}_\ep$, and, 
for any $\varphi\in C^1(\overline{\Omega}_\ep\times[0,\infty))$, 
if $(\hat{x},\hat{t})\in \overline{\Omega}_\ep\times(0,\infty)$ is a minimizer of $u^\ep-\varphi$, and if $\hat{x}\in\Omega_\ep$, 
then 
$$
\varphi_t(\hat{x},\hat{t})+H\left(\frac{\hat{x}}{\ep},D\varphi(\hat{x},\hat{t})\right)\geq 0;
$$
if $\hat{x}\in\partial\Omega_\ep$, then 
$$
\max\left\{
\varphi_t(\hat{x},\hat{t})+H\left(\frac{\hat{x}}{\ep},D\varphi(\hat{x},\hat{t})\right), 
B\left(\frac{\hat{x}}{\ep},D\varphi(\hat{x},\hat{t})\right)
\right\}\geq 0.
$$
{\rm(iii)} 
A continuous function $u^\ep$ is said to be a viscosity solution of \eqref{eq:CN1}--\eqref{eq:CN3} if $u^\ep$ is a viscosity subsolution and supersolution of \eqref{eq:CN1}--\eqref{eq:CN3}. 
\end{defn}

Henceforth, since we are always concerned with viscosity solutions, the adjective ``viscosity" is omitted.

\begin{thm}[Comparison principle for {\rm\eqref{eq:CN1}--\eqref{eq:CN2}}]\label{thm:comp}
Assume that {\rm(A1)--(A4)} hold. 
\begin{enumerate}
\item[{\rm(i)}]
Assume that {\rm(A5)} holds. 
Let $u\in \USC(\overline{\Omega}_\varepsilon\times[0,\infty))$ and $v\in \LSC(\overline{\Omega}_\varepsilon\times[0,\infty))$ be, respectively, 
a bounded subsolution and a bounded supersolution of \eqref{eq:CN1}--\eqref{eq:CN2} with $B=B_{\NN}$. 
If $u(\cdot,0)\le v(\cdot,0)$ on $\overline{\Omega}_\varepsilon$, 
then $u\le v$ on $\overline{\Omega}_\varepsilon\times[0,\infty)$. 
\item[{\rm(ii)}]
Assume that {\rm(A6)} holds. 
Let $u\in \USC(\overline{\Omega}_\varepsilon\times[0,\infty))$ and $v\in \LSC(\overline{\Omega}_\varepsilon\times[0,\infty))$ be, respectively, 
a bounded subsolution and a bounded supersolution of \eqref{eq:CN1}--\eqref{eq:CN2} with $B=B_{\C}$. 
If $u(\cdot,0)\le v(\cdot,0)$ on $\overline{\Omega}_\varepsilon$, 
then $u\le v$ on $\overline{\Omega}_\varepsilon\times[0,\infty)$. 
\end{enumerate}
\end{thm}
We refer to \cite[Theorem 3.4]{I11} or \cite[Theorem 3.1]{I13} for the proof of Theorem \ref{thm:comp} (i), and 
refer to \cite[Theorem 1 (A)]{BL} for the proof of Theorem \ref{thm:comp} (ii) in the stationary case. 
We can easily adapt the argument in \cite{BL} to the evolutionary problem.
Since Theorem \ref{thm:value} (i) is already established in \cite[Theorem 5.1]{I11} (see also \cite[Theorem 5.5]{I13}), we focus to prove Theorem \ref{thm:value} (ii) henceforth. 

\bigskip
We recall some basic results of the Skorokhod problem here. 
\begin{prop}[{\rm\cite[Proposition 4.4]{I11}}]\label{prop:stability-Sk}
Let $\{(\eta_k,v_k,l_k)\}_{k\in\N}\subset \cup_{x\in\ol\Omega}\SP(x)$ with $\eta_k(0)$ bounded. Assume that $\{|v_k|\}_{k\in\N}$ is uniformly integrable on every intervals $[0,T]$ with $0<T<\infty$. There exists a subsequence $\{(\eta_{k_j},v_{k_j},l_{k_j})\}_{j\in\mathbb N}$ of $\{(\eta_k,v_k,l_k)\}_{k\in\N}$ and $(\eta,v,l)\in \cup_{x\in\ol\Omega} \SP(x)$ such that for every $T>0$, $\eta_{k_j}\to \eta$ uniformly on $[0,T]$, $\dot{\eta}_{k_j}\,ds\to \dot{\eta}\,ds$ and $v_{k_j}\,ds\to v\,ds$ weakly-$\ast$ in $C([0,T],\mathbb R^n)^*$, $l_{k_j}\,ds\to l\,ds$ weakly-$\ast$ in $C([0,T],\R)^*$.
\end{prop}
In the above proposition, we denote by $X^\ast$ the dual space of the Banach space $X$. 
Regarding notation in the above proposition, we remark that the weak-star convergence in $C([0,T ])^\ast$ or $C([0,T ],\R^n)^\ast$ is usually stated as the weak convergence of measures.

\begin{prop}[{\rm\cite[Theorem 4.1]{I11}}]\label{lem:exist-Sk}
Let $v\in L^1_{\rm loc}([0,\infty),\mathbb R^n)$ and $x\in\overline\Omega$. Then there exists a pair $(\eta,l)\in \AC_{\rm loc}([0,\infty),\mathbb R^n)\times L^1_{\rm loc}([0,\infty),\mathbb R)$ such that $(\eta,v,l)\in \SP(x)$. 
\end{prop}

%The following result is quite similar to \cite[Theorem 4.1]{BIM}.
\begin{lem}\label{lem:vl}
Let $x\in\overline{\Omega}$, and $(\eta,v,l)\in\SP(x)$. 
There exists a constant $C>0$ such that $|\dot{\eta}(s)|\le C|v(s)|$ and $l(s)\le C|v(s)|$ for \textit{a.e.} $s\in [0,\infty)$. 
If $\gamma=\nu$, then 
$l(s)\le |v(s)|$ for \textit{a.e.} $s\in [0,\infty)$.
\end{lem}
Lemma \ref{lem:vl} is also known in \cite[Proposition 4.3]{I11}, but we give a proof for later usage. We include it here primarily to treat the case $\gamma=\nu$, which will be used later in the analysis of the prescribed contact angle boundary condition.
\begin{proof}
Since $l(s)=0$ when $\eta(s)\in\Omega$, we only need to consider the case where $\eta(s)\in\partial\Omega$. We take a function $\psi\in C^1(\mathbb R^n)$ satisfying  
\begin{align}
&\Omega=\{x\in\mathbb R^n\mid \psi(x)<0\}, 
\quad 
\partial\Omega=\{x\in\mathbb R^n\mid \psi(x)=0\}, \nonumber\\
&D\psi(x)\cdot\nu(x)>0 \ \textrm{for all} \ x\in\partial\Omega. 
\label{func:psi} 
\end{align}
In fact, since $\Omega+\Z^n=\Omega$, we can first consider this construction on $\T^n$. Let $\pi:\R^n\to\T^n$ be the standard projection, and let $\widetilde\Omega=\pi(\Omega)$. Fix the flat metric $|\cdot|$ on $\T^n$ induced from $\R^n$. For $\tilde x\in \T^n$, we define the distance function
\[d_{\partial \widetilde\Omega}(\tilde x):=\inf\limits_{\tilde y\in \partial \widetilde\Omega}|\tilde x-\tilde y|,\]
and define the signed distance as
\begin{equation*}
b_{\widetilde\Omega}(\tilde x):=
\begin{cases}
-&d_{\partial \widetilde\Omega}(\tilde x),\quad x\in \widetilde\Omega,
\\ &d_{\partial \widetilde\Omega}(\tilde x),\quad x\notin \widetilde\Omega.
\end{cases}
\end{equation*}
It is clear that there is a small constant $\rho>0$ such that $b_{\widetilde\Omega}$ is of class $C^1$ on the set $\widetilde\Omega_\rho:=\{\tilde x\in\T^n:\ d_{\partial \widetilde\Omega}(\tilde x)<\rho\}$. We take $\chi\in C^\infty(\R)$ satisfying $z\chi(z)>0$ for $z\neq 0$ and
\begin{equation*}
\chi(z)=
\begin{cases}
z,\quad &|z|\leq \frac{\rho}{4},
\\ 1,\quad &z\geq \frac{\rho}{2},
\\ -1,\quad &z\leq -\frac{\rho}{2}.
\end{cases}
\end{equation*}
It is clear that $\tilde \psi(\tilde x):=\chi(b_{\widetilde\Omega}(\tilde x))\in C^1(\T^n)$ and
\[\widetilde\Omega=\{\tilde x\in\T^n\mid \tilde\psi(x)<0\}, \quad \partial\widetilde\Omega=\{\tilde x\in\T^n\mid \tilde\psi(x)=0\},\quad D\tilde\psi(\tilde x)=\nu_{\partial\widetilde\Omega}(\tilde x)  \textrm{ for } \tilde x\in \partial\widetilde\Omega,\]
where $\nu_{\partial\widetilde\Omega}$ is the normal vector of $\partial\widetilde\Omega$. Then we periodically extend $\tilde \psi$ to get $\psi$. Since $\psi(\eta(s))\le 0$ and $\psi(x)=0$ for $x\in\partial\Omega$, if $\eta(s)$ is differentiable at $s$ and $\eta(s)\in\partial\Omega$, we have
\[0=\frac{d}{ds}\psi(\eta(s))=D\psi(\eta(s))\cdot\dot{\eta}(s)=|D\psi(\eta(s))|\nu(\eta(s))\cdot(v(s)-l(s)\gamma(\eta(s))),\]
where we note that $D\psi=|D\psi|\nu$. The above equality implies that
\[\nu(\eta(s))\cdot v(s)=l(s)\nu(\eta(s))\cdot\gamma(\eta(s))\]
when $\eta(s)\in\partial\Omega$. Let $\rho>0$ be a constant such that $\gamma\cdot\nu\ge\rho$. Then
\begin{equation}\label{vell}
  |v(s)|\ge \nu(\eta(s))\cdot v(s)=l(s)\nu(\eta(s))\cdot\gamma(\eta(s))\ge \rho l(s).
\end{equation}
Moreover, we have
\[|\dot{\eta}(s)|=|v(s)-l(s)\gamma(\eta(s))|\le \Big(1+\frac{\|\gamma\|_{L^\infty(\R^n)}}{\rho}\Big)|v(s)|,\]
which completes the proof. When $\gamma=\nu$, it is clear to see that we can take $\rho=1$.
\end{proof}

For simplicity, let $\ep=1$ and we write $V$ and $L$ for $V_{\C}^\ep$ and $L_{\C}$ which are defined by \eqref{func:VC} and \eqref{def:L2}, respectively, 
 \textit{throughout} the rest of this section.  

\begin{lem}\label{lem:value-1}
Let $\psi\in C^1(\overline{\Omega}\times[0,\infty))$ be a classical subsolution of \eqref{eq:CN1}--\eqref{eq:CN3} with $B=B_{\C}$. 
%Assume that $\psi(x,0)\le u_0(x)$ for all $x\in\overline\Omega$. 
Then, $\psi(x,t)\le  V(x,t) $ for all $(x,t)\in\overline{\Omega}\times[0,\infty)$.
\end{lem}
\begin{proof}
Let $(x,t)\in\overline{\Omega}\times[0,\infty)$ and $(\eta,v,l)\in \SP(x)$. Noting that 
\begin{align*}
\psi(\eta(t),0)-\psi(\eta(0),t)&=\int_0^t \frac{d}{ds}\psi(\eta(s),t-s) \,ds
\\ &=\int_0^tD\psi(\eta(s),t-s)\cdot\dot{\eta}(s)-\psi_t(\eta(s),t-s)\, ds
\\ &=\int_0^tD\psi(\eta(s),t-s)\cdot\big(v(s)-l(s)\nu(\eta(s))\big)-\psi_t(\eta(s),t-s)\,ds, 
\end{align*}
and that $\psi$ is a classical subsolution of \eqref{eq:CN1}--\eqref{eq:CN3}, 
by using the definition of $L_{\C}$, we obtain 
\begin{align*}
&\psi(x,t)-u_0(\eta(t))
\\ &\le \int_0^t\Big\{-D\psi(\eta(s),t-s)\cdot v(s)+l(s)D\psi(\eta(s),t-s)\cdot\nu(\eta(s))+\psi_t(\eta(s),t-s)\Big\} \,ds
\\ &\le \int_0^t\Big\{H(x,D\psi(\eta(s),t-s))+L(\eta(s),-v(s),-l(s))-h(\eta(s))l(s)|D\psi(\eta(s),t-s)|
\\ &\quad \hspace{50pt} +l(s)D\psi(\eta(s),t-s)\cdot\nu(\eta(s))+\psi_t(\eta(s),t-s)\Big\}\,ds
\\ &\le \int_0^t\Big\{L(\eta(s),-v(s),-l(s))+l(s)\big(D\psi(\eta(s),t-s)\cdot \nu(\eta(s))
\\ &\hspace{190pt} -h(\eta(s))|D\psi(\eta(s),t-s)|\big)\Big\}\,ds
\\ &\le \int_0^t L(\eta(s),-v(s),-l(s))\,ds.
\end{align*}
In the last inequality, we use the fact that $\psi$ is a classical subsolution, that is, we have
\[D\psi(x,t)\cdot \nu(x)-h(x)|D\psi(x,t)|\le 0\quad \text{for}\ x\in\partial\Omega,\]
and also the fact that $l(s)=0$ when $\eta(s)\in\Omega$. We then conclude that $\psi(x,t)\le V(x,t)$.
\end{proof}

\begin{lem}\label{lem:value-2}
There is a constant $C>0$ such that 
\[
|V(x,t)-u_0(x)|\le C t \quad\text{for all} \ (x,t)\in\overline{\Omega}\times[0,\infty). 
\]
\end{lem}
\begin{proof}
%The proof is similar to \cite[Lemma 5.3]{I11}.

For each $r>0$, take $u^r_0\in C^1(\overline{\Omega})$ such that 
$\|u^r_0-u_0\|_{W^{1,\infty}(\Omega)}\le r$. 
We choose a function $\psi\in C^1(\mathbb R^n)$ satisfying \eqref{func:psi}. 
Since $|h(x)|<1$ for all $\partial\Omega$, we can take $\tau>0$ large enough so that
\[\nu(x)\cdot \big(Du^r_0(x)-\tau D\psi(x)\big)=\nu(x)\cdot Du^r_0(x)-\tau|D\psi(x)|\le h(x)|Du_0^r-\tau D\psi(x)|
\]
for all $x\in\partial\Omega$. 
It is easy to see that $\tau$ is independent of $r$. Next, approximating 
the function
$s\mapsto \max\{-r,\min\{r,s\}\}$ for $s\in\mathbb R$ 
by a smooth function, we build a function $\zeta_r\in C^1(\mathbb R)$ so that $|\zeta_r(s)|\le r$, $0<\zeta_r'(s)\le 1$ for all $s\in\R$, 
and $\zeta'_r(0)=1$.
Take a constant $C>0$ which is independent of $r$ such that 
\[
H\left(x,Du^r_0(x)-\zeta_r'(\tau\psi(x))\tau D\psi(x)\right)\leq C\quad \textrm{for all}\quad x\in \Omega.
\]
Set 
$\phi(x,t):=-2r+u^r_0(x)-\zeta_r\big(\tau\psi(x)\big)-Ct$. 
Then, 
\begin{align*}
&D\phi(x,t)=Du^r_0(x)-\zeta_r'(\tau\psi(x))\tau D\psi(x), \\
&\zeta_r'(\tau\psi(x))\tau D\psi(x)=\tau D\psi(x)\quad \textrm{for}\quad x\in\partial\Omega,
\end{align*}
since $\psi(x)=0$ for $x\in\partial\Omega$, 
which implies that $\phi$ is a classical subsolution to \eqref{eq:CN1}--\eqref{eq:CN2}.  
Noting that 
$\phi(x,0)=-2r+u^r_0(x)-\zeta_r\big(\tau\psi(x)\big)\le u_0(x)$, 
by Lemma \ref{lem:value-1} we obtain 
\[
V(x,t)\ge  -2r+u^r_0(x)-\zeta_r\big(\tau\psi(x)\big)-Ct.
\]
We then let $r\to 0$ to get 
\begin{equation}\label{ineq1:lem:value-2}
V(x,t)\ge  u_0(x)-C t \quad \text{for all} \ 
(x,t)\in\overline{\Omega}\times[0,\infty).
\end{equation} 

Finally, noting that $(x,0,0)\in \SP(x)$, by the definition of $V(x,t)$, 
we obtain 
\begin{equation}\label{ineq2:lem:value-2}
V(x,t)
\le L(x,0,0)t+u_0(x)
\le -\min_{x\in\overline{\Omega}}\min_{p\in\R^n}H(x,p)t+u_0(x). 
\end{equation}
Combing \eqref{ineq1:lem:value-2} with \eqref{ineq2:lem:value-2}, 
we obtain the conclusion. 
\end{proof}

\begin{lem}\label{lem:value-3}
Let $R>0$. There are constants $C>0$ and $0<\tau<1$, depending only on $R$, $H$ and $h$, such that for any $(x,p,v,l)\in \overline{\Omega}\times B_R(0)\times\mathbb R^n\times \mathbb (0,\infty)$, if
\begin{equation}\label{<1}
  H(x,p)+L(x,-v,-l)\le 1-v\cdot p+h(x)l|p|,
\end{equation}
then we have $|v|\le C+\tau l$.
\end{lem}
\begin{proof}
Since $|h(x)|<1$ for all $x\in\partial\Omega$, 
we can take $\overline R>R>0$ large enough so that
\begin{equation}\label{tau}
  \tau:=\|h\|_{L^\infty(\R^n)}
  \frac{\overline R+R}{\overline R-R}<1.
\end{equation}
We choose a constant $C_1=C_1(\overline R)>0$ so that $C_1\ge  \max\limits_{(x,p)\in\overline\Omega\times \overline{B}_{\ol R}(0)}|H(x,p)|$. 
By definition,
\begin{align*}
L(x,-v,-l)&\ge  \max_{p\in B_{\ol R}(0)}(-v\cdot p-H(x,p)+h(x)l|p|)
\\ &\ge  \max_{p\in B_{\ol R}(0)}(-v\cdot p-C_1-\|h\|_{L^\infty(\R^n)}\overline R l)=\overline R|v|-C_1-\|h\|_{L^\infty(\R^n)}\overline R l.
\end{align*}
Now, let $(x,p,v,l)\in \overline\Omega\times B_R(0)\times\mathbb R^n\times (0,\infty)$ be a point such that \eqref{<1} holds. Then,
\[-C_1+\overline R|v|-C_1-\|h\|_{L^\infty(\R^n)} \overline Rl\le H(x,p)+L(x,-v,-l)\le 1+R|v|+\|h\|_{L^\infty(\R^n)} Rl,\]
which implies that
\[
|v|\le \frac{1+2C_1}{\overline R-R}+\tau l.
\]
\end{proof}

The following lemma is crucial for the proof of Theorem \ref{thm:value}. Although the proof is inspired by \cite[Lemma 5.5]{I11}, it differs fundamentally due to the additional term $h(x)l(s)|D\phi(\eta(s),t-s)|$ introduced by the prescribed contact angle boundary condition. To handle this term, we employ an iterative argument. The convergence of this iteration relies critically on the constant $\tau<1$ obtained in Lemma \ref{lem:value-3}.
\begin{lem}\label{lem:value-4}
Let $t>0$, $x\in\overline{\Omega}$, $\phi\in C^1(\overline{\Omega}\times[0,\infty))$ with $\|D\phi\|_{L^\infty(\ol\Omega\times[0,\infty))}<\infty$, and $\delta\in (0,1)$. Then, there exists $(\eta,v,l)\in\SP(x)$ such that 
\begin{multline*}
H(\eta(s),D\phi(\eta(s),t-s))-h(x)l(s)|D\phi(\eta(s),t-s)|+L(\eta(s),-v(s),-l(s))
\\ \le \delta-v(s)\cdot D\phi(\eta(s),t-s)
\end{multline*}
for \textit{a.e.} $s\in(0,t)$
\end{lem}
\begin{proof}
Because we consider $\gamma=\nu$ in \eqref{eq:Sk}, we can take $\rho=1$ in \eqref{vell} to get $l(s)\le |v(s)|$ for all $(\eta,v,l)\in \SP(x)$. In the following, we construct $(\eta,v,l)$ by induction.

For the first step, we take $v_1\equiv0$, $l_1\equiv 0$, and $\eta_1\equiv x$. It is clear that $(\eta_1,v_1,l_1)\in\SP(x)$. Next, for any $\delta>0$, and $s\in (0,t)$, we take $v_2(s)$ pointwisely such that
\[
H(\eta_1(s),D\phi(\eta_1(s),t-s))+L(\eta_1(s),-v_2(s),0)\le \delta-v_2(s)\cdot D\phi(\eta_1(s),t-s).
\]
Since $\delta<1$, we set 
$R:=\|D\phi\|_{L^\infty(\ol\Omega\times[0,\infty))}<\infty$ in Lemma \ref{lem:value-3} to get $|v_2(s)|\le C$ and $l_2(s)\le |v_2(s)|\le C$ 
for all $s\in[0, \infty)$ and some $C>0$ by Lemma \ref{lem:vl}. That is, $v_2\in L^\infty([0,t],\R^n)$. By Proposition \ref{lem:exist-Sk}, there is $(\eta_2,v_2,l_2)\in \SP(x)$.

Now, assume we have $(\eta_k,v_k,l_k)\in \SP(x)$ for $k\in\N$.
 For each $s\in (0,t)$, we take $v_{k+1}(s)$ pointwisely such that
\begin{align*}
H(\eta_k(s),D\phi&(\eta_k(s),t-s))+L(\eta_k(s),-v_{k+1}(s),-l_k(s))
\\ &\le \delta-v_{k+1}(s)\cdot D\phi(\eta_k(s),t-s)+h(\eta_k(s))l_k(s)|D\phi(\eta_k(s),t-s)|.
\end{align*}
%}
By  Lemma \ref{lem:value-3} once again, 
we obtain $|v_{k+1}(s)|\le C+\tau l_k(s)$ for all $s\in[0,\infty)$. 
Note that $l_k\le |v_k|$. 
By induction hypothesis, we obtain 
\[
|v_{k+1}(s)|\le C(1+\tau^2+\dots+\tau^{k-1})\le \frac{C}{1-\tau},
\]
since $\tau<1$, that is, $v_{k+1}\in L^\infty([0,t],\R^n)$. We then take $(\eta_{k+1},v_{k+1},l_{k+1})\in \SP(x)$ by Proposition \ref{lem:exist-Sk}. By Proposition \ref{prop:stability-Sk}, there is a subsequence $\{(\eta_k,v_k,l_k)\}_{k\in\N}$ which is denoted again by the same symbol and $(\eta,v,l)\in\SP(x)$ such that 
\begin{align*}
&\eta_k\to \eta \quad \text{uniformly on} \quad[0,t], \\
&v_k\, ds\to v\, ds \quad \text{weakly-}*  \ \text{in} \ C([0,t],\mathbb R^n)^*, \\
&l_k\, ds\to l\, ds \quad \text{weakly-}*  \ \text{in} \ C([0,t],\R)^*
\end{align*} 
as $k\to\infty$. 
Note that $\{v_k\}_{k\in \N}$ is uniformly bounded. Thus, we see that the pair $(v_{k}, l_k)$ weakly converges to $(v, l)$ in $L^2([0,t],\mathbb R^n)\times L^2([0,t],\R)$.

By Mazur's lemma, we take $\lambda_k=(\lambda_{k,1},\dots,\lambda_{k,N_k})$ of nonnegative numbers such that
\[\sum_{j=1}^{N_k}\lambda_{k,j}=1\]
and
\[
(\overline v_k,\overline  l_k):=\sum_{j=1}^{N_k}\lambda_{k,j}(v_{k+j+1}, l_{k+j})\to(v, l)\ \textrm{strongly in}\ L^2([0,t],\mathbb R^n)\times L^2([0,t],\R).
\]
Then we take a subsequence of $(\eta_k,v_k, l_k)$ such that $(\overline v_k,\overline l_k)\to (v, l)$ for a.e. $s\in (0,t)$. Since $\eta_k\to \eta$ uniformly, for all $\sigma>1$, one can take $k$ large enough such that
\begin{equation}\label{etae}
\begin{aligned}
H(\eta(s),D\phi&(\eta(s),t-s))+L(\eta(s),-v_{k+1}(s),- l_k(s))
\\ &\le \sigma\delta-v_{k+1}(s)\cdot D\phi(\eta(s),t-s)+h(\eta(s)) l_k(s)|D\phi(\eta(s),t-s)|.
\end{aligned}
\end{equation}
Noting that $(v,l)\mapsto L(x,-v,- l)$ is convex for all $x$, by \eqref{etae} we obtain 
\begin{align*}
&H(\eta(s),D\phi(\eta(s),t-s))+L(\eta(s),-\overline v_{k}(s),-\overline l_k(s))
\\ &=H(\eta(s),D\phi(\eta(s),t-s))+L(\eta(s),-\sum_{j=1}^{N_k}\lambda_{k,j}v_{k+j+1}(s),-\sum_{j=1}^{N_k}\lambda_{k,j} l_{k+j}(s))
\\ &\le H(\eta(s),D\phi(\eta(s),t-s))+\sum_{j=1}^{N_k}\lambda_{k,j}L(\eta(s),-v_{k+j+1}(s),- l_{k+j}(s))
\\ &\le \sigma\delta-\sum_{j=1}^{N_k}\lambda_{k,j}v_{k+j+1}(s)\cdot D\phi(\eta(s),t-s)+h(\eta(s))\sum_{j=1}^{N_k}\lambda_{k,j} l_{k+j}(s)|D\phi(\eta(s),t-s)|
\\ &
=\sigma\delta-\overline v_k(s)\cdot D\phi(\eta(s),t-s)+h(\eta(s))\overline  l_k(s)|D\phi(\eta(s),t-s)|. 
\end{align*}
Letting $k\to\infty$ and noting that $\sigma>1$ is arbitrary, we get
\begin{align*}
&H(\eta(s),D\phi(\eta(s),t-s))+L(\eta(s),-v(s),- l(s))
\\ &\le \delta-v(s)\cdot D\phi(\eta(s),t-s)+h(\eta(s)) l(s)|D\phi(\eta(s),t-s)|
\end{align*}
for  $a.e.$ $s\in [0,t]$.
\end{proof}

\begin{prop}[Dynamic programming principle]\label{prop:DPP}
Let $(x,t)\in\overline{\Omega}\times[0,\infty)$, and let $\tau:\SP(x)\to[0,t]$ 
be a non-anticipating map, that is, 
whenever $\alpha(s)=\beta(s)$ for $a.e.$ $s\in(0,\tau(\alpha))$, 
$\tau(\alpha)=\tau(\beta)$. 
Then, we have 
\begin{align*}
V(x,t)
=\inf\bigg\{\int_0^{\tau(\alpha)}
L_{\C}(\eta(s),-v(s),-l(s))\,ds+V(\eta(\tau(\alpha))&, t-\tau(\alpha))
\\ & \mid 
\alpha=(\eta,v,l)\in\SP(x)\bigg\}. 
\end{align*}
\end{prop}
Since we can check the dynamic programming principle directly as usual, 
we omit to give a proof.

\bigskip
\begin{proof}[Proof of Theorem {\rm\ref{thm:value} (ii)}]
By Lemma \ref{lem:value-2}, $|V(x,t)|<\infty$, and 
$\lim\limits_{t\to 0}V(x,t)=u_0(x)$. 
Let $V^\ast$ be the upper semi-continuous envelop of, that is, 
$V^\ast(x,t):=\lim_{r\to0}\{V(y,s)\mid |x-y|+|t-s|<r\}$. 
We first prove that $V^\ast$ is a subsolution to \eqref{eq:CN1}--\eqref{eq:CN2}. 

We argue by contradiction, and we suppose that there exists a test function $\phi\in C^1(\overline{\Omega}\times[0,\infty))$ satisfying  
$(V^*-\phi)(\hat x,\hat t)=0$ so that $V^*-\phi$ attains a strict maximum at 
$(\hat x,\hat t)\in\overline{\Omega}\times(0,\infty)$, and  
\begin{align*}
&\phi_t(\hat x,\hat t)+H(\hat{x},D\phi(\hat{x},\hat{t}))>0, 
&& \text{if} \ \hat{x}\in\Omega,  \\
&
\phi_t(\hat x,\hat t)+H(\hat{x},D\phi(\hat{x},\hat{t}))>0, \ \text{and} \  
B_{\C}(\hat{x},D\phi(\hat{x},\hat{t}))>0
&& \text{if} \ \hat{x}\in\partial\Omega. 
\end{align*}
We only consider the case where $\hat{x}\in\partial\Omega$. 
The case where $\hat x\in\Omega$ is easier to prove. We may assume that $\|D\phi\|_{L^\infty(\ol\Omega\times[0,\infty))}<\infty$ by modification, since we only analyze near $(\hat x,\hat t)$.

We choose $\delta>0$ small enough so that $\hat t>2\delta$, and 
\begin{equation}\label{nsub}
\left\{
\begin{array}{l}
\phi_t(x,t)+H(x,D\phi(x,t))\ge  2\delta, \\
\nu(x)\cdot D\phi(x,t)-h(x)|D\phi(x,t)|\ge  2\delta
\end{array}
\right. 
\end{equation}
for $(x,t)\in (\overline\Omega\cap B_{2\delta}(\hat x))\times [\hat t-2\delta,\hat t+2\delta]$. 
 Set
\begin{align*}
&\partial Q:=\big(\partial B_{2\delta}(\hat x)\times[\hat t-2\delta,\hat t+2\delta]\cup B_{2\delta}(\hat x)\times\{\hat t-2\delta\}\big)\cap\big(\overline{\Omega}\times[0,\infty)\big),\\
&m:=-\max_{\partial Q}(V^*-\phi).
\end{align*}
Note that since $(\hat x,\hat t)$ is a strict maximum point of $V^\ast-\phi$, $m>0$ and
$V^\ast(x,t)\le \phi(x,t)-m$ for all $(x,t)\in \partial Q$. 
By the definition of the upper semi-continuous envelope, we can choose a point $(\overline x,\overline t)\in\big(\overline\Omega\cap B_{\delta}(\hat x)\big)\times[\hat t-\delta,\hat t+\delta]$ so that
\[
(V^\ast-\phi)(\overline x,\overline t)>-\min\{\delta^2,m\}. 
\]

By Lemma \ref{lem:value-4}, we can find $\alpha:=(\eta,v, l)\in \SP(\overline x)$ such that for a.e. $s\ge  0$,
\begin{equation}\label{ve}
\begin{aligned}
L(\eta(s),-v(s),- l(s))+H(\eta(s),D\phi(\eta(s),\overline t-s))-&h(x) l(s)|D\phi(\eta(s),\overline t-s)|
\\ &\le \delta-v(s)\cdot D\phi(\eta(s),\overline t-s).
\end{aligned}
\end{equation}
Note that
\[
\sigma:=\overline t-(\hat t-2\delta)\ge  \delta,\quad 
\text{and} \quad
\dist(\overline x,\partial B_{2\delta}(\hat x))\ge  \delta.
\]
We set
\[
S:=\{s\in[0,\sigma]\mid  \eta(s)\in\partial B_{2\delta}(\hat x)\},
\]

Case 1. $S=\emptyset$. 
%For the case where $\tau=\infty$, we get $S=\emptyset$ by the definition of $\tau$. 
It is easy to see that $\alpha\mapsto \sigma$ is non-anticipating. 
By the dynamic programming principle, Proposition \ref{prop:DPP}, we have
\begin{align*}
\phi(\overline x,\overline t)
&<V^*(\overline x,\overline t)+\delta^2
\\ &
\le \int_0^{\sigma} L_{\C}(\eta(s),-v(s),- l(s))\,ds
+V^*(\eta(\sigma),\overline t-\sigma)+\delta^2
\\ &\le \int_0^{\sigma} 
\big[L_{\C}(\eta(s),-v(s),- l(s))+\delta\big]\, ds+\phi(\eta(\sigma),\overline t-\sigma),
\end{align*}
where for the last inequality we used $\sigma\ge \delta$. Since $S=\emptyset$, \eqref{nsub} always holds. Therefore, using \eqref{nsub} and \eqref{ve} we have
\begin{align*}
0&<\int_0^{\sigma} 
\Big[L_{\C}(\eta(s),-v(s),- l(s))+\delta+\frac{d}{ds}\phi(\eta(s),\overline t-s)\Big]\,ds
\\&\le 
\int_0^{\sigma} 
\Big[L_{\C}(\eta(s),-v(s),- l(s))+\delta+D\phi(\eta(s),\overline t-s)\cdot\dot{\eta}(s)-\phi_t(\eta(s),\overline t-s)\Big]ds
\\&\le \int_0^{\sigma} 
\Big[L_{\C}(\eta(s),-v(s),- l(s))+\delta
\\ &\quad 
\quad+D\phi(\eta(s),\overline t-s)\cdot\big(v(s)- l(s)\nu(\eta(s))\big)-\phi_t(\eta(s),\overline t-s)\Big]\,ds
\\&\le \int_0^{\sigma} \Big[2\delta-H(\eta(s),D\phi(\eta(s),\overline t-s))+h(\eta(s)) l(s)|D\phi(\eta(s),\overline t-s)|
\\ &\quad \quad
- l(s)D\phi(\eta(s),\overline t-s)\cdot\nu(\eta(s)))-\phi_t(\eta(s),\overline t-s)\Big]\,ds
\\ &\le \int_0^{\sigma}  l(s)\big[h(\eta(s))|D\phi(\eta(s),s)|-D\phi(\eta(s),\overline t-s)\cdot\nu(\eta(s)))\big]\,ds\le 0,
\end{align*}
which is a contradiction.

Case 2. $S\not=\emptyset$. We define $\tau:=\inf S$, which is the first hitting time of $\eta(s)$ onto $\partial B_{2\delta}(\hat x)$. We use the dynamic programming principle again, and the fact that $\eta(\tau)\in \partial Q$, we get
\begin{align*}
\phi(\overline x,\overline t)
&<V^*(\overline x,\overline t)+m
\\ &
\le \int_0^{\tau} L_{\C}(\eta(s),-v(s),- l(s))\,ds
+V^*(\eta(\tau),\overline t-\tau)+m
\\ &\le \int_0^{\sigma} 
L_{\C}(\eta(s),-v(s),- l(s))\, ds+\phi(\eta(\tau),\overline t-\tau).
\end{align*}
Arguing as the first case, we get a contradiction.

\medskip
We next prove that $V_\ast$ is a supersolution to \eqref{eq:CN1}--\eqref{eq:CN2}, 
where we denote by $V_\ast$ the lower semi-continuous envelop of $V$. 
We argue by contradiction, and we suppose that there exists a test function $\phi\in C^1(\overline{\Omega}\times[0,\infty))$ satisfying  
$(V_*-\phi)(\hat x,\hat t)=0$ so that $V_*-\phi$ attains a strict minimum at 
$(\hat x,\hat t)\in\overline{\Omega}\times(0,\infty)$, and  
\begin{align*}
&\phi_t(\hat x,\hat t)+H(\hat{x},D\phi(\hat{x},\hat{t}))<0, 
&& \text{if} \ \hat{x}\in\Omega,  \\
&
\phi_t(\hat x,\hat t)+H(\hat{x},D\phi(\hat{x},\hat{t}))<0, \ \text{and} \  
B_{\C}(\hat{x},D\phi(\hat{x},\hat{t}))<0 
&& \text{if} \ \hat{x}\in\partial\Omega. 
\end{align*}
We only consider the case where $\hat{x}\in\partial\Omega$. 
Take $0<2\delta<\hat{t}$ so that setting 
 $Q:=B_{2\delta}(\hat x)\times[\hat t-2\delta,\hat t+2\delta]$, 
we have
\[
\phi_t(x,t)+H(x,D\phi(x,t))\le 0\quad \textrm{and}\quad \nu(x)\cdot D\phi(x,t)-h(x)|D\phi(x,t)|\le 0
\]
for $(x,t)\in Q\cap (\overline{\Omega}\times[0,\infty))$.  
Set $m:=\min_{(\overline{\Omega}\times[0,\infty))\cap\partial Q}(V_*-\phi)>0$.
We may choose a point $(\overline x,\overline t)\in \overline{\Omega}\times[0,\infty)$ 
so that 
$(V_*-\phi)(\overline x,\overline t)<m$, $|\overline x-\hat x|<\delta$ 
and $|\overline t-\hat t|<\delta$. 
We take $(\eta,v, l)\in \SP(\overline x)$ so that
\[
V(\overline x,\overline t)+m>\int_0^{\overline t}L(\eta(s),-v(s),- l(s))ds+u_0(\eta(\overline t))
.\]
We set
\[
\tau=\min\{s\ge  0: (\eta(s),\overline t-s)\in\partial Q\}.
\]
It is clear that $\tau>0$, $\eta(s)\in Q\cap(\overline\Omega\times[0,\infty))$  for $s\in[0,\tau)$ and if $|\eta(\tau)-\hat x|<2\delta$, then by definition, $\tau=\overline t-(\hat t-2\delta)>\delta$. We have
\begin{align*}
\phi(\overline x,\overline t)+m&>\int_0^\tau L(\eta(s),-v(s),- l(s))\, ds+V(\eta(\tau),\overline t-\tau)
\\&\ge  \int_0^\tau L(\eta(s),-v(s),- l(s))\, ds+\phi(\eta(\tau),\overline t-\tau)+m. 
\end{align*}
Thus, 
\begin{align*}
0&>\int_0^\tau \Big[L(\eta(s),-v(s),- l(s))+\frac{d}{ds}\phi(\eta(s),\overline t-s)\Big]\, ds
\\&\ge  \int_0^\tau \Big[-v(s)\cdot D\phi(\eta(s),\overline t-s)-H(\eta(s),D\phi(\eta(s),\overline t-s))+h(\eta(s)) l(s)|D\phi(\eta(s),s)|
\\ &\quad +\dot\eta(s)\cdot D\phi(\eta(s),\overline t-s)\cdot\nu(\eta(s)))-\phi_t(\eta(s),\overline t-s)\Big]\, ds
\\ &\ge\int_0^\tau-l(s)\Big(\nu(\eta(s))\cdot D\phi(\eta(s),\bar t-s)-h(\eta(s))|D\phi(\eta(s),\bar t-s)|\Big)\, ds>0,
\end{align*}
which yields a contradiction.

By the comparison principle, Theorem \ref{thm:comp}, 
we obtain $V^*\le V_*$, which implies that $V(x,t)$ is continuous on $\overline{\Omega}\times[0,\infty)$. 
\end{proof}

\subsection{An example for the level set equation}\label{sec:example}

In this subsection, we provide a simple example of the level set equation to understand the value function given by \eqref{func:VC} with $\ep=1$. 
Consider
\begin{equation}\label{lse}
\begin{cases}
u_t+|Du|=0\quad \text{in}\quad \Omega\times(0,\infty),
\\ Du\cdot \nu=\cos\theta|Du| \quad \text{on}\quad \partial\Omega\times(0,\infty),
\\ u(x,0)=u_0(x) \quad\text{on} \ \overline{\Omega}, 
\end{cases}
\end{equation}
where $\Omega=\mathbb R^n\backslash B(0,1)$, 
$\theta\in [\frac{\pi}{2},\pi)$ 
%$\theta\in [\frac{\pi}{2},\frac{3\pi}{2}]\backslash\{\pi\}$ 
is a given constant, and $u_0(x)=x_1+2$ for $x=(x_1,x_2)\in\mathbb R^2$.

Since the associated Lagrangian is given by
\begin{equation*}
L_{\C}(x,-v,-l)=\sup_{p\in\R^n}\Big\{p\cdot v-|p|+l|p|\cos\theta\Big\}=
\begin{cases}
0\ &\text{if}\ |v|\leq 1-l\cos\theta,
\\ \infty\ &\text{otherwise},
\end{cases}
\end{equation*}
the solution of \eqref{lse} is given by
\[u(x,t)=\inf\Big\{u_0(\eta(t))\mid (\eta,v,l)\in\SP(x)\ \text{with}\ |v(s)|\leq 1-l(s)\cos\theta \ \text{for}\ s\in[0,t]\Big\}.
\]
For a later purpose, we define the cost functional by 
\[
J[\eta](x,t):=u_0(\eta(t))
\]
for $(\eta,v,l)\in\SP(x)$. 
By using the representation formula, we study the behavior of the 0-level set \[\Gamma_t:=\{x\in\R^2\mid u(x,t)=0\}\] in details. First, we notice that since $u_0$ is strictly increasing on $x_1$ and independent of $x_2$, the optimal trajectories of $u(x,t)$ should be the paths whose $\eta(t)$ is the smallest under the constraint.

We first consider the case where $\theta=\frac{\pi}{2}$. It is clear that the initial level set is $\{x_1=-2\}$, which is on the left of $B(0,1)$. When $\eta$ is in the interior of $\Omega$, we have $l=0$. Then the optimal trajectory $\eta^*$ among all paths $(\eta,v,l)\in\SP(x)$ with $\eta(s)\in\Omega$ for almost every $s\in[0,t]$ is just the straight lines $\eta^*(s)=x+s(-1,0)$. When $\eta$ is along the boundary $\partial \Omega$, we may have $l>0$. This case may cause difference. However, we have $|v|=\sqrt{|\dot\eta|^2+l^2}\leq 1$. To minimizing $\eta(t)$, we should take $l=0$ on the boundary. Therefore, the solution of \eqref{lse} coincides with that for the corresponding state constraint problem, since we always take $l=0$. For $t\in(0,2]$, the optimal trajectory is the straight lines $\eta^*(s)=x+s(-1,0)$. 
Note that the level set consists of two rays and an arc, and does not satisfy the contact angle condition in the classical sense. For $t>2$, if
\[
(x_1,x_2)\in\bar\Omega\cap \{(x_1,x_2)\in\R^n\mid x_1>0,\ 0\leq x_2\leq 1\},
\]
then we can intuitively see the optimal trajectory (see Figure \ref{Fig:Example01}). The optimal trajectory first reaches the boundary along a line tangent to it, then moves along the boundary, and finally leaves the boundary at the point $(0,1)$ and then proceeds straight to the left. Let us consider the polar coordinates $(r,\phi)$ defined by
\[r=\sqrt{x_1^2+x_2^2},\quad \phi=\arctan \frac{x_2}{x_1}.\]

\begin{figure}[htb]
\centering
\includegraphics[width=15cm]{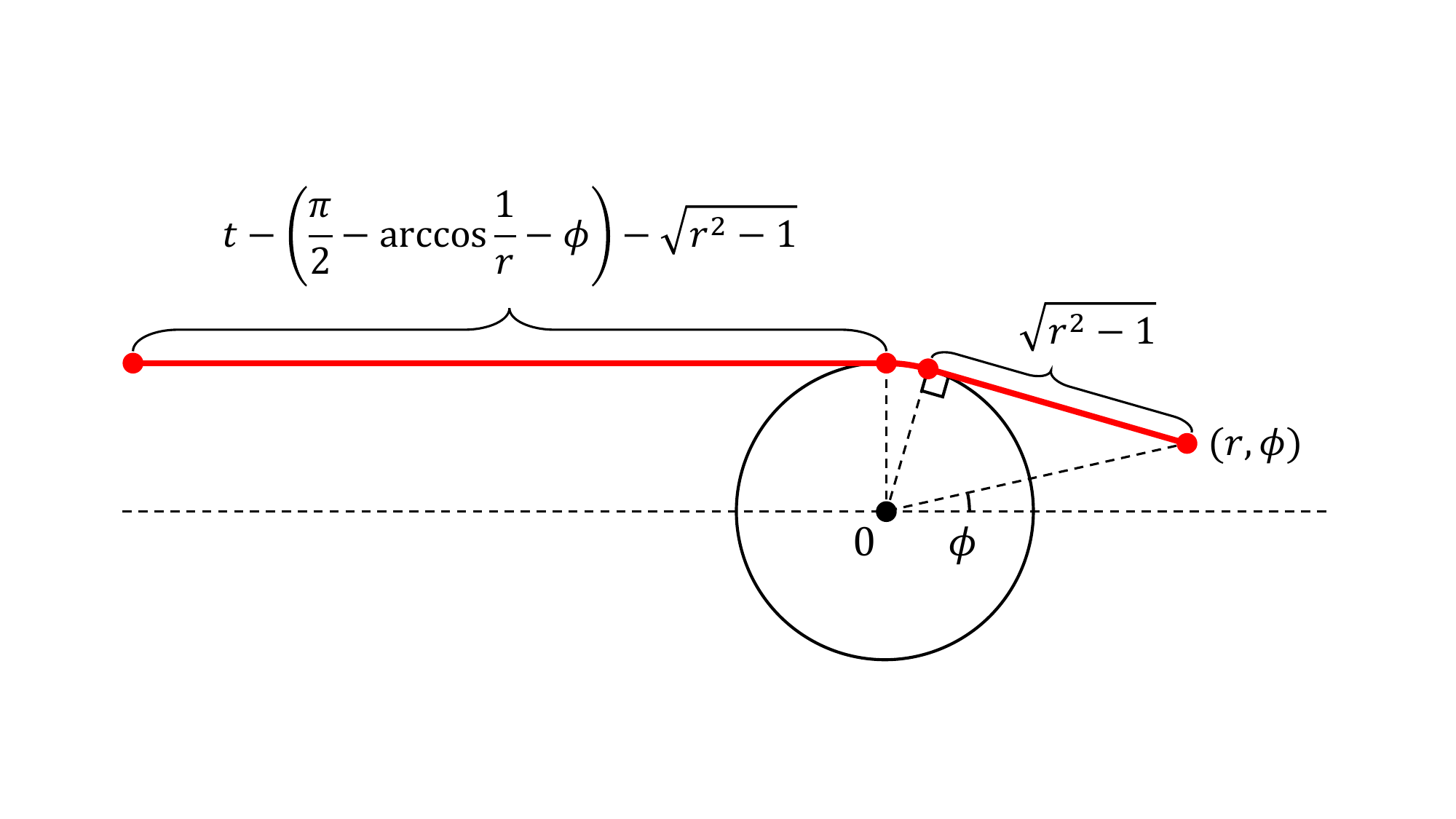}
\caption{Image picture of the optimal trajectory} 
\label{Fig:Example01}
\end{figure}

%we can easily see that the optimal trajectory of $u(x,t)$ satisfies
%\[\eta^*_1(t)=-\bigg[t-\Big(\frac{\pi}{2}-\arctan\frac{x_2}{\sqrt{1-x_2^2}}\Big)-\Big(x_1-\sqrt{1-x_2^2}\Big)\bigg].\]
%Thus, one can easily compute $u(x,t)$ and its 0-level set.

Therefore, we obtain 
\begin{align*}
&\eta^*_1(t)=-t+\Big(\frac{\pi}{2}-\arccos\frac{1}{r}-\phi\Big)+\sqrt{r^2-1}, \\
&
u((x_1,x_2),t)=\eta^*_1(t)+2. 
\end{align*}
By symmetry, $\Gamma_t$ leaves the boundary at the point $(1,0)$. To obtain the time $t_0$ when $\Gamma_t$ leaves the boundary, we set $r=1$ and $\phi=0$, then
\[u=-t_0+\frac{\pi}{2}+2=0,\]
which implies $t_0=\frac{\pi}{2}+2$. Thus, $\Gamma_t$ touches the boundary when $t\in [1,t_0]$, and satisfies the boundary condition in the classical sense when $t\in (2,t_0]$. To check that $\Gamma_t$ satisfies the boundary condition in the classical sense when $t\in (2,t_0]$, we proceed as follows. According to $u=0$, we get the equation describing $\Gamma_t$ in the coordinates $(r,\phi)$ as
\[\phi(r)=2-t+\frac{\pi}{2}-\arccos\frac{1}{r}+\sqrt{r^2-1}.\]
Calculating directly,
\[\frac{d\phi(r)}{dr}=\frac{\sqrt{r^2-1}}{r},\]
which is equal to zero when $r=1$. This means that $\Gamma_t$ is perpendicular to the boundary. See the picture of $0$-level set of $u$ in Figure \ref{Fig:Example02}. 
We also refer to \cite{CM} for a similar example. 

Interestingly, the behavior of $\Gamma_t$, as described mathematically above, closely resembles the Orowan mechanism in dislocation dynamics (see \cite{M}). For a discussion on the connection between Hamilton--Jacobi equations and dislocation dynamics, see \cite{CM,MNT} and the references therein.

\begin{figure}[htb]
\centering
\includegraphics[width=15cm]{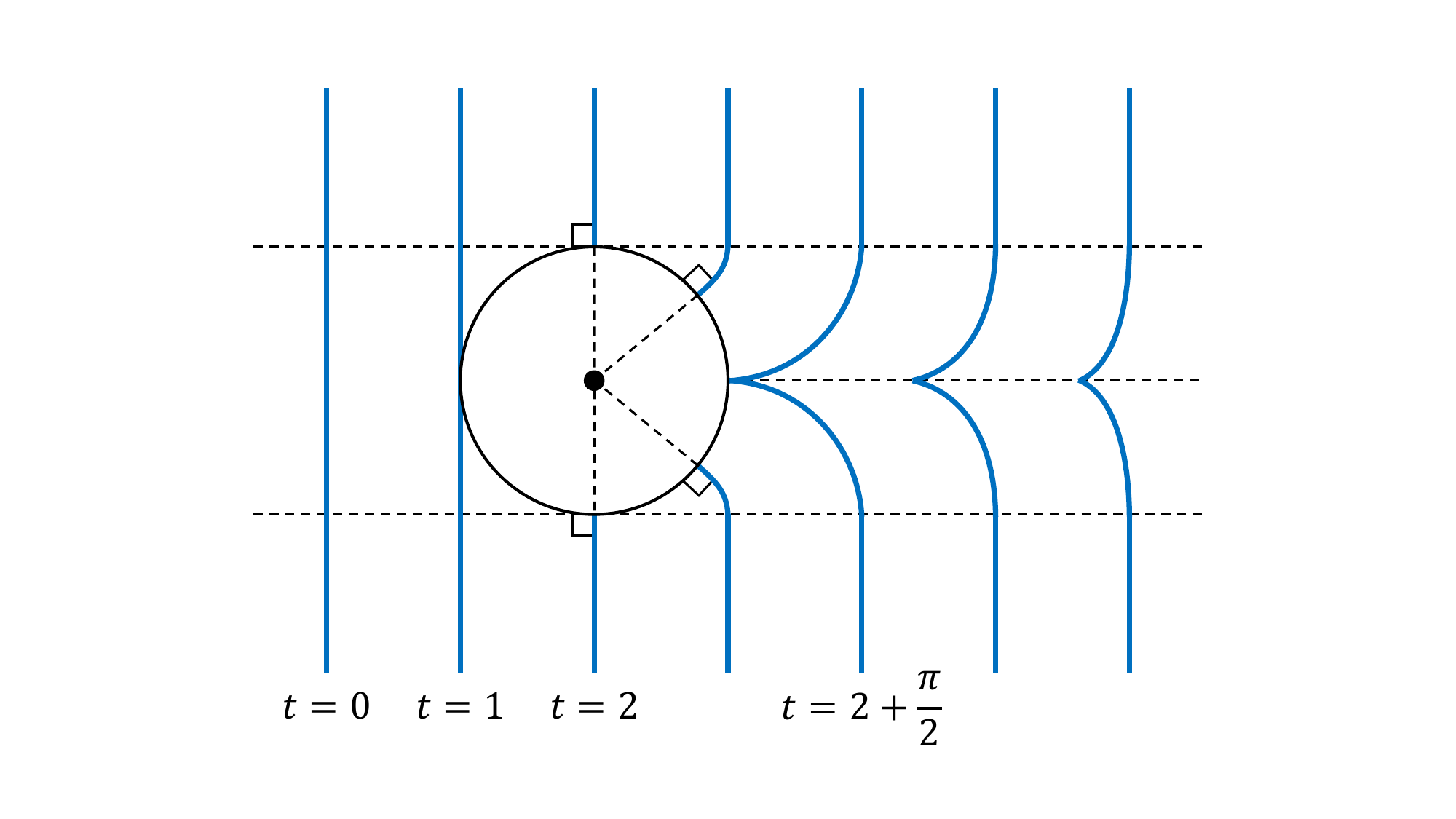}
\caption{Front propagation of $0$-level set of $u$} 
\label{Fig:Example02}
\end{figure}

\medskip
We next consider general $\theta\in(\frac{\pi}{2},\pi)$. Because of the constraint $|v(s)|\le 1-l(s)\cos\theta$ for $s\in[0,t]$, we can speed up on the boundary since $\cos\theta<0$. Therefore, we should take $l>0$ on the boundary, and then the solution of \eqref{lse} will be different from the one for the corresponding state constraint problem. 
More precisely, we take $(\eta(s),v(s),l(s))\in \SP(x)$ satisfying
\[|v(s)|=1-l(s)\cos\theta, 
\quad\text{and}\quad
\dot{\eta}(s)\perp\nu((\eta(s))
\]
to minimize $\eta_1(t)$. 
Since we always have $\dot{\eta}+l\nu=v$, we  have $|v|=\sqrt{|\dot\eta|^2+l^2}$. Therefore, 
\[
|\dot{\eta}|^2=|v|^2-l^2=(1-l\cos\theta)^2-l^2=-l^2\sin^2\theta-2l\cos\theta+1, 
\]
which implies the optimal reflection
\begin{equation}\label{l-star}
l^*(s)=-\frac{\cos\theta}{\sin^2\theta}.
\end{equation}

As a representative case, we study $\Gamma_t$ when $\theta=\frac{3\pi}{4}$ in more details here. We first find an optimal trajectory $\eta^\ast$ of $u((0,1),t)$ for $t\in(1,t_1)$, where $t_1>0$ will be fixed later. As calculated in \eqref{l-star}, we have $|\dot\eta^*(s)|=l^*(s)=\sqrt{2}$ along the boundary. Therefore, the optimal trajectory should first move along the boundary, and then goes straight to the left in parallel to the $x_1$-axis with $|\dot\eta^*(s)|=1$. This is because the trajectory can move faster along the boundary due to the reflection effect. Furthermore, we need to find an optimal place where $\eta^\ast$ should leave the boundary. To do so, for $\alpha\in[0,\frac{\pi}{2}]$, we set 
\[
\eta_\alpha^\ast(s):=
\left\{
\begin{array}{ll}
\begin{pmatrix}
\cos\sqrt{2}s  & -\sin\sqrt{2}s\\
\sin\sqrt{2}s & \cos\sqrt{2}s
\end{pmatrix} 
\begin{pmatrix}
0\\
1
\end{pmatrix}
& \text{for} \ s\in[0,\frac{\alpha}{\sqrt{2}}]\\
(-\sin\alpha,\cos\alpha)+s(-1,0) 
& \text{for} \ s\in[\frac{\alpha}{\sqrt{2}},t]. 
\end{array}
\right. 
\]
Then, 
\[
J[\eta_\alpha^\ast]((0,1),t)=-t+\frac{\alpha}{\sqrt{2}}-\sin\alpha+2. 
\]
Optimizing $J[\eta_\alpha^\ast]((0,1),t)$ on $\alpha$ to obtain $u((0,1),t)$, we find $\alpha=\frac{\pi}{4}$. 
From this consideration, we see that when a trajectory moves along the boundary, the optimal trajectory should leave at $(x_1,x_2)=(-\frac{1}{\sqrt{2}}, \frac{1}{\sqrt{2}})$. 
Therefore, it is clear now what is the optimal trajectory $\eta^\ast$ if the initial position starts at $x\in\Omega_L:=\{(x_1,x_2)\in\Omega\mid x_1\le0\}$. 
Thus, we can see that $\Gamma_t$ keeps straight lines until they arrive at $(-\frac{1}{\sqrt{2}}, \frac{1}{\sqrt{2}})$ and $(-\frac{1}{\sqrt{2}}, -\frac{1}{\sqrt{2}})$, that is, $t=t_1:=2-\frac{1}{\sqrt{2}}$ after $\Gamma_t$ first hits the boundary.

We next consider the case where the initial position starts at $x=(r,\phi)\in\Omega_R:=\{(x_1,x_2)\in\Omega\mid x_1\ge0\}$, where we use the polar coordinate $(r,\phi)$ once again. We only need the case where $\phi\in[0,\frac{\pi}{2}]$ by symmetry. For $\beta\in[0,\frac{\pi}{2}-\phi]$, we denote by $\eta_\beta^\ast$ the trajectory which moves from $(r,\phi)$ to $(1,\phi+\beta)$ with a straight line satisfying $|\dot{\eta}^\ast_\beta(s)|=1$, and from $(1,\phi+\beta)$ to $(1,\frac{3\pi}{4})$ along the boundary with $|\dot{\eta}^\ast_\beta(s)|=\sqrt{2}$, and goes straight in parallel to the $x_1$-axis with $|\dot{\eta}^\ast_\beta(s)|=1$ (see Figure \ref{Fig:Example03}). 
Noting that the length between $(r,\phi)$ and $(1+\phi+\beta)$ is given by 
$\sqrt{1+r^2-2r\cos\beta}$, we have 
\[
J[\eta^\ast_\beta]((r,\phi),t)=-t+\frac{3\pi}{4\sqrt{2}}-\frac{\alpha}{\sqrt{2}}-\frac{\phi}{\sqrt{2}}+\sqrt{1+r^2-2r\cos\alpha}-\frac{1}{\sqrt{2}}+2.
\]

\begin{figure}[htb]
\centering
\includegraphics[width=13cm]{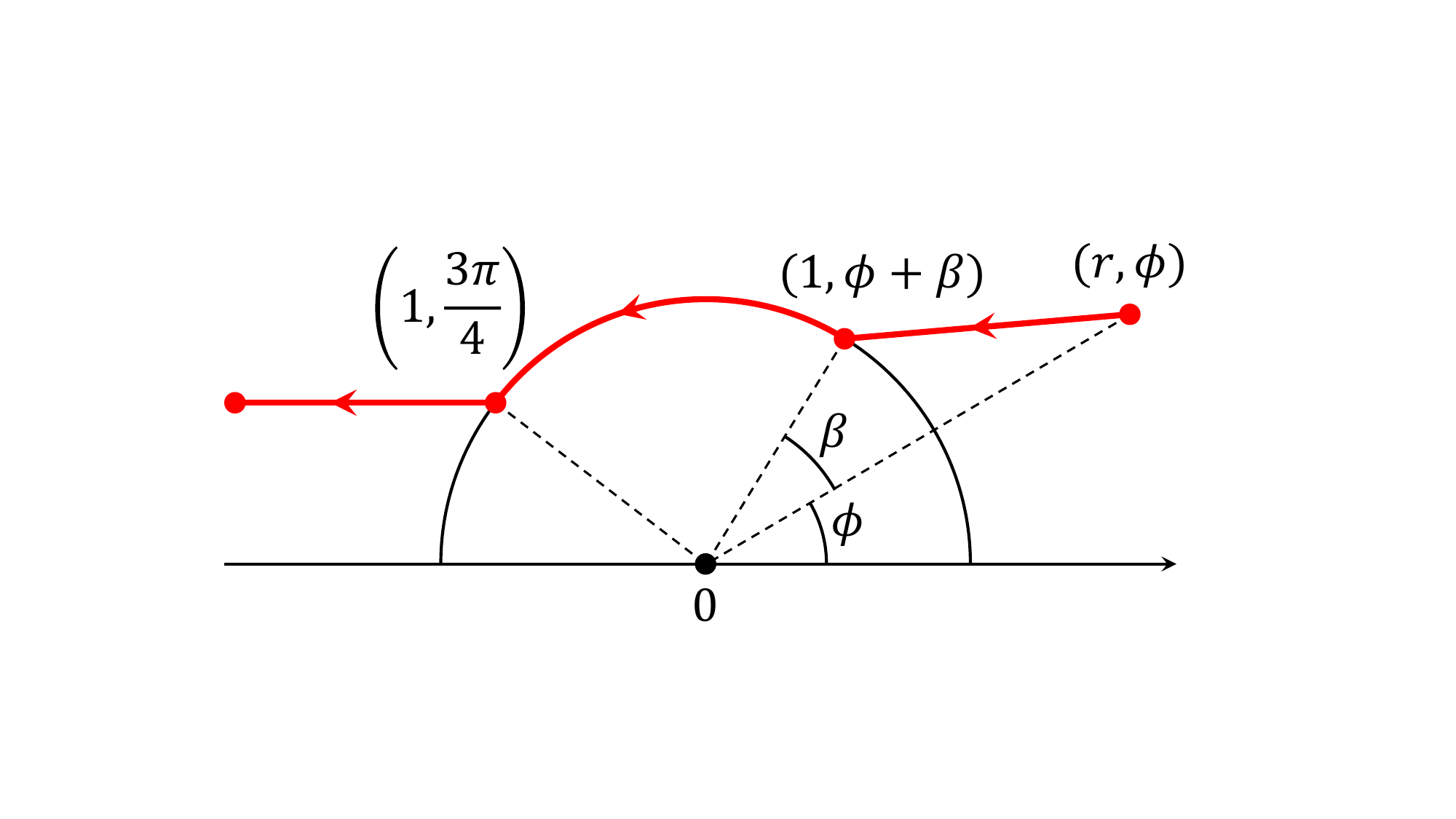}
\caption{Image picture of the optimal trajectory} 
\label{Fig:Example03}
\end{figure}

To minimize $J[\eta^\ast_\beta]((r,\phi),t)$ with respect to $\beta$, we differentiate the above quantity with respect to $\beta$, and finally conclude the minimum point $\beta$ satisfying
\[\cos\beta=\frac{1+\sqrt{2r^2-1}}{2r}.\]
Therefore, we conclude
\begin{equation}\label{u(r,phi)}
u((r,\phi),t)=-t+\frac{3\pi}{4\sqrt{2}}-\frac{1}{\sqrt{2}}\arccos\frac{1+\sqrt{2r^2-1}}{2r}-\frac{\phi}{\sqrt{2}}+\sqrt{r^2-\sqrt{2r^2-1}}-\frac{1}{\sqrt{2}}+2.
\end{equation}
Therefore, we can write down the equation for $\Gamma_t=\{u=0\}$ in $(r,\phi)$-coordinate
\[\phi(r)=2\sqrt{2}+\pi-2-\sqrt{2}t-\arccos\frac{1}{\sqrt{2}r}+\sqrt{2r^2-1},\]
where we used the fact that
\[\arccos\frac{1+\sqrt{2r^2-1}}{2r}=\arccos\frac{1}{\sqrt{2}r}-\frac{\pi}{4}\quad \text{and}\quad \sqrt{2(r^2-\sqrt{2r^2-1})}=\sqrt{2r^2-1}-1.\]
Calculating directly, we get
\[\frac{d\phi(r)}{dr}=\frac{\sqrt{2r^2-1}}{r},\]
which is equal to $1$ when $r=1$. This shows that $\Gamma_t$ meets the boundary with a contact angle $\frac{3\pi}{4}$, which coincides with the boundary condition. By symmetry, $\Gamma_t$ leaves the boundary at the point $(x_1,x_2)=(1,0)$. To obtain the time $t_2$ when $\Gamma_t$ leaves the boundary, we set $r=1$ and $\phi=0$, then $t_2=2+\frac{3\pi}{4\sqrt{2}}-\frac{1}{\sqrt{2}}$.

It remains to describe when $\Gamma_t$ changes from flat segments to a curved part. For such a point $(x_1,x_2)=(r,\phi)$, there are two optimal trajectories. One is the trajectory with the velocity $(-1,0)$, the other is the optimal trajectory that first reaches the boundary, move along the boundary, and then goes straight to the left. The first optimal trajectory gives
\[u((r,\phi),t)=-t+r\cos\phi+2.\]
Combining with \eqref{u(r,phi)}, we get the equation for which the changing points should satisfy
\begin{equation}\label{change}
\arccos\frac{1}{\sqrt{2}r}-\sqrt{2r^2-1}+\phi+\sqrt{2}r\cos\phi=\pi-2.
\end{equation}
%When $r=1$, $\phi=\frac{3\pi}{4}$, which coincides with the fact that $\Gamma_t$ starts to satisfies the boundary condition in the classical sense when $\Gamma_t$ reaches $(r,\phi)=(1,\frac{3\pi}{4})$. 
If $r\to \infty$, since the right hand side of \eqref{change} is finite, and $\sqrt{2r^2-1}\sim\sqrt{2}r$, it follows that $\cos\phi\to 1$, that is, $\phi\to 0$. Since $r\to \infty$ and $\phi\to0$, it follows that
\[\arccos\frac{1}{\sqrt{2}r}\to 0,\quad \sqrt{2r^2-1}\sim \sqrt{2}r,\quad \sqrt{2}r\cos\phi\sim \sqrt{2}r\Big(1-\frac{\phi^2}{2}\Big).\]
We conclude that $r\phi^2$ is bounded. Thus, for the changing point $(x_1,x_2)=(r,\phi)$, we have \[x_2=r\sin\phi\sim r\phi\sim \sqrt{r}\to \infty.\]
See Figure \ref{Fig:Example04} for the image of $\Gamma_t=\{u=0\}$.

\begin{figure}[htb]
\centering
\includegraphics[width=15cm]{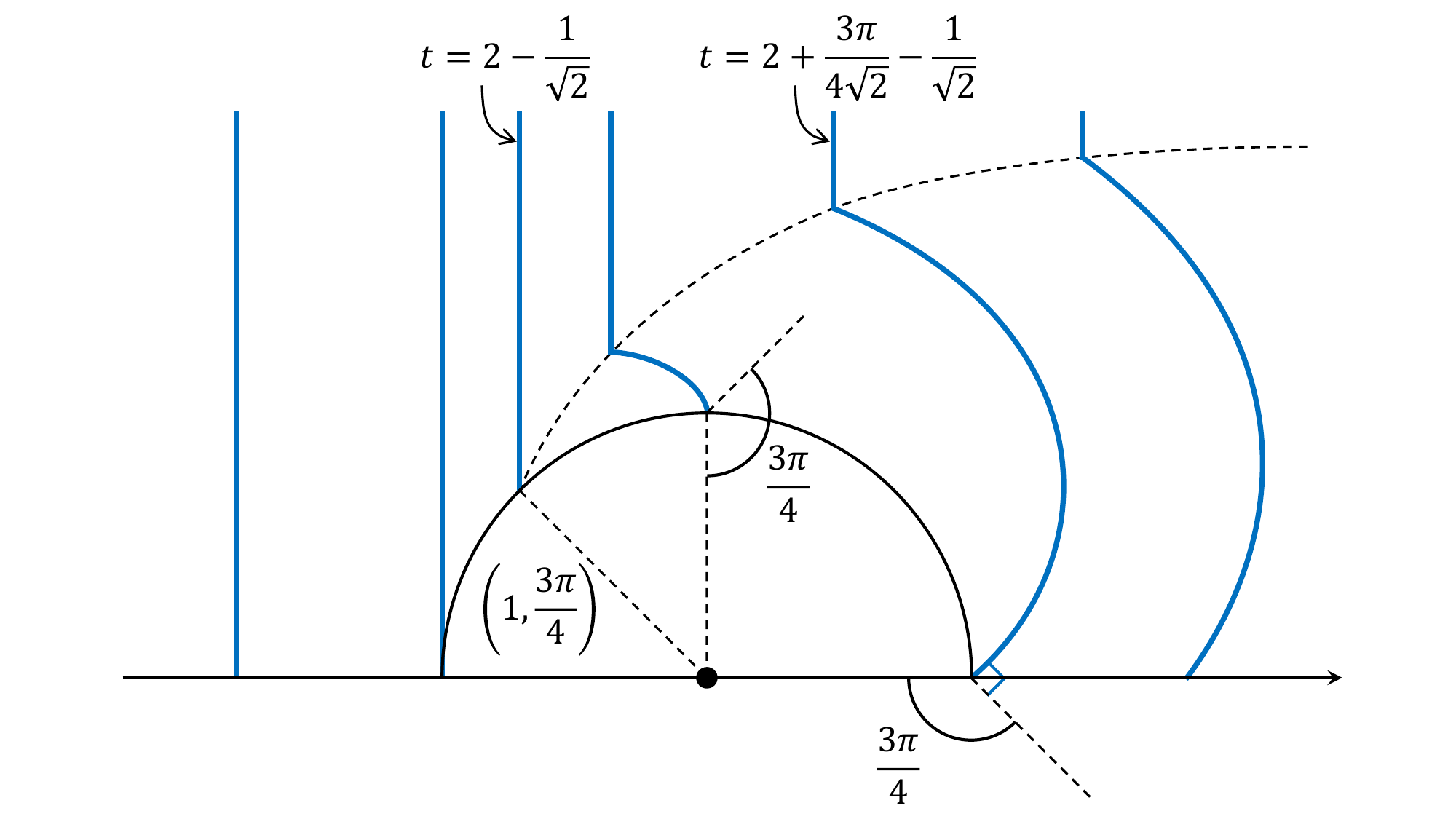}
\caption{Front propagation of $0$-level set of $u$} 
\label{Fig:Example04}
\end{figure}

\begin{rem}
We point out an observation for $\theta=\pi$. Let $l^\ast$ be given by \eqref{l-star}. 
Clearly, $l^\ast(s)$ is increasing with respect to $\theta\in[\frac{\pi}{2},\pi)$, 
and $l^\ast(s)\to\infty$ as $\theta\to\pi$. 
%This function is increasing in $[\frac{\pi}{2},\pi)$, decreasing in $(\pi,\frac{3\pi}{2}]$, has only two zero points at $\frac{\pi}{2}$ and $\frac{3\pi}{2}$, and tends to $+\infty$ as $\theta\to \pi$. 
This somehow shows a difficulty if we consider the case where $\theta=\pi$. At this time, $|\dot\eta^*(s)|=\csc\theta$. When $\theta\to \pi$, $|\dot\eta^*|\to\infty$. In this case, for the optimal trajectory $\eta^*$ of $u(x,t)$ with $x=(0,1)$, we have $\eta^*_1(t)\to -1-t$ as $\theta\to\pi$. One can see this since trajectories along the boundary $\partial \Omega$ can have a very large speed. Then they will spend a very short time along $\partial\Omega$. Thus, the optimal trajectory just quickly reaches the point $(-1,0)$ on $\partial\Omega$, and then goes directly left with $\dot\eta^*=(-1,0)$. Then we know $u(x,t)\to 1-t$ as $\theta\to \pi$, where we recall $x=(0,1)$. Here we raise a natural question: what happens when $\theta\to\pi$ for general case, which might be interesting to consider for future works.
\end{rem}
%for blue color

\section{Basic properties of value functions}\label{sec:bs}

In this section, we give several basic properties of the value functions 
$V_{\NN}^\ep$ and $V_{\C}^\ep$ including the Lipschitz estimate, the existence of extremal curves, and some estimates of extremal curves which are essential to obtain a rate of convergence of homogenization in Section \ref{sec:pthm}.

\begin{prop}[Lipschitz estimate]\label{prop:lip}
Assume that {\rm(A1)--(A4)} hold. 
\begin{enumerate}
\item[{\rm(i)}]
Assume that {\rm(A5)} holds. 
Let $V_{\NN}^\ep$ be the function defined by \eqref{func:VN}. 
There is $C>0$ which is independent of $\varepsilon$ such that
\[
\|(V_{\NN}^\ep)_t\|_{\Li(\overline{\Omega}_\ep\times[0,\infty))}
+\|DV_{\NN}^\ep\|_{\Li(\overline{\Omega}_\ep\times[0,\infty))}\leq C.
\]
\item[{\rm(ii)}]
Assume that {\rm(A6)} holds. 
Let $V_{\C}^\ep$ be the function defined by \eqref{func:VC}. 
There is $C>0$ which is independent of $\varepsilon$ such that
\[
\|(V_{\C}^\ep)_t\|_{\Li(\overline{\Omega}_\ep\times[0,\infty))}
+\|DV_{\C}^\ep\|_{\Li(\overline{\Omega}_\ep\times[0,\infty))}\leq C.
\]
\end{enumerate}
\end{prop}

%\begin{prop}[Lipschitz estimate]\label{prop:lip}
%Assume that {\rm(A1)--(A4)} hold. 
%\begin{enumerate}
%\item[{\rm(i)}]
%Assume that {\rm(A5)} holds. 
%Let $u$ be the unique viscosity solution to \eqref{eq:CN} with $B=B_N$. 
%There is $C>0$ which is independent of $\varepsilon$ such that
%\[
%\|u_t^\varepsilon\|_{\Li(\overline{\Omega}_\ep\times[0,\infty))}
%+\|Du^\varepsilon\|_{\Li(\overline{\Omega}_\ep\times[0,\infty))}\leq C.
%\]
%\item[{\rm(ii)}]
%Assume that {\rm(A6)} holds. 
%Let $u$ be the unique viscosity solution to \eqref{eq:CN} with $B=B_C$. 
%There is $C>0$ which is independent of $\varepsilon$ such that
%\[
%\|u_t^\varepsilon\|_{\Li(\overline{\Omega}_\ep\times[0,\infty))}
%+\|Du^\varepsilon\|_{\Li(\overline{\Omega}_\ep\times[0,\infty))}\leq C.
%\]
%\end{enumerate}
%\end{prop}

%This is a straightforward result of 
The Lipschitz estimate for \eqref{eq:CN1}--\eqref{eq:CN3} is rather standard, but we give a proof for the sake of completeness. 

\begin{proof}
Let $C$ be the constant which is given by Lemma \ref{lem:value-2}. 
Fix $s>0$, and set $u_1(x,t):=V^\ep_{\C}(x,t+s)$ and $u_2(x,t):=V^\ep_{\C}(x,t)-Cs$ are solutions of \eqref{eq:CN1}--\eqref{eq:CN2}. 
By Lemma \ref{lem:value-2}, we have
\[
u_1(x,0)=V^\ep_{\C}(x,s)\ge  u_0(x)-Cs=u_2(x,0).
\]
By the comparison principle, Theorem \ref{thm:comp}, 
we obtain $V^\ep_{\C}(x,t+s)\ge V^\ep_{\C}(x,t)-Cs$, 
which implies that $(V_{\C}^\ep)_t\ge-C$. 
Similarly, we obtain $(V_{\C}^\ep)_t\le C$. 
Therefore, 
by the coercivity of $H$, we get $|DV_{\C}^\ep|\leq C$. 
\end{proof}

Based on the Lipschitz estimate, Proposition \ref{prop:lip}, 
we can modify $H(y, p)$ for $|p| > C$ without changing the solutions to \eqref{eq:CN1}--\eqref{eq:CN3} and ensure that, 
for all $(y,p) \in \R^n\times \mathbb{R}^n$, 
\begin{equation}\label{eq:K_0H}
    \frac{|p|^2}{2}-K_0 \leq H(y, p) \leq \frac{|p|^2}{2}+K_0
\end{equation}
for some constant $K_0 >0$ that depends only on $H$ and $\left\|Du_0\right\|_{L^\infty(\mathbb{R}^n)}$. 
%Consequently, for all $(y,v)  \in \R^n \times \mathbb{R}^n$,
%\begin{equation}\label{eqn:K_0L}
%\frac{|v|^2}{2}-K_0 \leq L(y, v) \leq \frac{|v|^2}{2}+K_0,
%\end{equation}
%where $L: \R^n \times \mathbb{R}^n \to \mathbb{R
%}$ is the Legendre transform of $H$.

\begin{prop}\label{prop:extremal-main}
Assume that {\rm(A1)--(A4)} hold. 
Let $(x,t)\in\overline{\Omega}_\ep\times(0,\infty)$. 
\begin{enumerate}
\item[{\rm(i)}]
Assume that {\rm(A5)} holds. 
There exists $(\eta^\ep,v^\ep,l^\ep)\in\SP_\ep(x)$ such that 
\[
V_{\NN}^\ep(x,t)
=
\int_0^{t}L_{\NN}\left(\frac{\eta^\ep(s)}{\ep}, -v^\ep(s),-l^\ep(s)\right)\,ds
+u_0(\eta^\ep(t)). 
\]
Moreover, there exists $M_0>0$ which is independent of $\ep$ 
such that
\[
\|\dot\eta^\ep\|_{\Li([0,t])}+
\|l^\ep\|_{\Li([0,t])}\leq M_0.
\] 
\item[{\rm(ii)}]
Assume that {\rm(A6)} holds. There exists $(\eta^\ep,v^\ep,l^\ep)\in\SP_\ep(x)$ such that 
\[
V_{\C}^\ep(x,t)
=
\int_0^{t}L_{\C}\left(\frac{\eta^\ep(s)}{\ep}, -v^\ep(s),-l^\ep(s)\right)\,ds
+u_0(\eta^\ep(t)). 
\]
Moreover, there exists $M_0>0$ which is independent of $\ep$ 
such that
\[
\|\dot\eta^\ep\|_{\Li([0,t])}+
\|l^\ep\|_{\Li([0,t])}\leq M_0.
\]
\end{enumerate}
\end{prop}

Existence of extremal curves for $V_{\NN}^\ep$ is known. See \cite[Theorem 5.6]{I13}.  
We give a proof for the existence of extremal curves for $V_{\C}^\ep$ here. 
For simplicity, we set $\ep=1$. The existence of extremal curves of $V^\ep_{\C}$ is a direct consequence of Lemma \ref{lem:lscc} below. 

\begin{lem}\label{lem:basic-L}
Let $t,\delta>0$, $(\eta,v, l)\in L^\infty([0,t],\mathbb R^n)\times L^1([0,t],\mathbb R^n)\times L^1([0,t],[0,\infty))$. Let $i\in\mathbb N$ and set 
\begin{equation}\label{Li}
  L_{\C}^i(x,-v,- l):=\max_{p\in \overline B_i(0)}(-p\cdot v-H(x,p)+h(x) l|p|).
\end{equation}
Assume that $\eta(s)\in\overline\Omega$ for all $s\in[0,t]$. 
Then, there exists a function $p\in L^\infty([0,t],\mathbb R^n)$ such that \begin{equation}\label{H+Li}
  H(\eta(s),p(s))+L_{\C}^i(\eta(s),-v(s),- l(s))\le -v(s)\cdot p(s)+h(\eta(s)) l(s)|p(s)|+\delta
\end{equation}
for \textit{a.e.} $s\in[0,t]$,
\end{lem}
\begin{proof}
Note that for each $(x,-v,- l)\in\ol \Omega\times\R^n\times\R$, there is a point $p=p(x,v, l)\in \overline B_i(0)$ such that 
\[
L_{\C}^i(x,-v,- l)=-p\cdot v-H(x,p)+h(x) l|p|.
\]
By the continuity of $H$ and $L_{\C}^i$, 
for all $\delta>0$, there is $r=r(x,v, l)>0$ such that
\[
L_{\C}^i(y,-\tilde v,-\tilde  l)+H(y,p)
\le -\tilde v\cdot p+h(y)\tilde l|p|+\delta 
\]
for all 
$(y,\tilde v,\tilde l)\in (\overline\Omega\cap B_r(x))\times B_r(v)\times B_r( l)$. 
Hence, as $\overline\Omega\times\mathbb R^n\times\mathbb R$ is $\sigma$-compact, we may choose a sequence $(x_k,v_k, l_k,p_k,r_k)\subset \overline\Omega\times \mathbb R^n\times \mathbb R\times \overline B_i(0)\times (0,\infty)$ such that
\[
\overline\Omega\times\mathbb R^n\times\mathbb R\subset \bigcup_{k\in\mathbb N}B_{r_k}(x_k)\times B_{r_k}(v_k)\times B_{r_k}( l_k), 
\]
and for all $k\in\mathbb N$,
\[
L_{\C}^i(y,-\tilde v,-\tilde  l)+H(y,p_k)\le -\tilde v\cdot p_k+h(y)\tilde l|p_k|+\delta 
\]
for all 
$(y,\tilde v,\tilde l)\in B_{r_k}(x_k)\times B_{r_k}(v_k)\times B_{r_k}( l_k)$. 

Now, define $U_k:=(\overline\Omega\cap B_{r_k}(x_k))\times B_{r_k}(v_k)\times B_{r_k}( l_k)$ for $k\in\mathbb N$ and define the function $P:\overline\Omega\times\mathbb R^n\times\mathbb R\to\mathbb R^n$ by
\[
P(x,v, l):=p_k\quad \text{for all} \  (x,v, l)\in U_k\backslash \bigcup_{j<k}U_j, \ \text{and}\  k\in\mathbb N.\]
It is clear that $P$ is Borel measurable, and $P(x,v, l)\in \overline{B}_i(0)$ and
\[
L_{\C}^i(x,-v,- l)+H(x,P(x,v, l))\le -v\cdot P(x,v, l)+h(x) l|P(x,v, l)|+\delta,
\]
for $(x,v, l)\in \overline\Omega\times\mathbb R^n\times\mathbb R$.
We define $p\in L^\infty([0,t],\mathbb R^n)$ by $p(s):=P(\eta(s),-v(s), -l(s))$ for all $s\in[0,t]$. We see that $p(s)\in B_i(0)$, and 
\eqref{H+Li} holds for \textit{a.e.} $s\in[0,t]$. 
\end{proof}

The $L^\infty$-boundedness of $p$ in Lemma \ref{lem:basic-L} is essential for our analysis. In particular, in Lemma \ref{lem:lscc}, we need to pass to the limit under weak convergence in $L^1$ for the triple $(\dot\eta_k,v_k,l_k)$, while controlling the integral
\[\int_0^th(\eta_k(s))l_k(s)|p(s)|\, ds.\]
The boundedness of $p$ allows us to ensure that this term is well-behaved under weak convergence of $l_k$ and uniform convergence of $\eta_k$, thus enabling a lower semicontinuity argument. In \cite{BIM}, a similar passage to the limit is achieved using the Lipschitz continuity of the viscosity solutions. However, this approach is not applicable when we deal with the metric function introduced in Section \ref{sec:mt}, where such gradient information is not directly available. Therefore, boundedness of $p$ in \eqref{H+Li} is technically necessary.

Moreover, the lower semicontinuity lemma (Lemma \ref{lem:lscc}) below is new, and plays a crucial role in establishing the existence of minimizing curves for the metric function defined in Section \ref{sec:mt}. In particular, it ensures that the infimum in $m^*$ defined by \eqref{func:mc-star} is actually attained for some pair of boundary points. This property is essential for our proof and for further analysis on the prescribed contact angle boundary problem.

\begin{lem}\label{lem:lscc}
Let $\{(\eta_k,v_k, l_k)\}_{k\in\mathbb N}\subset \cup_{x\in\overline\Omega}\SP(x)$ with $\eta_k(0)$ bounded.
Assume that there are constants $C>0$ and $t>0$ such that
\[
\int_0^t L_{\C}(\eta_k(s),-v_k(s),- l_k(s))\, ds\le C\quad \textrm{for all}\ k\in\mathbb N.
\]
Then, there is $(\eta,v, l)\in  \cup_{x\in\ol\Omega}\SP(x)$, and a subsequence 
$\{(\eta_{k_j},v_{k_j}, l_{k_j})\}_{j\in\N}$ such that 
$\eta_{k_j}\to\eta$ uniformly on $[0,t]$, 
$(\dot\eta_{k_j},v_{k_j}, l_{k_j})\to(\dot\eta,v, l)$ weakly in $L^1([0,t],\R^n)\times L^1([0,t],\R^n)\times L^1([0,t],\R)$ 
as $j\to\infty$, and
\[
\int_0^t L_{\C}(\eta(s),-v(s),- l(s))\,ds
\le \liminf_{k\to\infty}\int_0^t L_{\C}(\eta_k(s),-v_k(s),- l_k(s))\,ds.
\]
\end{lem}
\begin{proof}
Since $H(x,p)$ satisfies \eqref{eq:K_0H}, we have
\begin{equation}\label{maxmax}
\begin{aligned}
L_{\C}(x,-v,- l)&\ge  \max_{p\in \mathbb R^n}\Big\{-p\cdot v-H(x,p)+h(x) l|p|\Big\}
\\ &\ge  
\max_{p\in\mathbb R^n}\Big\{-p\cdot v-\frac{|p|^2}{2}-K_0+h(x) l|p|\Big\}
\end{aligned}
\end{equation}
for all $(x,v,l)\in \overline{\Omega}\times\R^n\times[0,\infty)$. 
By Lemma \ref{lem:vl}, $l\leq |v|$ for $(\eta,v,l)\in \SP(x)$ and all $x\in\ol\Omega$. We only need to consider the case where $l\leq |v|$. We assume $l>0$. Then we have $|v|>-h(x)l\geq 0$, otherwise $|v|\leq -h(x)l<l$, which leads to a contradiction. Now $v\neq 0$. If the second maximum in \eqref{maxmax} is achieved at $p=0$, we take $\bar p=-(|v|+h(x)l)\frac{v}{|v|}$ to get
\[-\bar p\cdot v-\frac{|\bar p|^2}{2}-K_0+h(x) l|\bar p|=\frac{1}{2}(|v|+h(x)l)^2-K_0>-K_0,\]
which leads to a contradiction. Thus, the maximum is taken at $p\neq 0$. Then we take the derivative with respect to $p$ to get
\[
-v=\frac{\partial}{\partial p}\Big(\frac{|p|^2}{2}-h(x) l|p|\Big)=p-h(x) l\frac{p}{|p|}=\Big(|p|-h(x)l\Big)\frac{p}{|p|}.
\]
Note that $|p|>0\ge h(x)l$, 
the unit vector in the direction of $p$ is $-v/|v|$, and $|v|=|p|-h(x)l$, which implies that
\[
p=-(|v|+h(x) l)\frac{v}{|v|}.
\]
Then,
\begin{equation}\label{v+l2}
\begin{aligned}
 L_{\C}(x,-v,- l)
 &\ge  
 \max_{p\in\mathbb R^n}\Big\{-p\cdot v-\frac{|p|^2}{2}-K_0+h(x) l|p|\Big\}
 \\&=
 \frac{(|v|+h(x) l)^2}{2}-K_0\ge  -K_0,
 \end{aligned}
\end{equation}
for all $(x,v,l)\in \overline{\Omega}\times\R^n\times[0,\infty)$ with $l\leq |v|$. Here we note that the above estimate naturally holds for $l=0$. Similarly, we have
\begin{equation}\label{v-}
  L_{\C}(x,-v,- l)\le \frac{(|v|+h(x) l)^2}{2}+K_0, 
\end{equation}
for all $(x,v,l)\in \overline{\Omega}\times\R^n\times[0,\infty)$ with $l\leq |v|$.

Since $(\eta_k,v_k, l_k)\in \cup_{x\in\overline\Omega}\SP(x)$,
by Lemma \ref{lem:vl}, 
we know that $0\le  l_k\le |v_k|$. Noting that $\|h\|_{L^\infty(\R^n)}<1$, we have
\[|v_k(s)|+h(\eta_k) l_k(s)\ge  (1-\|h\|_{L^\infty(\R^n)})|v_k(s)|>0.\]
Combining this with \eqref{v+l2}, we obtain
\begin{align*}
C&\ge  \int_0^t L_{\C}(\eta_k(s),-v_k(s),- l_k(s))\,ds
\ge  \int_0^t \bigg[\frac{(|v_k(s)|+h(\eta_k(s)) l_k(s))^2}{2}-K_0\bigg]\,ds
\\ &\ge  \int_0^t \bigg[\frac{(1-\|h\|_{L^\infty(\R^n)})^2}{2}|v_k(s)|^2-K_0\bigg]\,ds.
\end{align*}
According to \cite[Theorem 2.12]{BGH}, $|v_k|$ is uniformly integrable. Then by Proposition \ref{prop:stability-Sk} (see also \cite[Lemma 5.4]{I13}, which is a consequence of the Dunford--Pettis theorem), 
there is a subsequence $\{(\eta_{k_j},v_{k_j}, l_{k_j})\}_{j\in\N}$ 
such that $\eta_{k_j}\to\eta$ uniformly on $[0,t]$, and 
$(\dot\eta_{k_j},v_{k_j}, l_{k_j})\to(\dot\eta,v, l)$ weakly in $L^1([0,t],\R^n)\times L^1([0,t],\R^n)\times L^1([0,t],\R)$ 
as $j\to\infty$.

It remains to show the lower semi-continuity. By Lemma \ref{lem:basic-L}, for any $i\in\mathbb N$, there is $p\in L^\infty([0,t],\R^n)$ such that $p(s)\in B_i(0)$ for all $s\in[0,t]$, 
and 
\begin{equation}\label{1/i}
  H(\eta(s),q(s))+L_{\C}^i(\eta(s),-v(s),- l(s))\le -v(s)\cdot p(s)+h(\eta(s)) l(s)|p(s)|+\frac{1}{i}
\end{equation}
for \textit{a.e.} $s\in[0,t]$. x
For all $k\in\mathbb N$,
\begin{align*}
&\int_0^t L_{\C}(\eta_k(s),-v_k(s),- l_k(s))\, ds
\\ &\ge  
\int_0^t \Big[-v_k(s)\cdot p(s)-H(\eta_k(s),p(s))+h(\eta_k(s)) l_k(s)|p(s)|\Big]\,ds. 
\end{align*}
Since $p\in L^\infty([0,t],\R^n)$, $(v_{k_j},l_{k_j})$ converges weakly in $L^1([0,t],\R^n)\times L^1([0,t],\R)$ as $j\to\infty$, we have
\begin{align*}
&\lim_{j\to\infty}\int_0^t \Big[-v_{k_j}(s)\cdot p(s)-H(\eta_{k_j}(s),p(s))+h(\eta_{k_j}(s)) l_{k_j}(s)|p(s)|\Big]\,ds
\\ &=\int_0^t \Big[-v(s)\cdot p(s)-H(\eta(s),p(s))+h(\eta(s)) l(s)|p(s)|\Big]ds
\end{align*}
Using \eqref{1/i} we get
\begin{align*}
&\liminf_{k\to\infty}\int_0^t L_{\C}(\eta_k(s),-v_k(s),- l_k(s))\,ds
\\ &\ge  \int_0^t \Big[-v(s)\cdot p(s)-H(\eta(s),p(s))+h(\eta(s)) l(s)|p(s)|\Big]\,ds
\\ &\ge  \int_0^t \Big[L_{\C}^i(\eta(s),-v(s),- l(s))-\frac{1}{i}\Big]\,ds.
\end{align*}
Now, by the monotone convergence theorem, letting $i\to\infty$, 
we get the conclusion. 
\end{proof}

\begin{proof}[Proof of Proposition {\rm\ref{prop:extremal-main}}]
We only prove (ii) since (i) is similar. 
Let $(x,t)\in\overline{\Omega}_\ep \times(0,\infty)$, and 
by Lemma \ref{lem:lscc}, there exists an extremal curve 
$(\eta^\ep,v^\ep,l^\ep)\in\SP_\ep(x)$ 
for $V_{\C}^\ep(x,t)$. 
Take $t_0\in(0,\infty)$ so that  $\eta^\ep$ is differentiable at $t_0$. 
For all $t>t_0$, we have 
\[
\frac{V_{\C}^\varepsilon(\eta^\ep(t),t)-V_{\C}^\varepsilon(\eta^\ep(t_0),t_0)}{t-t_0}=\frac{1}{t-t_0}\int_{t_0}^{t}L_{\C}\left(\frac{\eta^\ep(s)}{\varepsilon},-v^\ep(s),- l^\ep(s)\right)\,ds. 
\]
By Proposition \ref{prop:lip}, we have 
\begin{equation*}
 \begin{aligned}
&
\frac{V_{\C}^\varepsilon(\eta^\ep(t),t)-V_{\C}^\varepsilon(\eta^\ep(t_0),t_0)}{t-t_0}
\\
= \ & \frac{V_{\C}^\varepsilon\left(\eta^\ep(t), t\right)-V_{\C}^\varepsilon\left(\eta^\ep(t_0), t\right) 
+ V_{\C}^\varepsilon\left(\eta^\ep(t_0), t\right)-V_{\C}^\varepsilon\left(\eta^\ep(t_0), t_0\right)}{t-t_0}\\
\leq \ & 
\frac{\left|V_{\C}^\varepsilon\left(\eta^\ep(t), t\right)-V_{\C}^\varepsilon\left(\eta^\ep(t_0), t\right)\right|}{\left|\eta^\ep(t)-\eta^\ep(t_0)\right|} \cdot\frac{\left|\eta^\ep(t)-\eta^\ep(t_0)\right|}{t-t_0}+ \frac{\left|V_{\C}^\varepsilon\left(\eta^\ep(t_0), t\right)-V_{\C}^\varepsilon\left(\eta^\ep(t_0), t_0\right)\right|}{t-t_0}\\
\leq\  & C \frac{\left|\eta^\ep(t)-\eta^\ep(t_0)\right|}{t-t_0}+C,
    \end{aligned}
\end{equation*}
Therefore, letting $t\to t_0$, by Lemma \ref{lem:vl}, we get
\begin{equation}\label{ineq:extremal-1}
C+C|v^\ep(t_0)|\ge C+C|\dot\eta^\ep(t_0)|
\ge L_{\C}\left(\frac{\eta^\ep(t_0)}{\varepsilon},-v^\ep(t_0),- l^\ep(t_0)\right). 
\end{equation}
By using \eqref{v+l2}, we have 
\begin{equation}\label{ineq:extremal-2}
\begin{aligned}
&L_{\C}\left(\frac{\eta^\ep(t_0)}{\varepsilon},-v^\ep(t_0),- l^\ep(t_0)\right)
\\ &\ge  \frac{1}{2}\big(|v^\ep(t_0)|+h(\eta^\ep(t_0)) l^\ep(t_0)\big)^2-K_0\ge  \frac{1}{2}(1-\|h\|_{L^\infty(\R^n)})^2|v^\ep(t_0)|^2-K_0. 
\end{aligned}
\end{equation}
Combing \eqref{ineq:extremal-1} and \eqref{ineq:extremal-2}, we obtain the boundedness of $|v^\ep(t_0)|$ which is independent of $\ep$. 
Therefore, by Lemma \ref{lem:vl}, we obtain 
$\|\dot\eta^\ep\|_{\Li([0,\frac{t}{\ep}])}+
\|l^\ep\|_{\Li([0,\frac{t}{\ep}])}\leq M_0$
for some $M_0\ge0$ which is independent of $\ep$. 
\end{proof}

\section{Metric functions}\label{sec:mt}
\subsection{Definition of metric functions, and basic properties}
In this subsection, we define the metric functions 
$m_{\NN}, m_{\C}:[0,\infty)\times\overline{\Omega}\times\overline{\Omega}\to\R$ associated with Hamilton--Jacobi equations under the oblique derivative boundary conditions, and the prescribed contact angle boundary conditions. 

\begin{defn}
Let $x,y\in\overline{\Omega}$ and $t\ge0$. 
Define 
\begin{align}
&m_{\NN}(t,x,y)
:=
\inf\left\{\int_0^t 
L_{\NN}(\eta(s),-v(s),-l(s))\,ds\mid 
(\eta,v,l)\in\SP(x) \ \text{with} \ \eta(t)=y
\right\}, 
\label{func:mn}\\
&m_{\C}(t,x,y)
:=
\inf\left\{\int_0^t 
L_{\C}(\eta(s),-v(s),-l(s))\,ds\mid 
(\eta,v,l)\in\SP(x) \ \text{with} \ \eta(t)=y
\right\}. 
\label{func:mc}
\end{align}
\end{defn}

Note that if we take $\eta\in \AC([0,t],\ol \Omega)$ with $\eta(0)=x$ and $\eta(t)=y$, then $(\eta,\dot\eta,0)$ solves $\SP(x)$. 
Therefore, the sets in the infima in \eqref{func:mn} and \eqref{func:mc} 
are not empty. 

We then extend the functions $m_{\NN}$ and $m_{\C}$ to the functions on the whole space $[0,\infty)\times\R^n\times\R^n$ as in \cite{HJMT}. 

\begin{defn}
Let $x,y\in\R^n$ and $t\ge0$. 
Define 
\begin{align}
&m_{\NN}^\ast(t,x,y)
:=
\inf\{m_{\NN}(t,\hat x,\hat y) \mid \hat x,\hat y\in\partial\Omega,\ \hat x-x\in Y,\ \hat y-y\in Y\}, 
\label{func:mn-star}\\
&m_{\C}^\ast(t,x,y)
:=
\inf\{m_{\C}(t,\hat x,\hat y) \mid \hat x,\hat y\in\partial\Omega,\ \hat x-x\in Y,\ \hat y-y\in Y\}. 
\label{func:mc-star}
\end{align}
\end{defn}

\begin{lem}[{\rm\cite[Proposition 2.2]{HJMT}}]\label{HJMT:prop-2.2}
Let $M_0$ be the constant given by Proposition {\rm\ref{prop:extremal-main}}. 
Let $t\ge  \delta>0$ for some $\delta>0$, and $x,y\in\mathbb R^n$ with $|x-y|\le M_0t$. 
Then, there is an absolutely continuous curve $\xi:[0,t]\to \ol \Omega$ such that $\xi(0)=\tilde x$ and $\xi(t)=\tilde y$ for some $\tilde x,\tilde y\in\partial \Omega$ with $x-\tilde x\in Y$ and $\tilde y-y\in Y$. Moreover, for some $M_\Omega>0$ that only depends on $\partial\Omega$ and $n$, we have
\[
\|\dot\xi\|_{L^\infty([0,t])}\leq M_\Omega\Big(M_0+\frac{2\sqrt{n}}{\delta}\Big).\]%,\quad \tilde m^*(t,x,y)\leq \bigg(\frac{C_b^2(M_0+\frac{2\sqrt{n}}{\delta})^2}{2}+K_0\bigg)t.\]
\end{lem}

\begin{lem}\label{lem:mb}
Let $m^\ast$ be either $m^\ast_{\NN}$ or $m^\ast_{\C}$ defined by \eqref{func:mn-star} or \eqref{func:mc-star}, respectively, 
and let $M_0$, $M_\Omega$ be the constants given by Proposition {\rm\ref{prop:extremal-main}} and Lemma {\rm\ref{HJMT:prop-2.2}}. 
For each $(t,x,y)\in(0,\infty)\times\mathbb R^n\times\mathbb R^n$, $m^*(t,x,y)$ is bounded. Moreover, let $t\ge  \delta>0$ for some $\delta>0$ and $x,y\in\mathbb R^n$ with $|x-y|\le M_0t$. 
Then, we have 
\[m^*(t,x,y)\le\bigg(\frac{M_\Omega^2(M_0+\frac{2\sqrt{n}}{\delta})^2}{2}+K_0\bigg)t.\]
\end{lem}
\begin{proof}
Let $(\eta,v,l)\in\SP(x)$. It is easy to see that there is $K_1\ge  K_0$ such that for \textit{a.e.} $s\in(0,\infty)$
\begin{equation}\label{K1}
  L_{\NN}(\eta(s),-v(s),-l(s))\ge  \frac{1}{2}|v(s)|^2-C\|g\|_{L^\infty(\R^n)}|v|-K_0\ge -K_1.
\end{equation}
Also, by \eqref{v+l2}, we see that $m^*_{\NN}(t,x,y)$ and $m^\ast_{\C}(t,x,y)$ are bounded from below. 

It remains to prove that $m^*_{\NN}(t,x,y)$ and $m^*_{\C}(t,x,y)$ are bounded from above. 
We only prove that $m^\ast_{\C}$ is bounded from above as we can similarily prove that $m^\ast_{\NN}$ is bounded from above. 
We now fix $(t,x,y)\in[\delta,\infty)\times\mathbb R^n\times\mathbb R^n$ with $|x-y|\le M_0t$. 
By Lemma \ref{HJMT:prop-2.2}, there is an absolutely continuous curve 
$\xi:[0,t]\to \ol \Omega$ such that $\xi(0)=\tilde x$ and $\xi(t)=\tilde y$ for some $\tilde x,\tilde y\in\partial \Omega$ with $x-\tilde x\in Y$ and $\tilde y-y\in Y$. 
Note that $(\xi,\dot\xi,0)\in\SP(\tilde{x})$, which implies that 
\[
m^*_{\C}(t,x,y)\le \int_0^t L_{\C}(\xi(s),-\dot\xi(s),0)\,ds\le \Big(\frac{1}{2}|\dot\xi(s)|^2+K_0\Big)t\le \bigg(\frac{M_\Omega^2(M_0+\frac{2\sqrt{n}}{\delta})^2}{2}+K_0\bigg)t.
\]
%The proof is now complete.
\end{proof}

\begin{prop}\label{prop:ex-m}
Let $m_{\NN}$, $m_{\C}$, $m^\ast_{\NN}$, $m^\ast_{\C}$ be the functions defined by \eqref{func:mn}, \eqref{func:mc}, \eqref{func:mn-star}, \eqref{func:mc-star}, respectively.  
\begin{enumerate}
\item[{\rm(i)}] 
For each $(t,x,y)\in(0,\infty)\times\overline{\Omega}\times\overline{\Omega}$, 
there is $(\eta,v,l)\in\SP(x)$ with $\eta(t)=y$ such that 
\[
m_{\NN}(t,x,y)
=
\int_0^t L_{\NN}(\eta(s),-v(s),-l(s))\,ds.
\] 
Moreover, for $(t,x,y)\in(0,\infty)\times\mathbb R^n\times\mathbb R^n$, 
there are 
$\tilde x,\tilde y\in\partial \Omega$ with $x-\tilde x\in Y$ and $\tilde y-y\in Y$, and 
$(\eta,v,l)\in\SP(\tilde{x})$ such that $\eta(t)=\tilde y$, and
\[
m^*_{\NN}(t,x,y)
=\int_0^t L_{\NN}(\eta(s),-v(s),-l(s))\,ds.
\]
\item[{\rm(ii)}] 
For each $(t,x,y)\in(0,\infty)\times\overline{\Omega}\times\overline{\Omega}$, 
there is $(\eta,v,l)\in\SP(x)$ with $\eta(t)=y$ such that 
\[
m_{\C}(t,x,y)
=
\int_0^t L_{\C}(\eta(s),-v(s),-l(s))\,ds.
\] 
Moreover, for $(t,x,y)\in(0,\infty)\times\mathbb R^n\times\mathbb R^n$, 
there are 
$\tilde x,\tilde y\in\partial \Omega$ with $x-\tilde x\in Y$ and $\tilde y-y\in Y$, and 
$(\eta,v,l)\in\SP(\tilde{x})$ such that $\eta(t)=\tilde y$, and
\[
m^*_{\C}(t,x,y)
=\int_0^t L_{\C}(\eta(s),-v(s),-l(s))\,ds.
\]
\end{enumerate}
\end{prop}

\begin{proof}
We only prove (ii), since we can prove (i) similarly. 
Similar to Lemma \ref{lem:mb}, we can show that $m_{\C}(t,x,y)$ is bounded from above for each $t>0$ and $x,y\in\R^n$.
Take $\{(\eta_k,v_k, l_k)\}_{k\in\mathbb N}\subset \SP(x)$ with $\eta_k(t)=y$ so that 
\[
\lim_{k\to\infty}\int_0^t L_{\C}(\eta_k(s),-v_k(s),-l_k(s))\,ds=m_{\C}(t,x,y).
\]
%We only need to consider the case where $m(t,x,y)<\infty$. Otherwise $m(t,x,y)\equiv\infty$. For $k$ large, there is a constant $C>0$ independent of $k$ such that
%\[\int_0^t \big[L(\eta_k(s),-v_k(s))+g(\eta_k(s)) l_k(s)\big]ds\le C.\]
Then,
\[\int_0^t L_{\C}(\eta_k(s),-v_k(s),-l_k(s))\,ds\]
is uniformly bounded from above. 
By the lower semi-continuity property, Lemma \ref{lem:lscc}, 
there is $(\eta,v, l)\in \SP(x)$ and a subsequence $\{(\eta_{k_j},v_{k_j}, l_{k_j})\}_{j\in\N}$ such that 
$\eta_{k_j}\to\eta$ uniformly on $[0,t]$ with $\eta(t)=y$, 
and $(\dot\eta_{k_j},v_{k_j}, l_{k_j})\to(\dot\eta,v, l)$ weakly in 
$L^1([0,t],\R^n)\times L^1([0,t],\R^n)\times L^1([0,t],\R)$ as $j\to\infty$, and
\[\int_0^t L_{\C}(\eta(s),-v(s),-l(s))\,ds\le m_{\C}(t,x,y).\]
By the definition of $m_{\C}(t,x,y)$, we obtain 
\[
\int_0^t 
L_{\C}(\eta(s),-v(s),-l(s))\, ds
=
m_{\C}(t,x,y).
\]

Next, we take $\tilde x_k,\tilde y_k\in\partial \Omega$ with $\tilde x_k-x\in Y$ and $\tilde y_k-y\in Y$ and $(\eta_k,v_k, l_k)\in\SP(\tilde x_k)$ 
with $\eta_k(t)=\tilde y_k$ such that
\[
m_{\C}(t,\tilde x_k,\tilde y_k)
=
\int_0^t L_{\C}(\eta_k(s),-v_k(s),-l(s))\,ds, 
\]
and
\[
\lim_{k\to\infty}m_{\C}(t,\tilde x_k,\tilde y_k)=m^*_{\C}(t,x,y).
\]
By Lemma \ref{lem:mb}, there is a constant $C>0$ independent of $k$ such that
\[\int_0^t 
L_{\C}(\eta_k(s),-v_k(s),-l_k(s))\,ds\le Ct.
\]
%Since $\eta_k(0)=\tilde x_k$ is bounded, 
By the compactness, there is $(\eta,v, l)\in\SP(x)$ satisfying $\eta(0),\eta(t)\in\partial\Omega$ and a subsequence $\{(\eta_{k_j},v_{k_j}, l_{k_j})\}_{j\in\N}$ such that 
$\eta_{k_j}\to\eta$ uniformly $[0,t]$, and $(\dot\eta_{k_j},v_{k_j}, l_{k_j})\to(\dot\eta,v, l)$ weakly in $L^1([0,t],\R^n)\times L^1([0,t],\R^n)\times L^1([0,t],\R)$ as $j\to\infty$ and
\[
\int_0^t 
L_{\C}(\eta(s),-v(s),-l(s))\,ds\le m^*_{\C}(t,x,y).
\]
By the definition of $m^*_{\C}(t,x,y)$ we obtain 
\[
\int_0^t L_{\C}(\eta(s),-v(s),-l(s))\,ds=m^*_{\C}(t,x,y), 
\]
which completes the proof. 
\end{proof}

\subsection{Difference of metric functions and extended metric functions}
We first give a proof of a kind of subadditivity properties of $m_{\NN}^\ast$ and $m_{\C}^\ast$. 

\begin{prop}\label{prop:tri3}
Let $m^\ast$ be either $m^\ast_{\NN}$ or $m^\ast_{\C}$ defined by \eqref{func:mn-star} or \eqref{func:mc-star}, respectively. 
Let $M_0$ be the constant given by Proposition {\rm\ref{prop:extremal-main}}. 
There exists a constant $C>0$ such that 
for all $t,\tau>0$ with $t\ge  1$ or $\tau\ge  1$, and 
$x,y,z\in\mathbb R^n$ satisfying $|x-y|\le M_0t$ and $|y-z|\le M_0\tau$, 
we have 
\[
m^*(t+\tau,x,z)\le m^*(t,x,y)+m^*(\tau,y,z)+C.
\]
\end{prop}

\begin{proof}
We first consider the case where $m^\ast=m^\ast_{\C}$.  
Without loss of generality, we can assume $t\ge  1$. 
By Proposition \ref{prop:ex-m}, there are 
$\tilde{x}, \tilde{y}, \tilde{y}', \tilde{z}\in\partial\Omega$ satisfying 
$\tilde x-x, \tilde y-y, \tilde{y}'-y, \tilde{z}-z\in Y$, and  
\[
(\eta_1,v_1, l_1)\in\SP(\tilde{x}) \  \text{with} \  \eta_1(t)=\tilde y, 
\quad  
(\eta_2,v_2, l_2)\in\SP(\tilde{y}') \ \text{with} \ \eta_2(\tau)=\tilde z
\] 
such that 
\begin{align*}
&m^*_{\C}(t,x,y)=\int_0^tL_{\C}(\eta_1(s),-v_1(s),- l_1(s))\,ds, \\
&m^*_{\C}(\tau,y,z)=\int_0^\tau L_{\C}(\eta_2(s),-v_2(s),- l_2(s))\,ds. 
\end{align*}
By the same argument in the proof of \cite[Proposition 2.3]{HJMT}, 
we can prove that 
there exists $d\in\{0,\frac{1}{4},\frac{1}{2},\frac{3}{4},\dots,\lfloor t\rfloor-\frac{1}{4}\}$ such that
\begin{equation}\label{c13}
  \int_d^{d+\frac{1}{4}}
  L_{\C}(\eta_1(s),-v_1(s),- l_1(s))ds\le 
M
\end{equation}
for some $M>0$ which only depends on $M_\Omega, M_0, K_0, n$. 
Here, we denote by $\lfloor r\rfloor$ the greatest integer less than or equal to $r\in\R$.
By Lemma \ref{HJMT:prop-2.2}, we can find a path $\xi:[0,1]\to\overline \Omega$ such that $\xi(0)=\tilde y$ and $\xi(1)=\tilde y'$ with 
$\|\dot{\xi}\|_{\Li([0,t])}\le C$ for some $C>0$, 
since $|\tilde y-\tilde y'|\le 2\sqrt{n}$.

%Recall that $d$ is given in \eqref{c13}. 
Set
\begin{equation*}
(\eta_3(s), l_3(s)):=
\begin{cases}
(\eta_1(s), l_1(s)) &\text{for} \ 0\le s\le d,
\\ (\eta_1(2(s-d)+d),2 l_1(2(s-d)+d)) & \text{for} \ d\le s\le d+\frac{1}{8},
\\ (\eta_1(s+\frac{1}{8}), l_1(s+\frac{1}{8})) &\text{for} \ d+\frac{1}{8}\le s\le t-\frac{1}{8},
\\ (\xi(8(s-t+\frac{1}{8})),0) &\text{for} \ t-\frac{1}{8}\le s\le t,
\\ (\eta_2(s-t), l_2(s-t)) &\text{for} \ t\le s\le t+\tau.
\end{cases}
\end{equation*}
Letting 
$v_3(s):=\dot{\eta}_3(s)+ l_3(s)\nu(\eta_3(s))$, 
we can easily see that 
\[
(\eta_3,v_3, l_3)\in \SP(\tilde{x}) \quad\text{with} \quad \eta_4(t+\tau)=\tilde{z}. 
\]
Therefore,
\[
m^*_{\C}(t+\tau,x,z)\le \int_0^{t+\tau}L_{\C}(\eta_3(s),-v_3(s),- l_3(s))\,ds,
\]
that is, 
\begin{align*}
m^*_{\C}(t+\tau,x,z)\le& \int_0^dL_{\C}(\eta_1(s),-v_1(s),- l_1(s)) \,ds
\\ 
&+\int_d^{d+\frac{1}{8}}L_{\C}(\eta_1(2(s-d)+d),-2v_1(2(s-d)+d),-2 l_1(2(s-d)+d))\,ds
\\ 
&+\int_{d+\frac{1}{8}}^{t-\frac{1}{8}}
L_{\C}\left(\eta_1\left(s+\frac{1}{8}\right),-v_1\left(s+\frac{1}{8}\right),- l_1\left(s+\frac{1}{8}\right)\right)\,ds
\\ 
&+\int_{t-\frac{1}{8}}^t
L_{\C}\left(\xi\left(8\left(s-t+\frac{1}{8}\right)\right),-8\dot\xi\left(8\left(s-t+\frac{1}{8}\right)\right),0\right)\,ds
\\ 
&+\int_t^{t+\tau}L_{\C}(\eta_2(s-t),-v_2(s-t),- l_2(s-t))\,ds. 
\end{align*}
Note that, by \eqref{v+l2}, 
\begin{equation}\label{int0d}
\begin{aligned}
&\int_0^dL_{\C}(\eta_1(s),-v_1(s),- l_1(s))\,ds
\\ 
&+\int_{d+\frac{1}{8}}^{t-\frac{1}{8}}L_{\C}\left(\eta_1\left(s+\frac{1}{8}\right),-v_1\left(s+\frac{1}{8}\right),- l_1\left(s+\frac{1}{8}\right)\right) \,ds
\\ 
&=\int_0^dL_{\C}(\eta_1(s),-v_1(s),- l_1(s))\,ds
+\int_{d+\frac{1}{4}}^tL_{\C}(\eta_1(s),-v_1(s),- l_1(s))\,ds
\\ 
&=\int_0^tL_{\C}(\eta_1(s),-v_1(s),- l_1(s))\,ds
-\int_d^{d+\frac{1}{4}}L_{\C}(\eta_1(s),-v_1(s),- l_1(s))\,ds 
\\ &\le m^*(t,x,y)+\frac{K_0}{4}. 
\end{aligned}
\end{equation}
Also, by \eqref{v-}, and \eqref{v+l2} once again, 
\begin{align}
&\int_d^{d+\frac{1}{8}}
L_{\C}(\eta_1(2(s-d)+d),-2v_1(2(s-d)+d),-2 l_1(2(s-d)+d))\,ds
\label{intd}\\
&
=\frac{1}{2}\int_d^{d+\frac{1}{4}}L_{\C}(\eta_1(s),-2v_1(s),-2 l_1(s))\,ds
\nonumber\\
&\le \frac{1}{2}\int_d^{d+\frac{1}{4}}\Big[\frac{1}{2}(2|v_1(s)|+2h(\eta_1(s)) l_1(s))^2+K_0\Big]\,ds
\nonumber\\ &=\int_d^{d+\frac{1}{4}}(|v_1(s)|+h(\eta_1(s)) l_1(s))^2ds+\frac{K_0}{8}
\nonumber\\ &\le 2\int_d^{d+\frac{1}{4}}L_{\C}(\eta_1(s),-v_1(s),- l_1(s))\,ds+\frac{5K_0}{8}\le 
2M+\frac{5K_0}{8}. \nonumber
\end{align}
Moreover, we have
\begin{align*}
&\int_{t-\frac{1}{8}}^tL_{\C}\left(\xi\left(8\left(s-t+\frac{1}{8}\right)\right),-8\dot\xi\left(8\left(s-t+\frac{1}{8}\right)\right),0\right)\,ds
\\ &=\frac{1}{8}\int_0^1L_{\C}(\xi(s),-8\dot\xi(s),0)\,ds
\le 
4\int_0^1 |\dot\xi(s)|^2ds+\frac{K_0}{8}
\le 4C^2+\frac{K_0}{8},
\end{align*}
and
\[
\int_t^{t+\tau}L_{\C}(\eta_2(s-t),-v_2(s-t),- l_2(s-t))ds=m^*_{\C}(\tau,y,z).
\]
Combining all inequalities above, we get that there exists $C>0$ such that
\[
m^*(t+\tau,x,z)\le m^*(t,x,y)+m^*(\tau,y,z)+C, 
\]
which completes the proof. 

Now, we consider the case where $m^*=m^*_{\NN}$. The argument is quite similar to the case where $m^*=m^*_{\C}$, so we only point out different parts. By \eqref{K1}, $L_{\NN}$ is bounded from below by $-K_1$. We then obtain \eqref{c13} by using a similar argument to the above, where $M>0$ depends on $K_1$ instead of $K_0$. In \eqref{int0d}, we get $m^*(t,x,y)+\frac{K_1}{4}$ in the last inequality by \eqref{K1} again. The main difference is in the estimate of \eqref{intd}. Note that
\[\int_d^{d+\frac{1}{4}}|v_1(s)|\, ds\le \frac{1}{2}\bigg(\int_d^{d+\frac{1}{4}}|v_1(s)|^2\, ds\bigg)^{1/2}.\]
From \eqref{c13} and Lemma \ref{lem:vl} we know that there is $C>0$ such that
\begin{align*}
M&\ge \int_d^{d+\frac{1}{4}}\bigg[\frac{|v_1(s)|^2}{2}-K_0-C\|g\|_{L^\infty(\R^n)}|v_1(s)|\bigg]\, ds
\\ &\geqslant \frac{1}{2}\int_d^{d+\frac{1}{4}}|v_1(s)|^2\, ds-\frac{K_0}{4}-\frac{1}{2}C\|g\|_{L^\infty(\R^n)} \bigg(\int_d^{d+\frac{1}{4}}|v_1(s)|^2\, ds\bigg)^{1/2},
\end{align*}
which implies that both
\[\int_d^{d+\frac{1}{4}}|v_1(s)|^2\, ds\]
and
\[\int_d^{d+\frac{1}{4}}l_1(s)\, ds\leqslant \int_d^{d+\frac{1}{4}}C|v_1(s)|\, ds\leqslant \frac{C}{2}\bigg(\int_d^{d+\frac{1}{4}}|v_1(s)|^2\, ds\bigg)^{1/2}\]
are bounded. Therefore,
\begin{align*}
&\int_d^{d+\frac{1}{8}}
L_{\NN}(\eta_1(2(s-d)+d),-2v_1(2(s-d)+d),-2 l_1(2(s-d)+d))\,ds
\\ &=\frac{1}{2}\int_d^{d+\frac{1}{4}}L_{\NN}(\eta_1(s),-2v_1(s),-2\ell_1(s))\, ds
\\ &\le \frac{1}{2}\int_d^{d+\frac{1}{4}}\bigg[\frac{(2|v_1(s)|)^2}{2}+K_0+2g(\eta_1(s))\ell_1(s)\bigg]\, ds
\end{align*}
is bounded.
\end{proof}

\begin{prop}\label{prop:difference-mm}
Let $m$ be either $m_{\NN}$ or $m_{\C}$ defined by \eqref{func:mn} or \eqref{func:mc}, respectively, 
and $m^\ast$ be either $m^\ast_{\NN}$ or $m^\ast_{\C}$ defined by \eqref{func:mn-star} or \eqref{func:mc-star}, respectively. 
Let $M_0$ be the constant given by Proposition {\rm\ref{prop:extremal-main}}. 
There exists a constant $C>0$ such that for all $t\ge  1$ and $x,y\in\ol \Omega$ with $|x-y|\le M_0t$,
\[
|m^*(t,x,y)-m(t,x,y)|\le C.
\]
\end{prop}

\begin{proof}
We only give a proof for $m_{\C}$, $m^\ast_{\C}$ since we can similarly prove for $m_{\NN}$, $m^\ast_{\NN}$. 
We prove 
\[
m_{\C}^*(t,x,y)\le m_{\C}(t,x,y)+C. 
\]
Symmetrically, we can prove the other inequality. 

By Proposition \ref{prop:ex-m}, there is $(\eta_1,v_1, l_1)\in\SP(x)$ with $\eta_1(t)=y$ 
such that 
\[
m_{\C}(t,x,y)=\int_0^t L_{\C}(\eta_1(s),-v_1(s),- l_1(s))\,ds.
\]
Let $t\ge  1$ and $x,y\in\ol \Omega$ with $|x-y|\le M_0t$. 
By using Lemma \ref{HJMT:prop-2.2}, 
there is a path $\xi_1:[0,t]\to\overline\Omega$ connecting $x$ and $y$ with $\|\dot\xi_1\|_{\Li([0,t])}\le M$ for some $M>0$. 
Also, note that $(\xi_1,\dot\xi_1,0)\in\SP(x)$ with $\xi_1(t)=y$,  
which implies 
\[
m_{\C}(t,x,y)=\int_0^t L_{\C}(\eta_1(s),-v_1(s),- l_1(s))\,ds
\le \int_0^t
L_{\C}(\xi_1(s),\dot\xi_1(s),0)\,ds
\le \Big(\frac{M^2}{2}+K_0\Big)t 
\]
by \eqref{v-}. 
As in the proof of Proposition \ref{prop:tri3}, 
there exists $d\in\{0,\frac{1}{4},\frac{1}{2},\frac{3}{4},\dots,\lfloor t\rfloor-\frac{1}{4}\}$ such that \eqref{c13} holds. 
Take any $\tilde x,\tilde y\in\ol \Omega$ such that $\tilde x-x\in Y$ and $\tilde y-y\in Y$. 
Since $|\tilde x-x|\le \sqrt{n}$, $|\tilde y-y|\le \sqrt{n}$ and $t\ge  1$, 
there is a path $\xi_2:[0,1]\to\ol \Omega$ connecting $\tilde x$ to $x$, and a path $\xi_3:[0,1]\to\ol \Omega$ connecting $\tilde y$ to $y$ 
with 
$\|\dot\xi_2\|_{L^\infty([0,1])}+\|\dot\xi_3\|_{L^\infty([0,1])}\le C$. 

Define
\begin{equation*}
(\eta_2(s), l_2(s))
:=
\begin{cases}
(\xi_2(16s),0)& \text{for} \ 0\le s\le \frac{1}{16}
\\  
(\eta_1(s-\frac{1}{16}), l_1(s-\frac{1}{16}))
& \text{for} \ \frac{1}{16}\le s\le d+\frac{1}{16},
\\ 
(\eta_1(2(s-d-\frac{1}{16})+d),2 l_1(2(s-d-\frac{1}{16})+d) 
& 
\text{for} \ d+\frac{1}{16}\le s\le d+\frac{3}{16},
\\ 
(\eta_1(s+\frac{1}{16}), l_1(s+\frac{1}{16})) &
 \text{for} \ d+\frac{3}{16}\le s\le t-\frac{1}{16},
\\ 
(\xi_3(16(s-t+\frac{1}{16})),0) & \text{for} \ t-\frac{1}{16}\le s\le t.
\end{cases}
\end{equation*}
Then, 
$(\eta_2,v_2, l_2)$ is an admissible control of $m_{\C}^*(t,x,y)$. 
Therefore,
\begin{align*}
&\, m_{\C}^*(t,x,y)\\
\le 
&\, 
\int_{0}^t L_{\C}(\eta_2(s), -v_2(s), -l_2(s))\, ds\\
=&\, 
\int_0^{\frac{1}{16}}L_{\C}(\xi_2(16s),-16\dot\xi_2(16s),0)\,ds
\\ &
+\int_{\frac{1}{16}}^{d+\frac{1}{16}}
L_{\C}\left(\eta_1\left(s-\frac{1}{16}\right),-v_1\left(s-\frac{1}{16}\right),- l_1\left(s-\frac{1}{16}\right)\right)\,ds
\\ &
+\int_{d+\frac{1}{16}}^{d+\frac{3}{16}}
L_{\C}\bigg(\eta_1\left(2\left(s-d-\frac{1}{16}\right)+d\right),-2v_1\left(2\left(s-d-\frac{1}{16}\right)+d\right),
\\ &\hspace{200pt} -2 l_1\left(2\left(s-d-\frac{1}{16}\right)+d\right)\bigg)\,ds
\\ &
+\int_{d+\frac{3}{16}}^{t-\frac{1}{16}}
L_{\C}\left(\eta_1\left(s+\frac{1}{16}\right),-v_1\left(s+\frac{1}{16}\right),- l_1\left(s+\frac{1}{16}\right)\right)\,ds
\\ &
+\int_{t-\frac{1}{16}}^t 
L_{\C}\left(\xi_3\left(16\left(s-t+\frac{1}{16}\right)\right),-16\dot\xi_3\left(16\left(s-t+\frac{1}{16}\right)\right),0\right)\,ds.
\end{align*}
Here, 
by repeating an argument in the proof of Proposition \ref{prop:tri3}, we get
\begin{align*}
&\int_{\frac{1}{16}}^{d+\frac{1}{16}}
L_{\C}\left(\eta_1\left(s-\frac{1}{16}\right),-v_1\left(s-\frac{1}{16}\right),- l_1\left(s-\frac{1}{16}\right)\right)\,ds
\\ &\quad 
+\int_{d+\frac{3}{16}}^{t-\frac{1}{16}}
L_{\C}\left(\eta_1\left(s+\frac{1}{16}\right),-v_1\left(s+\frac{1}{16}\right),- l_1\left(s+\frac{1}{16}\right)\right)\,ds 
\\ &\le m_{\C}(t,x,y)
-\int_d^{d+\frac{1}{4}}
L_{\C}(\eta_1(s),-v_1(s),- l_1(s))\,ds
\le m_{\C}(t,x,y)+\frac{K_0}{4}, 
\end{align*}
and
\begin{align*}
&\int_{d+\frac{1}{16}}^{d+\frac{3}{16}}
L_{\C}\bigg(\eta_1\left(2\left(s-d-\frac{1}{16}\right)+d\right),-2v_1\left(2\left(s-d-\frac{1}{16}\right)+d\right),
\\ &\hspace{200pt} -2 l_1\left(2\left(s-d-\frac{1}{16}\right)+d\right)\bigg)\,ds
\\ &
=\frac{1}{2}\int_d^{d+\frac{1}{4}}
L_{\C}\left(\eta_1(s),-2v_1(s),-2 l_1(s)\right)\,ds\le 2M+\frac{5K_0}{8}, 
\end{align*}
and
\begin{multline*}
\int_0^{\frac{1}{16}}
L_{\C}(\eta_2(16s),-16\dot\eta_2(16s),0)\, ds\\
+\int_{t-\frac{1}{16}}^t 
L_{\C}\left(\eta_3\left(16\left(s-t+\frac{1}{16}\right)\right),-16\dot\eta_3\left(16\left(s-t+\frac{1}{16}\right)\right),0\right)\,ds
\le 16C^2+\frac{K_0}{8}. 
\end{multline*}
Combining all inequalities above, we get that there exists $C>0$ such that
\[
m^*_{\C}(t,x,y)\le m_{\C}(t,x,y)+C, 
\]
which completes the proof. 
\end{proof}

\subsection{Subadditivity and Superadditivity}\label{subsec:sa}

\begin{lem}\label{lem:subad}
Let $m^\ast$ be either $m^\ast_{\NN}$ or $m^\ast_{\C}$ defined by \eqref{func:mn-star} or \eqref{func:mc-star}, respectively.  
Let $M_0$ be the constant given by Proposition {\rm\ref{prop:extremal-main}}. 
There exists $C\ge 0$ such that for $t\ge  1$, $y\in\R^n$ with $|y|\le M_0t$, we have
\[
m^*(2t,0,2y)\le 2m^*(t,0,y)+C.
\]
\end{lem}

\begin{proof}
We only consider the case of $m^\ast_{\C}$, and we can similarly prove for $m^\ast_{\NN}$. 
By Proposition \ref{prop:tri3} we know
\[m_{\C}^*(2t,0,2y)\le m_{\C}^*(t,0,y)+m_{\C}^*(t,y,2y)+C
\]
for some $C\ge0$. 
Therefore, we only need to prove 
\[
m^*_{\C}(t,y,2y)\le m^*_{\C}(t,0,y)+C.
\]

By Proposition \ref{prop:ex-m}, 
there exist $\tilde x,\tilde y\in\partial\Omega$ with $\tilde x-0\in Y$, $\tilde y-y\in Y$, and 
$(\eta_1,v_1,l_1)\in\SP(\tilde{x})$ with $\eta(t)=\tilde y$ such that 
\[
m^\ast_{\C}(t,0,y)=\int_0^tL_{\C}(\eta_1(s),-v_1(s),-l_1(s))\,ds. 
\]
As in the proof of Proposition \ref{prop:tri3}, 
there exists $d\in\{0,\frac{1}{4},\frac{1}{2},\frac{3}{4},\dots,\lfloor t\rfloor-\frac{1}{4}\}$ such that
\begin{equation}\label{c3}
  \int_d^{d+\frac{1}{4}}L_{\C}(\eta_1(s),-v_1(s),-l_1(s))\,ds\le M
\end{equation}
for some $M\ge0$. 
Take $k\in\mathbb Z^n$ so that $\tilde x+k\in y+Y$ and define 
$\tilde\eta_1:=\eta_1+k$. 
%Then, $(\tilde{\eta},v,l)$ still solves \eqref{SP1}. 
Take $\tilde z\in\partial \Omega$ so that $\tilde z-2y\in Y$. 
By Proposition \ref{HJMT:prop-2.2} we can find a path 
$\xi:[0,1]\to\overline\Omega$ such that $\xi(0)=\tilde y+k$ and $\xi(1)=\tilde z$ with 
\[
\|\dot{\xi}\|_{\Li([0,1])}\le 6\sqrt{n}C_b=:C,
\]
since
\[|\tilde y+k-\tilde z|\le |\tilde y-y|+|\tilde x|+|\tilde x+k-y|+|2y-\tilde z|\le 4\sqrt{n}.\]
Define
\begin{equation*}
\begin{aligned}
(\eta_2(s),l_2(s))
:=&
\begin{cases}
(\tilde{\eta}_1(s),l_1(s)) & 
\text{for} \ 0\le s\le d,
\\ 
(\tilde{\eta}_1(2(s-d)+d),2l_1(2(s-d)+d) 
&\text{for} \ d\le s\le d+\frac{1}{8},
\\ (\tilde{\eta}_1(s+\frac{1}{8}),l_1(s+\frac{1}{8})) 
&
\text{for} \ d+\frac{1}{8}\le s\le t-\frac{1}{8},
\\ (\xi(8(s-t+\frac{1}{8})),0) &\text{for} \ t-\frac{1}{8}\le s\le t, 
\end{cases}
\\ v_2(s):=&\dot{\eta}_2(s)+l_2(s)\nu(\eta_2(s)).
\end{aligned}
\end{equation*}
Then, $(\eta_3,v_3,l_3)\in \SP(x)$ is an admissible control for $m^*(t,y,2y)$. Note that $\dot{\eta}_1=\dot{\tilde{\eta}}_1$. 
Thus, 
\begin{align*}
&m^*_{\C}(t,y,2y)\\
\le&\,  
\int_0^d L_{\C}(\tilde{\eta}_1(s),-v_1(s),-l_1(s))\, ds
\\ &+\int_{d}^{d+\frac{1}{8}}
L_{\C}(\tilde{\eta}_1(2(s-d)+d),-2v_1(2(s-d)+d),-2l_1(2(s-d)+d))\,ds
\\ &
+\int_{d+\frac{1}{8}}^{t-\frac{1}{8}}
L_{\C}\bigg(\eta_1\bigg(s+\frac{1}{8}\bigg),-v_1\bigg(s+\frac{1}{8}\bigg),-l_1\bigg(s+\frac{1}{8}\bigg)\bigg)\,ds
\\ &
+\int_{t-\frac{1}{8}}^{t}L_{\C}\bigg(\eta_2\bigg(8(s-t+\frac{1}{8}\bigg)\bigg),-8\dot{\eta}_2\bigg(8\bigg(s-t+\frac{1}{8}\bigg)\bigg),0\bigg)\,ds
\\
=&\, 
\int_0^d L_{\C}(\eta_1(s),-v_1(s),-l_1(s))\,ds
\\ &
+\int_{d}^{d+\frac{1}{8}}
L_{\C}(\eta_1(2(s-d)+d),-2v_1(2(s-d)+d),-2l_1(2(s-d)+d))\,ds
\\ &
+\int_{d+\frac{1}{8}}^{t-\frac{1}{8}}
L_{\C}\bigg(\eta_1\bigg(s+\frac{1}{8}\bigg),-v_1\bigg(s+\frac{1}{8}\bigg),-l_1\bigg(s+\frac{1}{8}\bigg)\bigg)\,ds
\\ &
+\int_{t-\frac{1}{8}}^{t}
L_{\C}\bigg(\xi\bigg(8(s-t+\frac{1}{8}\bigg)\bigg),-8\dot{\xi}\bigg(8\bigg(s-t+\frac{1}{8}\bigg)\bigg),0\bigg)\,ds.
\end{align*}
Repeating an argument in the proof of Proposition \ref{prop:tri3}, 
we obtain
\begin{align*}
&\int_0^d L(\eta_1(s),-v_1(s),-l_1(s))ds
\\ &+\int_{d+\frac{1}{8}}^{t-\frac{1}{8}}L\bigg(\eta_1\bigg(s+\frac{1}{8}\bigg),-v_1\bigg(s+\frac{1}{8}\bigg),-l_1\bigg(s+\frac{1}{8}\bigg)\bigg)ds\le m^*(t,0,y)+\frac{K_0}{4},
\end{align*}
and
\begin{align*}
&\int_{d}^{d+\frac{1}{8}}
L_{\C}(\eta_1(2(s-d)+d),-2v_1(2(s-d)+d),-2l_1(2(s-d)+d))\,ds
\\ &\le 2M+\frac{5K_0}{8},
\end{align*}
and
\begin{align*}
&\int_{t-\frac{1}{8}}^{t}
L_{\C}\bigg(\xi\bigg(8\bigg(s-t+\frac{1}{8}\bigg)\bigg),-8\dot{\xi}\bigg(8\bigg(s-t+\frac{1}{8}\bigg)\bigg),0\bigg)\, ds
\\ &
=\frac{1}{8}\int_0^1L(\xi(s),8\dot{\xi}(s),0)\, ds\le 4\int_0^1|\dot{\xi}(s)|^2\, ds+\frac{K_0}{8}\le 4C^2+\frac{K_0}{8}. 
\end{align*}
Combining all inequalities above, we obtain the conclusion. 
\end{proof}

\begin{lem}\label{prop:super-ad}
Let $m^\ast$ be either $m^\ast_{\NN}$ or $m^\ast_{\C}$ defined by \eqref{func:mn-star} or \eqref{func:mc-star}, respectively.  
Let $M_0$ be the constant given by Proposition {\rm\ref{prop:extremal-main}}. 
There is $C>0$ such that for all $t\ge  1$, $|y|\le M_0t$, we have
\[
2m^*(t,0,y)\le m^*(2t,0,2y)+C.
\]
\end{lem}

\begin{proof}
We only consider the case of $m^\ast_{\C}$, and we can similarly prove for $m^\ast_{\NN}$. 
By Proposition \ref{prop:ex-m}, there exist $\tilde{x}, \tilde{z}\in\partial\Omega$ with $\tilde x-0\in Y$, $\tilde z-2y\in Y$, 
and $(\eta_1,v_1, l_1)\in\SP(\tilde{x})$ with $\eta_1(2t)=\tilde z$ such that 
\[
m^\ast_{\C}(2t,0,2y)=
\int_0^{2t}L_{\C}(\eta_1(s),-v_1(s),-l(s))\, ds. 
\]
Set $\gamma(s):=(\eta_1(s),s)$ for $s\in [0,2t]$. 
By using Burago's cutting lemma (see \cite[Lemma 2]{Burago}, and also \cite[Lemma 2.1]{TY}), we do surgery as follows. 

First, owing to Burago's cutting lemma, 
there exists a collection of disjoint intervals $\{[a_i,b_i]\}_{1\le i\le k}\subset [0,2t]$ with $k\le \frac{n+2}{2}$ such that
\begin{equation}\label{eq:Burago}
\sum_{i=1}^{k}\big(\gamma(b_i)-\gamma(a_i)\big)
=\left(\frac{\tilde{z}-\tilde{x}}{2},t\right). 
\end{equation}
See Figure \ref{Fig:1}. 
\begin{figure}[htb]
\centering
\includegraphics[width=12.5cm]{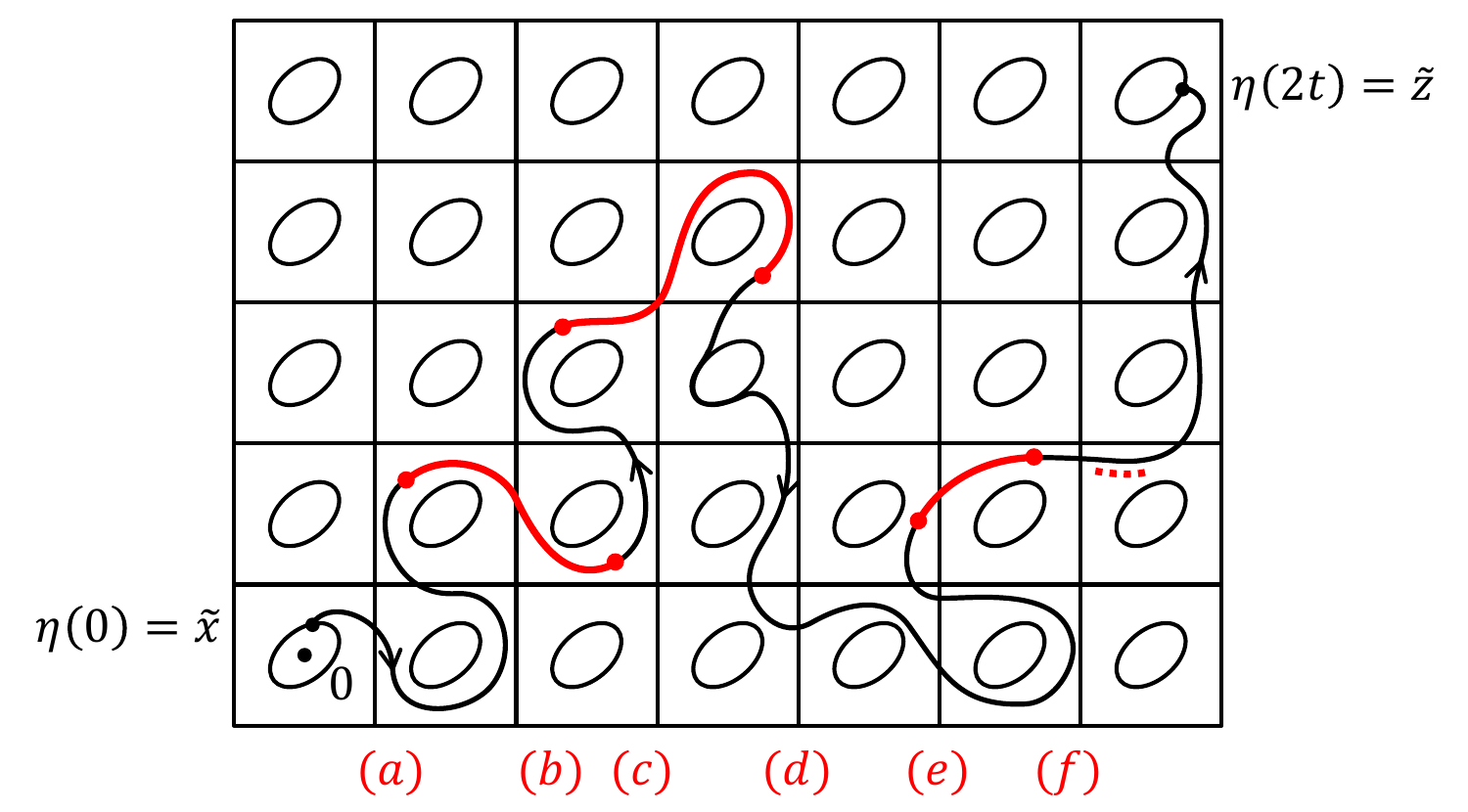}
\caption{Image Picture of $\eta_1$:  
(a) $\eta_1(a_1)$, 
(b) $\eta_1(b_1)$, 
(c) $\eta_1(a_2)$, 
(d) $\eta_1(b_2)$, 
(e) $\eta_1(a_3)$, 
(f) $\eta_1(b_3)$. 
} \label{Fig:1}
\end{figure}

Set 
$t_0=0$, $t_j:=\sum\limits_{i=1}^j(b_i-a_i)$ for $1\le j\le k$. 
It is clear to see that $t_k=t$. 
Then, consider a periodic shift of $\eta_1|_{[a_j,b_j]}$ for $1\le j\le k$ as follows: set 
\[
\left\{
   \begin{aligned}
   &\eta_2(0^-):=\hat{x}\in\partial\Omega\cap Y,  
   &  &  \\
   &\eta_2(s):=\eta_1(a_i+s)+\bm{d}_i
   & \text{for} & \ s\in(t_{i-1},t_i], \ i=1,\ldots, k 
    \end{aligned}
\right.
\]
for some $\bm{d}_i\in\Z^n$. 
Note that $\eta_2$ is a discontinuous trajectory on $\overline{\Omega}$ 
in general.  
Take $\{\bm{d}_i\}_{1\le i\le k}\subset\Z^n$ so that 
\begin{equation}\label{shift:periodic}
\eta_2(0^+)-0\in Y,
\quad 
\eta_2(t_i^+)-\eta_2(t_i^-)\in Y 
\quad
\text{for} \ i\in\{1,\ldots, k-1\}.  
\end{equation}
See Figure \ref{Fig:2}. 
\begin{figure}[htb]
\centering
\includegraphics[width=7.5cm]{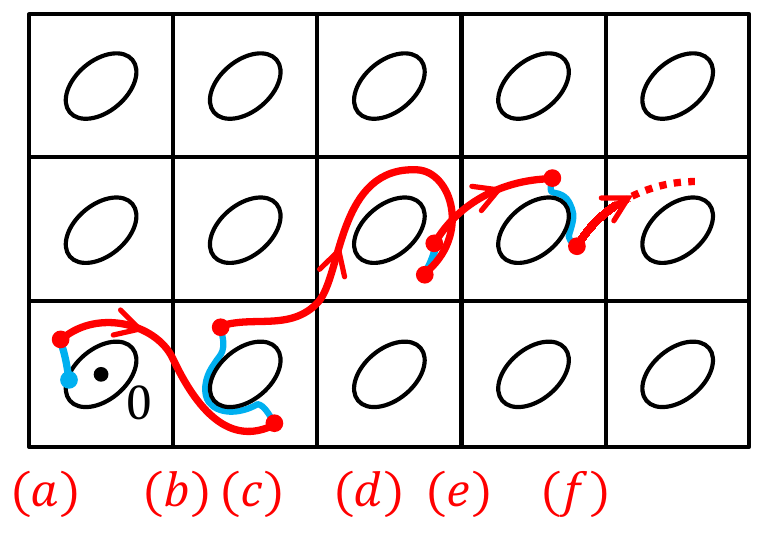}
\caption{Image Picture of $\eta_2$:  
(a) $\eta_1(a_1)+\bm{d}_1=:\eta_2(t_0)$, 
(b) $\eta_1(b_1)+\bm{d}_1=:\eta_2(t_1^-)$, 
(c) $\eta_1(a_2)+\bm{d}_2=:\eta_2(t_1^+)$, 
(d) $\eta_1(b_2)+\bm{d}_2=:\eta_2(t_2^-)$, 
(e) $\eta_1(a_3)+\bm{d}_3=:\eta_2(t_2^+)$, 
(f) $\eta_1(b_3)+\bm{d}_3=:\eta_2(t_2^-)$. 
} 
\label{Fig:2}
\end{figure}
\begin{figure}[htb]
\centering
\includegraphics[width=13cm]{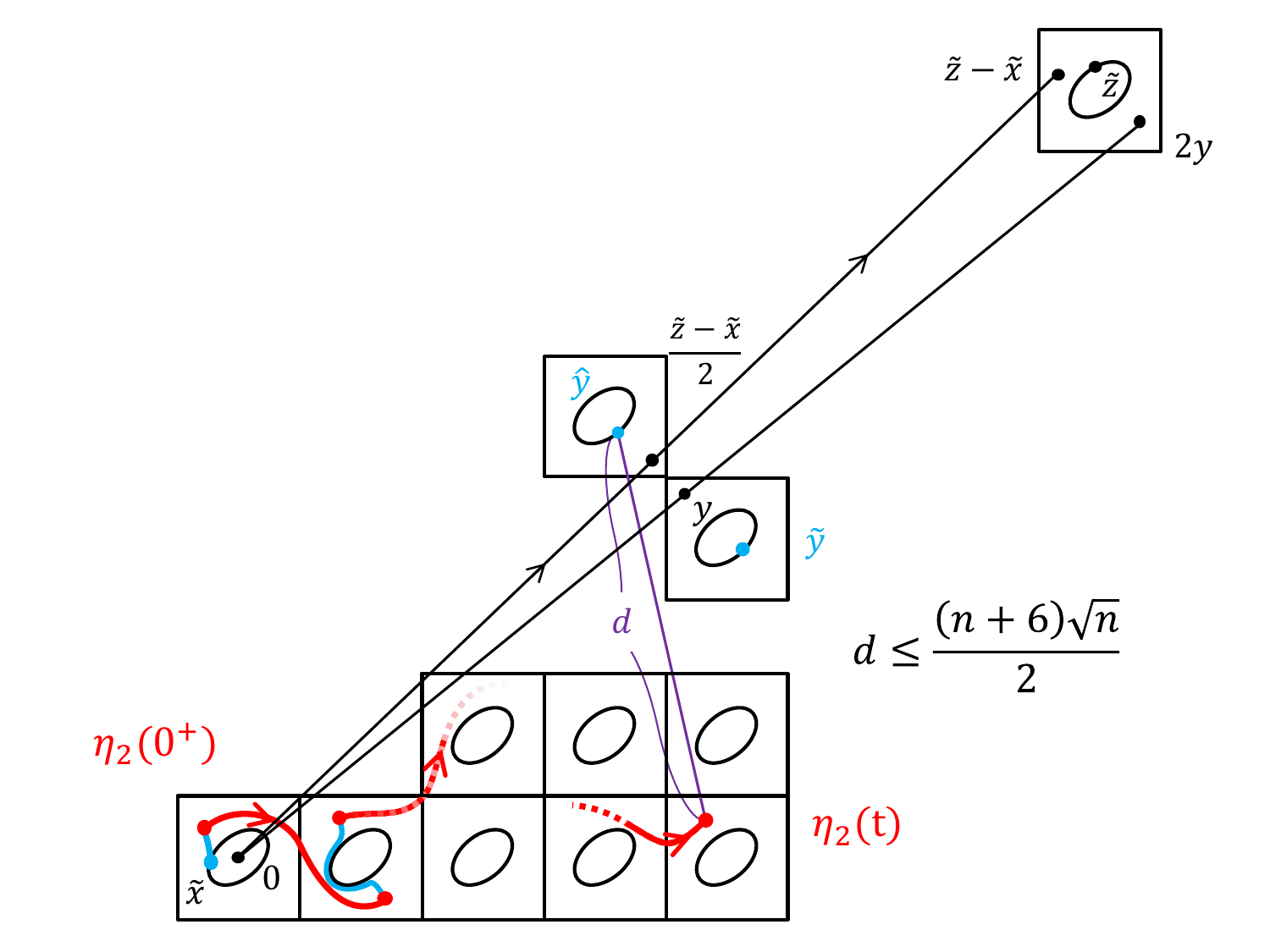}
\caption{Overall image of $\eta_2$. 
} 
\label{Fig:3}
\end{figure}
Note that 
\[
|\eta_2(t^+_j)-\eta_2(t^-_j)|\le \sqrt{n}\quad \text{for\ all}\ j\in\{1,\dots,k-1\}.
\]
Moreover, take $\hat{y}, \tilde{y}\in\partial\Omega$ so that 
\[
\hat y-\frac{\tilde z-\tilde x}{2}\in Y, 
\quad
\tilde y-y\in Y. 
\]
Note that
\[
|\tilde y-\hat y|\le |\tilde y-y|+\left|y-\frac{\tilde z-\tilde x}{2}\right|+\left|\frac{\tilde z-\tilde x}{2}-\hat y\right|\le 3\sqrt{n}.
\]
Also, noting that 
\[
\eta_2(t_k)
=\sum_{i=1}^{k}\eta_1(b_i)-\eta_1(a_i)+\sum_{i=0}^{k-1}\eta_1(t_{i}^{+})-\eta_1(t_{i}^{-})+\eta_2(t_{0}^{-}), 
\]
by using \eqref{eq:Burago}, \eqref{shift:periodic}, we have 
\begin{equation}\label{dist:far}
|\eta_2(t_k)-\hat{y}|\le 
\left|\hat{y}-\frac{\tilde{z}-\tilde{x}}{2}\right|+(k+1)\sqrt{n}
\le (k+2)\sqrt{n}
\le \frac{(n+6)\sqrt{n}}{2}. 
\end{equation}
We define 
\begin{align*}
& 
l_2(s):=l_1(a_i+s) && \text{for} \ s\in(t_{i-1},t_i], \ i\in\{1,\ldots,k\}, \\
& 
v_2(s):=\dot{\eta}_2(s)+l_2(s)\nu(\eta_2(s)) 
&& \text{for} \ s\in[0,t]. 
\end{align*}

Next, we would connect discontinuities of $\eta_2$ to construct an absolutely continuous trajectory. 
However, we would keep the time interval $t$. 
Therefore, we first use change of variables of time to save time. 
To do so, note that there exists $C>0$ such that, by Lemma \ref{lem:mb},  
\[
m_{\C}^*(2t,0,2y)= \int_0^{2t}L_{\C}(\eta_1(s),-v_1(s),- l_1(s))\, ds
\le
Ct, 
\]
and by repeating an argument in the proof of Proposition \ref{prop:tri3},
\begin{equation}\label{Leta}
-\frac{K_0}{4}\le 
\int_d^{d+\frac{1}{4}}L_{\C}(\eta_2(s),-v_2(s),-l_2(s))ds\le 
C
\end{equation}
for some 
$d\in\{0,\frac{1}{4},\frac{1}{2},\frac{3}{4},\dots,\lfloor t\rfloor-\frac{1}{4}\}$. 
Set 
\begin{align*}
(\eta_3(s), l_3(s))&:=
\begin{cases}
(\eta_2(s),l_2(s))& \text{for} \ s\in[0,d], 
\\ 
(\eta_2(2(s-d)+d),2l_2(2(s-d)+d))&
\text{for} \ s\in[d, d+\frac{1}{8}],
\\ 
(\eta_2(s+\frac{1}{8}),l_2(s+\frac{1}{8}))&
\text{for} \ s\in[d+\frac{1}{8}, t-\frac{1}{8}], 
\end{cases}\\
v_3(s)&:=\dot{\eta}_3(s)+l_3(s)\nu(\eta_3(s)) 
\qquad 
\text{for} \ s\in\left[0,t-\frac{1}{8}\right]. 
\end{align*}

Then, 
by \eqref{Leta} and arguing as in \eqref{intd}, we get
\begin{align*}
&
\left|\int_0^t 
L_{\C}(\eta_2(s),-v_2(s),-l_2(s))\,ds
-\int_0^{t-\frac{1}{8}}L_{\C}(\eta_3(s),-v_3(s),-l_3(s))\,ds\right|\\
\\ 
=&\, 
\left|\int_d^{d+\frac{1}{4}} 
L_{\C}(\eta_2(s),-v_2(s),-l_2(s))\,ds
-\int_d^{d+\frac{1}{8}}L_{\C}(\eta_3(s),-v_3(s),-l_3(s))\,ds\right|\\
\le&\, 
\left|\int_d^{d+\frac{1}{4}}L_{\C}(\eta_2(s),-v_2(s),-l_2(s))\, ds\right|
\\&\quad 
+\left|\int_d^{d+\frac{1}{8}}
L_{\C}(\eta_2(2(s-d)+d),-2v_2(2(s-d)+d),-2l_2(2(s-d)+d))\,ds\right|
\\
\le&\, 
C
\end{align*}
for some $C>0$.

Next, 
by using Lemma \ref{HJMT:prop-2.2}, 
we construct paths 
$\{\xi_i\}_{i=1}^{k+2}\in\AC([0,\tau_k],\overline{\Omega})$ to connect discontinuity points 
``$\hat{x}$ and $\eta_2(0^+)$", 
``$\eta_2(t_i^-)$ and $\eta_2(t_i^+)$" for $i=1,\ldots, k-1$, 
``$\eta_2(t_k)$ and $\hat{y}$", and 
``$\hat{y}$ and $\tilde{y}$", where we set $\tau_k:=\frac{1}{8(k+2)}$.
Note that the most distance of two points above is at most 
$\frac{(n+6)\sqrt{n}}{2}$ by \eqref{dist:far}. 
By using Lemma \ref{HJMT:prop-2.2}, we have an estimate of velocity 
\[
\|\dot{\xi}_i\|_{\Li([0,\tau_k])}\le 
M_\Omega\left(
\frac{(n+6)\sqrt{n}}{2}\cdot 8(k+2)
+2\sqrt{n}\cdot 8(k+2)
\right),\quad \text{for\ all}\ i=1,\dots,k+2.
\]

Now, gluing $\{\xi_i\}_{i=1}^{k+2}$ to $\eta_3$, we obtain 
\[
\eta_4\in\AC([0,t],\overline{\Omega}) \ \text{with} \ 
\eta(0), \eta(t)\in\partial\Omega \ 
\text{satisfying} \ 
\eta(0)-0\in Y,\ \eta(t)-y\in Y. 
\]
More precisely, we first calculate the jumping points $\{\ol t_i\}_{0\le i\le k-1}$ of $\eta_3$. By the definition of $\eta_3$, for $i\in\{0,\dots,k-1\}$ with $t_i\le d$, we have $\ol t_i=t_i$. For $i\in\{0,\dots,k-1\}$ with $d\le t_i\le d+\frac{1}{4}$, we have $\ol t_i=\frac{t_i+d}{2}$. For $i\in\{0,\dots,k-1\}$ with $t_i\ge d+\frac{1}{4}$, we have $\ol t_i=t_i-\frac{1}{8}$. For $0\le j\le k-1$, we set
\begin{align*}
& 
\eta_4(s):=\xi_{j+1}\left(s-\ol t_j-j\tau_k\right) && 
\text{for} \ s\in\left[\ol t_j+j\tau_k, \ol t_j+(j+1)\tau_k\right], \\
& 
\eta_4(s):=
\eta_3(s-(j+1)\tau_k)
&& 
\text{for} \ s\in\left[\ol t_j+(j+1)\tau_k, \ol t_{j+1}+(j+1)\tau_k\right].  
\end{align*}
Finally, we set 
\[
\eta_4(s):=
\xi_{k+1}(s-\ol t_k-k\tau_k)
\quad\text{for} \ s\in\left[\ol t_k+k\tau_k, \ol t_k+(k+1)\tau_k\right]. 
\]
and
\[
\eta_4(s):=
\xi_{k+2}(s-\ol t_k-(k+1)\tau_k)
\quad\text{for} \ s\in\left[\ol t_k+(k+1)\tau_k, \ol t_k+(k+2)\tau_k\right],
\]
where we note that $\ol t_k+(k+2)\tau_k=t$. 
Also, for $0\le j\le k-1$ we set 
\begin{align*}
& 
l_4(s):=
0
&& 
\text{for} \ s\in\left[\ol t_j+j\tau_k, \ol t_j+(j+1)\tau_k\right],  
\\
& 
l_4(s):=l_3\left(s-(j+1)\tau_k\right) && 
\text{for} \ s\in\left[\ol t_j+(j+1)\tau_k, \ol t_{j+1}+(j+1)\tau_k\right].
\end{align*}
We also set
\[l_4(s):=0\quad 
\text{for} \ s\in\left[\ol t_k+k\tau_k, \ol t_k+(k+1)\tau_k\right]\cup\left[\ol t_k+(k+1)\tau_k, \ol t_k+(k+2)\tau_k\right].\] 
We set 
\[
v_4(s):=\dot{\eta}_4(s)+l_4(s)\nu(\eta_4(s))\quad
\text{for} \ s\in[0,t]. 
\]
Then, it is easy to check that $(\eta_4,l_4,v_4)\in\SP(0)$ with 
$\eta_4(t)=\tilde y$. 
Therefore, by the definition of $m^\ast_{\C}$, 
\begin{align}
m^*_{\C}(t,0,y)&\le \int_0^t
L_{\C}(\eta_4(s),-v_4(s),-l_4(s))\,ds
\nonumber
\\ &\le 
\int_0^t
L_{\C}(\eta_2(s),-v_2(s),-l_2(s))\,ds+C
\nonumber\\ &
= 
\sum_{i=1}^k\int_{a_i}^{b_i}
L_{\C}(\eta_1(s),-v_1(s),-l_1(s))\,ds+C. 
\label{ineq:sup1}
\end{align}

For the other interval cut by Burago's cutting lemma, 
we set 
\[
[0,2t]\setminus\bigcup_{i=1}^k[a_i,b_i]
=:
\bigcup_{i=1}^{\tilde{k}}[e_i,f_i]
\]
for some $\tilde{k}\in\N$ with $\tilde{k}\le k+1$. 
Repeating the same argument as above, we obtain 
\begin{equation}\label{ineq:sup2}
m^*_{\C}(t,0,y)
\le 
\sum_{i=1}^{\tilde{k}}\int_{e_i}^{f_i}
L_{\C}(\eta_1(s),-v_1(s),-l_1(s))\,ds+C. 
\end{equation}
Combining \eqref{ineq:sup1} with \eqref{ineq:sup2}, we obtain 
\begin{align*}
2m^*_{\C}(t,0,y)
\le&\, 
\sum_{i=1}^k\int_{a_i}^{b_i}
L_{\C}(\eta_1(s),-v_1(s),-l_1(s))\,ds+C\\
&\, 
+
\sum_{i=1}^{\tilde{k}}\int_{e_i}^{f_i}
L_{\C}(\eta_1(s),-v_1(s),-l_1(s))\,ds+C \\
=&\, 
\int_{0}^{2t}
L_{\C}(\eta_1(s),-v_1(s),-l_1(s))\,ds+2C
=
m^\ast_{C}(2t,0,2y)+2C,  
\end{align*}
which finishes the proof. 
\end{proof}

\section{Proof of Theorem \ref{thm:main}}\label{sec:pthm}

By using the subadditivity and the superadditivity of extended metric functions, we prove Theorem \ref{thm:main} following the argument in \cite[Section 4]{HJMT}. We provide the proof here for completeness. In this section, we denote by $m^*$ either $m^*_{\NN}$ or $m^*_{\C}$. 
We also denote by $M_0$ the constant given by Proposition {\rm\ref{prop:extremal-main}}. 

The following lemma is a straight forward result of the subadditivity of $m^*$, Lemma \ref{lem:subad}, and the Fekete lemma.
\begin{lem} 
For any $t>0$ and $x,y\in\R^n$ with $|x-y|\le M_0 t$, the following limit exists
\[\ol m^*(t,x,y):=\lim_{k\to+\infty}\frac{1}{k}m^*(kt,kx,ky).\]
\end{lem}

\begin{lem}\label{lem:mcon}
Let $\ep>0$, $t\ge \ep$ and $x,y\in \R^n$ with $|x-y|\le M_0t$. Then, there exists a constant $C>0$ such that
\[\bigg|\ol m^*(t,x,y)-\ep m^*\Big(\frac{t}{\ep},\frac{x}{\ep},\frac{y}{\ep}\Big)\bigg|\le C\ep.\]
\end{lem}
\begin{proof}
Without loss of generality, we can consider the case where $x=0$. We only prove the direction
\[\ep m^*\Big(\frac{t}{\ep},0,\frac{y}{\ep}\Big)-\ol m^*(t,x,y)\le C\ep.\]
The proof of the other direction is similar.

Let $\tilde t:=\frac{t}{\ep}$ and $\tilde y:=\frac{y}{\ep}$. Then $\tilde t\ge 1$ and $|\tilde y|\le M_0\tilde t$. By Lemma \ref{prop:super-ad} we have
\[2(m^*(\tilde t,0,\tilde y)-C)\le m^*(2\tilde t,0,2\tilde y)-C.\]
For $k\in\mathbb N$, iterating the above for $k$ times, we get
\[2^k(m^*(\tilde t,0,\tilde y)-C)\le m^*(2^k\tilde t,0,2^k\tilde y)-C.\]
Dividing both sides by $2^k$ and sending $k\to+\infty$ we obtain
\[m^*(\tilde t,0,\tilde y)-C\le \ol m^*(\tilde t,0,\tilde y),\]
which implies
\[\ep m^*\Big(\frac{t}{\ep},0,\frac{y}{\ep}\Big)-C\ep\le \ep\ol m^*\Big(\frac{t}{\ep},0,\frac{y}{\ep}\Big)=\ol m^*(t,0,y),\]
where for the last equality, we used the homogeneity of $\ol m^*$.
\end{proof}

Let $u^\ep$ be the solution of \eqref{eq:CN1}-\eqref{eq:CN3}. 
According to Theorem \ref{thm:value} and Proposition \ref{prop:extremal-main}, it is direct to check that
\[u^\ep(x,t)=\inf\bigg\{\ep m\Big(\frac{t}{\ep},\frac{x}{\ep},\frac{y}{\ep}\Big)+u_0(y)\mid |x-y|\le M_0t,\ y\in\ol \Omega_\ep\bigg\}.\]
We now know that $u^\ep$ converges to the solution of \eqref{eq:limit}, which is represented by
\[u(x,t)=\inf\bigg\{t\ol{L}\Big(\frac{x-y}{t}\Big)+u_0(y)\mid y\in\R^n\bigg\}.\]
Now we need to show that
\[\ol u(x,t):=\inf\bigg\{\ol m^*(t,x,y)+u_0(y)\mid |x-y|\le M_0t,\ y\in\R^n\bigg\}\]
conincides with $u(x,t)$.

\begin{lem}
We have $u(x,t)=\bar u(x,t)$ for all $(x,t)\in\R^n\times[0,\infty)$.
\end{lem}
\begin{proof}
Let $\delta>0$. Since $u^\ep$ uniformly converges to $u$ on $\ol\Omega_\ep$, we have
\[|u^\ep(x,t)-u(x,t)|\le \delta\quad \text{for\ all}\ x\in \ol \Omega_\ep\]
for $\ep$ small enough. 
By Lemma \ref{lem:mcon}, we have
\[\bigg|\ol m^*(t,x,y)-\ep m^*\Big(\frac{t}{\ep},\frac{x}{\ep},\frac{y}{\ep}\Big)\bigg|\le C\ep\]
for any $y\in\R^n$ with $|x-y|\le M_0t$. 
For any $\ep>0$ and $t\ge \ep$, there exists $y_{\ep,t,x}\in \ol{\Omega}_\ep$ with $|y_{\ep,t,x}-x|\le M_0t$ such that
\[u^\ep(x,t)=\ep m\Big(\frac{t}{\ep},\frac{x}{\ep},\frac{y_{\ep,t,x}}{\ep}\Big)+u_0(y_{\ep,t,x}).\] 
By Proposition \ref{prop:difference-mm},
\[\bigg|u^\ep(x,t)-\ep m^*\Big(\frac{t}{\ep},\frac{x}{\ep},\frac{y_{\ep,t,x}}{\ep}\Big)-u_0(y_{\ep,t,x})\bigg|\le C\ep.\] 
Thus,
\begin{align*}
\ol u(x,t)-u(x,t)&\le \ol m^*(t,x,y_{\ep,t,x})+u_0(y_{\ep,t,x})-u(x,t)
\\ &\le \ep m^*\Big(\frac{t}{\ep},\frac{x}{\ep},\frac{y_{\ep,t,x}}{\ep}\Big)+C\ep+u_0(y_{\ep,t,x})-u(x,t)
\\ &\le u^\ep(x,t)+2C\ep-u(x,t)
\\ &\le 2C\ep+\delta,\quad \text{for\ all}\ x\in \ol\Omega_\ep.
\end{align*}
Now, note that the left hand side is independent of $\ep$. For each $x\in \R^n$, we choose suitable $\ep_n\to 0$ such that $x\in \ol\Omega_{\ep_n}$, and then let $\delta\to 0$, we get $\ol u\le u$.

On the other hand, for $\delta>0$, there is $\ol y_{t,x,\delta}\in \R^n$ with $|\ol y_{t,x,\delta}-x|\le M_0t$ such that
\[\ol u(x,t)+\delta\ge \ol m^*(t,x,\ol y_{t,x,\delta})+u_0(\ol y_{t,x,\delta}).\]
We conclude that
\begin{align*}
u(x,t)&\le u^\ep(x,t)+\delta
\\ &\le \ep m^*\Big(\frac{t}{\ep},\frac{x}{\ep},\frac{\ol y_{t,x,\delta}}{\ep}\Big)+C\ep+u_0(\ol y_{t,x,\delta})+\delta
\\ &\le  \ol m^*(t,x,\ol y_{t,x,\delta})+u_0(\ol y_{t,x,\delta})+2C\ep+\delta
\\ &\le \ol u(x,t)+2C\ep+2\delta,\quad \text{for\ all}\ x\in \ol\Omega_\ep.
\end{align*}
We choose suitable $\ep_n\to 0$ such that $x\in \ol\Omega_{\ep_n}$, and then let $\delta\to 0$, we get $u\le \ol u$.
\end{proof}

\noindent {\it Proof of Theorem \ref{thm:main}. }
We first consider $0<t \le \ep$. By the comparison principle, Lemma \ref{lem:value-2}, we have
\[|u^\ep(x,t)-u_0(x)|\le Ct\quad \text{and}\quad |u(x,t)-u_0(x)|\le Ct.\]
We obtain that there is $C>0$ so that
\[u^\ep(x,t)-u(x,t)\le Ct\le C\ep.\]

We now consider $t>\ep$. It suffices to prove
\[|u^\ep(x,t)-\ol u(x,t)|\le C\ep.\]
On one hand, by Lemma \ref{lem:mcon} we have
\begin{align*}
\ol u(x,t)&=\inf\bigg\{\ol m^*(t,x,y)+u_0(y)\mid |x-y|\le M_0t,\ y\in\R^n\bigg\}
\\ &\le \inf\bigg\{\ep m\Big(\frac{t}{\ep},\frac{x}{\ep},\frac{y}{\ep}\Big)+u_0(y)+C\ep\mid |x-y|\le M_0t,\ y\in\ol \Omega_\ep\bigg\}
\\ &\le u^\ep(x,t)+C\ep.
\end{align*}

On the other hand, by Proposition \ref{prop:difference-mm}, we have
\begin{align*}
u^\ep(x,t)&=\inf\bigg\{\ep m\Big(\frac{t}{\ep},\frac{x}{\ep},\frac{y}{\ep}\Big)+u_0(y)\mid |x-y|\le (M_0+\sqrt{n}\ep)t,\ y\in\ol \Omega_\ep\bigg\}
\\ &\le \inf\bigg\{\ep m^*\Big(\frac{t}{\ep},\frac{x}{\ep},\frac{y}{\ep}\Big)+u_0(y)+C\ep\mid |x-y|\le (M_0+\sqrt{n}\ep)t,\ y\in\ol \Omega_\ep\bigg\}.
\end{align*}
We claim that
\begin{align*}
&\inf\bigg\{\ep m^*\Big(\frac{t}{\ep},\frac{x}{\ep},\frac{y}{\ep}\Big)+u_0(y)+C\ep\mid |x-y|\le (M_0+\sqrt{n}\ep)t,\ y\in\ol \Omega_\ep\bigg\}
\\ &\le \inf\bigg\{\ep m^*\Big(\frac{t}{\ep},\frac{x}{\ep},\frac{y}{\ep}\Big)+\|Du_0\|_{L^\infty(\R^n)}\sqrt{n}\ep+u_0(y)+C\ep\mid |x-y|\le M_0t,\ y\in\R^n\bigg\}.
\end{align*}
We only need to take care of the case where the second infimum above is achieved at $y\notin \ol \Omega_\ep$. There exist $\frac{\tilde y}{\ep},\frac{\tilde x}{\ep}\in \partial\Omega$ with $\frac{\tilde y}{\ep}-\frac{y}{\ep}\in Y$ and $\frac{\tilde x}{\ep}-\frac{x}{\ep}\in Y$ such that
\[m^*\Big(\frac{t}{\ep},\frac{x}{\ep},\frac{y}{\ep}\Big)=m\Big(\frac{t}{\ep},\frac{\tilde x}{\ep},\frac{\tilde y}{\ep}\Big).\]
Also, by the definition of $m^*$ we have
\[m\Big(\frac{t}{\ep},\frac{\tilde x}{\ep},\frac{\tilde y}{\ep}\Big)\ge m^*\Big(\frac{t}{\ep},\frac{x}{\ep},\frac{\tilde y}{\ep}\Big)\]
Using the above two inequalities, we get
\[\ep m^*\Big(\frac{t}{\ep},\frac{x}{\ep},\frac{y}{\ep}\Big)+u_0(y)+\|Du_0\|_{L^\infty(\R^n)}\sqrt{n}\ep\ge \ep m^*\Big(\frac{t}{\ep},\frac{x}{\ep},\frac{\tilde y}{\ep}\Big)+u_0(\tilde y),\]
which confirms our claim. We then conclude that there is $C>0$ such that
\begin{align*}
u^\ep(x,t)&\le \inf\bigg\{\ep m^*\Big(\frac{t}{\ep},\frac{x}{\ep},\frac{y}{\ep}\Big)+u_0(y)+C\ep\mid |x-y|\le (M_0+\sqrt{n}\ep)t,\ y\in\ol \Omega_\ep\bigg\}
\\ &\le \inf\bigg\{\ep m^*\Big(\frac{t}{\ep},\frac{x}{\ep},\frac{y}{\ep}\Big)+\|Du_0\|_{L^\infty(\R^n)}\sqrt{n}\ep+u_0(y)+C\ep\mid |x-y|\le M_0t,\ y\in\R^n\bigg\}
\\ &\le \inf\big\{\ol m^*(t,x,y)+C\ep+u_0(y)\mid |x-y|\le M_0t,\ y\in\R^n\bigg\}
\\ &\le \ol u(x,t)+C\ep.
\end{align*}
The proof is now complete.
\qed

\section*{Acknowledgements}

%The authors thank Yuxi Han, Jiwoong Jang, Hung V. Tran for their valuable comments to the first draft of the paper.
The work of HM was partially supported by the JSPS grants: KAKENHI
\#21H04431, \#22K03382, \#24K00531, \#25K07072.
%\#22K03382, \#21H04431, \#20H01816, \#19H00639.
%{\color{red}The work of PN was ...}
%This work was done during a visit of PN in the university of Tokyo. He acknowledges the hospitality of the university.

\section*{Declarations}

\noindent {\bf Conflict of interest statement:} The authors state that there is no conflict of interest.

\medskip

\noindent {\bf Data availability statement:} Data sharing not applicable to this article as no datasets were generated or analysed during the current study.

%\end{comment}


\begin{thebibliography}{}

%\bibitem{CGMT}
%F. Cagnetti, D. Gomes, H. Mitake, H. V. Tran,
%\emph{A new method for large time behavior of degenerate viscous
%Hamilton-Jacobi equations with convex Hamiltonians},
%Ann. Inst. H. Poincar\'e Anal. Non Lin\'eaire
%32(1), 183--200 (2015).

\bibitem{A} 
O. Alvarez, 
\emph{Homogenization of Hamilton-Jacobi equations in perforated sets}, 
J. Differential Equations, 159 (1999), no. 2, 543 -- 577.

\bibitem{AI} 
O. Alvarez, H. Ishii, 
\emph{Hamilton-Jacobi equations with partial gradient and application to homegenization}, 
Commun. Partial Differ. Equ., 26 (1999), no. 5-6, 983 -- 1002.

\bibitem{BL} 
G. Barles, P.-L. Lions, 
\emph{Fully nonlinear neumann type boundary conditions for first-order Hamilton-Jacobi equations}, 
Nonlinear Analysis, Theory, Methods \& Applications, 16 (1991), 143 -- 153.

\bibitem{BIM} 
G. Barles, H. Ishii, H. Mitake, 
\emph{On the large time behavior of solutions of Hamilton-Jacobi equations associated with nonlinear boundary conditions}, 
Arch. Rational Mech. Anal., 204 (2012), 515 -- 558.

\bibitem{Burago}
D. Burago,
\emph{Periodic metrics}, 
Adv. Soviet Math. 9 (1992), 205--210.

\bibitem{BGH}
G. Buttazzo, M. Giaquinta, S. Hildebrandt,
\emph{One-dimensional variational problems: an introduction},
Oxford Lecture Series in Mathematics and Its Applications, New York: Clarendon, 1998.

\bibitem{CM}
P. Cardaliaguet, C. Marchi, 
\emph{Regularity of the eikonal equation with Neumann boundary conditions in the plane: application to fronts with nonlocal terms},
SIAM J. Control Optim. 45(3), 1017--1038 (2006). 

\bibitem{CDI}
I. Capuzzo-Dolcetta, H. Ishii,
\emph{On the rate of convergence in homogenization of Hamilton--Jacobi equations},
Indiana Univ. Math. J., 50 (2001), 1113--1129.

\bibitem{C}
W. Cooperman,
\emph{A near-optimal rate of periodic homogenization for convex Hamilton-Jacobi equations},
Arch. Rational Mech. Anal., 245 (2022), 809 -- 817.

%\bibitem{DZ2}
%A. Davini, M. Zavidovique,
%\emph{Aubry sets for weakly coupled systems of Hamilton--Jacobi equations},
%SIAM J. Math. Anal., 46 (2014), 3361 -- 3389.

%\bibitem{Evans2}
%L. C. Evans,
%\emph{The perturbed test function method for viscosity solutions of nonlinear PDE},
%Proc. Roy. Soc. Edinburgh Sect. A 111 (1989), no. 3-4, 359--375.

%\bibitem{Evans}
%L. C. Evans,
%\emph{Periodic homogenisation of certain fully nonlinear partial differential equations},
%Proceedings of the Royal Society of Edinburgh, 120 A (1992), 245 -- 265.

\bibitem{G} Y. Giga, \emph{Surface Evolution Equations--A Level Set Approach}, Monogr. Math., 99, Birkh\"auser, 2006.

\bibitem{HJ}
Y. Han, J. Jang,
\emph{Rate of convergence in periodic homogenization for convex Hamilton--Jacobi equations with multiscales},
Nonlinearity, 36 (2023), 5279 -- 5297.

\bibitem{HJMT}
Y. Han, W. Jing, H. Mitake, H. V. Tran,
\emph{Quantitative homogenization of state-constraint Hamilton-Jacobi equations on perforated domains and applications},
Arch. Rational Mech. Anal.,
249 (2025), 249:18.

\bibitem{HI} K. Horie, H. Ishii, 
\emph{Homogenization of Hamilton-Jacobi equations on domains with small scale periodic structure}, 
Indiana Univ. Math. J. 47 (1998), 1011 -- 1058. 

\bibitem{HTZ}
B. Hu, S. Tu, J. Zhang,
\emph{Polynomial convergence rate for quasi-periodic homogenization of Hamilton-Jacobi equations and application to ergodic estimates},
Commun. Partial Differential Equations, 50 (2025), 211 -- 244.

\bibitem{I11} H. Ishii, 
\emph{Weak KAM aspects of convex Hamilton-Jacobi equations with Neumann type boundary conditions}, 
J. Math. Pures Appl., 95 (2011), 99 -- 135.

\bibitem{I13} H. Ishii, 
\emph{A short introduction to viscosity solutions and the large time behavior of solutions of Hamilton-Jacobi equations}. 
In: Hamilton-Jacobi Equations: Approximations, Numerical Analysis and Applications. Lecture Notes in Mathematics, Springer, Berlin, Heidelberg, 2013.

%\bibitem{LPV}
%P.-L. Lions, G. Papanicolaou and S. Varadhan,
%\emph{Homogenization of Hamilton--Jacobi equation}.
%unpublished preprint (1987).


%\bibitem{J1}
%W. Jing, 
%\emph{A unified homogenization approach for the Dirichlet problem in perforated domains},  
%SIAM J. Math. Anal. 52 (2020), no. 2, 1192--1220. 

\bibitem{J2}
W. Jing, 
\emph{Convergence rate for the homogenization of diffusions in dilutely perforated domains with reflecting boundaries},  
Minimax Theory Appl. 8 (2023), no. 1, 85--108. 


\bibitem{LS}
P.-L. Lions, A.-S. Sznitman, 
\emph{Stochastic differential equations with reflecting boundary conditions}, 
Comm. Pure Appl. Math. 37 (1984), no. 4, 511--537. 

\bibitem{L}
F. Lin, 
\emph{On current developments in partial differential equations}, 
Commun. Math. Res., 36 (2020), no. 1, 1--30.

\bibitem{KLS}
C. Kenig, F. Lin, Z. Shen,
\emph{Homogenization of elliptic systems with Neumann boundary conditions},
J. Amer. Math. Soc., 26 (2013), 901--937.

\bibitem{MMTXY}
H. Mitake, C. Mooney, H. V. Tran, J. Xin, Y. Yu,
\emph{Bifurcation of homogenization and nonhomogenization of the curvature G-equation with shear flows},
Math. Ann. 391 (2025), 3077--3111.

\bibitem{MN}
H. Mitake, P. Ni, 
\emph{Rate of convergence for homogenization of nonlinear weakly coupled Hamilton-Jacobi systems}, 
J. Differential Equations, 440 (2025), 113442.

%\bibitem{MT}
%H. Mitake, H. V. Tran,
%\emph{Homogenization of weakly coupled systems of Hamilton-Jacobi equations with fast switching rates},
%Arch. Rational Mech. Anal., 211 (2014), 733 -- 779.

\bibitem{MTY}
H. Mitake, H. V. Tran, Y. Yu,
\emph{Rate of convergence in periodic homogenization of Hamilton--Jacobi equations: the convex setting},
Arch. Rational Mech. Anal., 233 (2019), 901 -- 934.

\bibitem{MS}
H. Mitake, S. Sato,
\emph{On the rate of convergence in homogenization of time-fractional Hamilton--Jacobi equations},
NoDEA Nonlinear Differential Equations Appl., 30 (2023) 30:68.

\bibitem{MNT}
H. Mitake, P. Ni, H. V. Tran,  
\emph{Quantitative homogenization of convex Hamilton--Jacobi equations with $u/\ep$-periodic Hamiltonians}, 
Preprint is available from arXiv:2507.00663.

\bibitem{M}
Y. Matsukawa,
\emph{Crystallography of Precipitates in Metals and Alloys: (2) Impact of Crystallography on Precipitation Hardening},
Crystallography, Chapter 3, IntechOpen, London, 2019.

\bibitem{NT}
H. Nguyen-Tien,
Optimal convergence rate for homogenization of convex Hamilton-Jacobi equations in the periodic spatial-temporal environment,
Asymptotic Anal., 138 (2024), 135 -- 150.

\bibitem{QSTY}
J. Qian, T. Sprekeler, H. V. Tran, Y. Yu,
\emph{Optimal rate of convergence in periodic homogenization of viscous Hamilton-Jacobi equations},
Multiscale Modeling \& Simulation, 22 (2024), 1558 -- 1584.

%\bibitem{Hung-book}
%H. V. Tran,
%Hamilton-Jacobi equations—theory and applications.
%Graduate Studies in Mathematics, 213. American Mathematical Society.

\bibitem{S}
Z. Shen,
\emph{Homogenization of boundary value problems in perforated Lipschitz domains},
J. Differential Equations, 376 (2023), 283 -- 339.
%arXiv: 2112.06896.

\bibitem{TY}
H. V. Tran, Y. Yu,
\emph{Optimal convergence rate for periodic homogenization of convex Hamilton-Jacobi equations},
to appear in Indiana Univ. Math. J.
%arXiv: 2112.06896.

\bibitem{T}
S. Tu,
\emph{Rate of convergence for periodic homogenization of convex Hamilton-Jacobi equations in one dimension},
Asymptotic Anal., 121 (2021), 171 -- 194.

\bibitem{XY}
J. Xin. Y. Yu,
\emph{Periodic homogenization of the inviscid G-equation for incompressible flows},
Commun. Math. Sci. 8 (2010) 1067--1078.

\end{thebibliography}
\end{document}